\documentclass[reqno,11pt]{article}

\usepackage{a4wide,eucal,enumerate,mathrsfs,xspace}
\usepackage{amsmath,amssymb,epsfig,bbm}

\usepackage[usenames]{color}

\usepackage{ifthen}
\usepackage{tikz}
\usetikzlibrary{patterns}
\usetikzlibrary{decorations.markings}

\tikzset{
  set arrow inside/.code={\pgfqkeys{/tikz/arrow inside}{#1}},
  set arrow inside={end/.initial=>, opt/.initial=},
  /pgf/decoration/Mark/.style={
    mark/.expanded=at position #1 with
    {
      \noexpand\arrow[\pgfkeysvalueof{/tikz/arrow inside/opt}]{\pgfkeysvalueof{/tikz/arrow inside/end}}
    }
  },
  arrow inside/.style 2 args={
    set arrow inside={#1},
    postaction={
      decorate,decoration={
        markings,Mark/.list={#2}
      }
    }
  },
}

\numberwithin{equation}{section}

\newtheorem{theorem}{Theorem}[section]

\newtheorem{corollary}[theorem]{Corollary}
\newtheorem{lemma}[theorem]{Lemma}
\newtheorem{proposition}[theorem]{Proposition}
\newtheorem{definition}[theorem]{Definition}

\newtheorem{example}[theorem]{Example}

\newtheorem{remark}[theorem]{Remark}

\newtheorem{innerass}{Assumption}
\newenvironment{mainassumption}[1]
  {\renewcommand\theinnerass{#1}\innerass}
  {\endinnerass}

\usepackage{pdfsync}
\usepackage{hyperref}

\newcommand{\N}{\mathbb{N}}

\newcommand{\R}{\mathbb{R}}

\newcommand{\Z}{\mathbb{Z}}

\newcommand{\CC}{\mathscr{C}}

\renewcommand{\SS}{\mathscr{S}}

\newcommand{\UU}{\mathscr{U}}

\newcommand{\cE}{{\ensuremath{\mathcal E}}}
\newcommand{\cF}{{\ensuremath{\mathcal F}}}

\newcommand{\cN}{{\ensuremath{\mathcal N}}}

\newcommand{\cP}{{\ensuremath{\mathcal P}}}

\newcommand{\cR}{{\ensuremath{\mathcal R}}}

\newcommand{\cY}{{\ensuremath{\mathcal Y}}}

\newcommand{\sfc}{{\sf c}}
\newcommand{\sfd}{{\sf d}}
\newcommand{\sfe}{{\sf e}}

\newcommand{\sfn}{{\sf n}}

\newcommand{\sfp}{{\sf p}}

\newcommand{\sfr}{{\sf r}}
\newcommand{\sfs}{{\sf s}}
\newcommand{\sft}{{\sf t}}

\newcommand{\sfv}{{\sf v}}

\newcommand{\sfD}{{\sf D}}
\newcommand{\sfE}{{\sf E}}

\newcommand{\sfQ}{{\sf Q}}

\newcommand{\sfW}{{\sf W}}

\newcommand{\sfY}{{\sf Y}}

\newcommand{\frH}{{\frak H}}

\newcommand{\rmd}{{\mathrm d}}
\newcommand{\rme}{{\mathrm e}}

\newcommand{\rmC}{{\mathrm C}}
\newcommand{\rmD}{{\mathrm D}}

\newcommand{\rmM}{{\mathrm M}}

\newcommand{\rmW}{{\mathrm W}}

\newcommand{\Kliminf}{K\kern-3pt-\kern-2pt\mathop{\rm
lim\,inf}\limits}  
\newcommand{\Klimsup}{K\kern-3pt-\kern-2pt\mathop{\rm lim\,sup}\limits}  

\newcommand{\argmin}{\mathop{\rm argmin}\limits}   
\renewcommand{\d}{{\mathrm d}}
\newcommand{\restr}[1]{\lower3pt\hbox{$|_{#1}$}}
\newcommand{\topref}[2]{\stackrel{\eqref{#1}}#2}
\newcommand{\Leb}[1]{{\mathscr L}^{#1}}      
\newcommand{\la}{{\langle}}                  
\newcommand{\ra}{{\rangle}}
\newcommand{\down}{\downarrow}              
\newcommand{\up}{\uparrow}

\newcommand{\eps}{\varepsilon}  
\newcommand{\nchi}{{\raise.3ex\hbox{$\chi$}}}
\newcommand{\media}{\mkern12mu\hbox{\vrule height4pt           %
          depth-3.2pt                                 
          width5pt}\mkern-16.5mu\int\nolimits}        

\def\qed{\ifmmode 
  \else \leavevmode\unskip\penalty9999 \hbox{}\nobreak\hfill
  \fi               
    \qquad           \hbox{\hskip.5em $\square$
                \hskip.1em}}
\def\endproofsym{\qed}

\newenvironment{proof}[1][Proof]{\def\endproofsym{\qed}\trivlist\item[\hskip\labelsep{%
\noindent{\normalfont\emph{#1}.}\hskip .321429\parindent}]\ignorespaces}
{\endproofsym\endtrivlist}

\newcommand{\Holes}{\mathfrak H}
\newcommand{\Pf}{\mathfrak P_f}
\newcommand{\leftl}{-}
\newcommand{\rightl}{+}
\newcommand{\lrl}{\pm}
\newcommand{\uleftl}{^-}
\newcommand{\urightl}{^+}
\newcommand{\ulrl}{^\pm}
\newcommand{\sigmato}{\stackrel\sigma\to}
\newcommand{\Jmp}[1]{\mathrm{Jmp}_{#1}}
\newcommand{\Var}[1]{\mathrm{Var}_{#1}}
\newcommand{\mVar}[1]{\text{\sl{Var}}_{#1}}
\newcommand{\VE}{{\rm VE}}
\newcommand{\BV}[1]{\mathrm{BV}\kern-2pt_{#1}}
\DeclareMathOperator{\var}{Var_\mathsf{d}}
\DeclareMathOperator{\graph}{graph}
\newcommand{\varP}{\mVar{\sfd,\sfe}}
\newcommand{\varC}{\mVar{\sfd,\sfc}}
\DeclareMathOperator{\Ju}{J_u}
\DeclareMathOperator{\Gs}{\cR}
\DeclareMathOperator{\Cf}{Trc} 
\DeclareMathOperator{\Fd}{\mathsf c}
\DeclareMathOperator{\Cd}{\mathrm{GapVar}_\delta}
\DeclareMathOperator{\Li}{Li}
\DeclareMathOperator{\Ls}{Ls}
\DeclareMathOperator{\dom}{D}
\newcommand{\Jump}[2]{{\mathrm J}^{#1}_{#2}}
\newcommand{\topol}{\sigma}
\newcommand{\rmJ}{\mathrm J}
\newcommand{\sfdp}{\sfd_\R}
\newcommand{\sigmap}{\sigma_\R}
\newcommand{\SSd}{\SS\kern-3pt_\sfd}
\newcommand{\SSD}{\SS_\sfD}
\newcommand{\mytag}[2]{\hyperref[#1#2]{$\la$#1\ifthenelse{\equal{#2}{}}{}{.#2}$\ra$}}
\title{Viscous corrections of the\\
  Time Incremental Minimization Scheme and\\
  Visco-Energetic Solutions to \\
  Rate-Independent Evolution Problems}
\begin{document}

\author{Luca Minotti, Giuseppe Savar\'e
\thanks{Universit\`a di Pavia. 
email:
\textsf{luca.minotti01@universitadipavia.it, giuseppe.savare@unipv.it}. 
G.S.~has been partially supported by
PRIN10/11, PRIN15 grants from MIUR for the project \emph{Caculus of Variations}
and by IMATI-CNR}
}

\maketitle

\begin{abstract} 
  We propose the new notion of Visco-Energetic solutions to 
  rate-independent systems $(X,\cE,\sfd)$
  driven by a time dependent energy $\cE$ and a dissipation
  quasi-distance $\sfd$ 
  in a general metric-topological space $X$.

  As for the classic Energetic approach, 
  solutions can be obtained by solving a modified
  time Incremental Minimization Scheme,
  where at each step the dissipation quasi-distance $\sfd$ 
  is incremented by a viscous correction $\delta$ (e.g.~proportional
  to the square of the  distance $\sfd$), 
  which penalizes far distance jumps by
  inducing a localized version of the stability condition. 
  
  We prove a general convergence result and a typical characterization
  by Stability and Energy Balance 
  in a setting comparable to the standard energetic one, thus
  capable to cover a wide range of applications. 
  The new refined Energy Balance condition compensates the localized
  stability and provides a careful description of the jump behavior:
  at every jump the solution follows an optimal transition,
  which resembles in a suitable variational sense the discrete scheme
  that has been implemented for the whole construction.
\end{abstract}


{\small\tableofcontents}
\section{Introduction}
\addtocontents{toc}{\vspace{\normalbaselineskip}}
Since the pioneering papers 
\cite{Mielke-Theil-Levitas02,Mielke-Theil04},
energetic solutions 
(also called irreversible quasi-static evolutions in the fracture
models studied in \cite{Francfort-Marigo98,DalMaso-Francfort-Toader05,DalMaso-Toader02})
to rate-independent evolutionary systems driven
by time-dependent functionals have played a crucial role 
and provided a unifying framework for many different applied models,
such as 
shape memory alloys~\cite{Mielke-Theil-Levitas02,AuMiSt08RIMI}, 
crack propagation~\cite{DalMaso-Toader02,DalMaso-Francfort-Toader05} 
elastoplasticity~\cite{Miel04EMIE,Francfort-Mielke06,DaDeMo06QEPL,DaDeMoMo06,MaiMie08?GERI}, 
damage in brittle materials \cite{MieRou06RIDP,
  BoMiRo07?CDPS,Mielke11} 
delamination~\cite{KoMiRo06RIAD},
ferroelectricity~\cite{MieTim06EMMT}, and
superconductivity~\cite{SchMie05VPSC}. 
We refer to the recent monograph \cite{Mielke-Roubicek15} for a
complete discussion and overview of the theory and its applications.

In its simplest \emph{metric} formulation, 
a Rate-Independent System (R.I.S.) $(X,\cE,\sfd)$ can be described by
a metric space $(X,\sfd)$ and a time-dependent energy functional
$\cE:[0,T]\times X\to \R$. Energetic solutions can be obtained as 
a limit of piecewise constant interpolant of 
discrete solutions $U^n_\tau$ obtained by recursively solving the time
Incremental Minimization scheme
\begin{equation}
  \label{eq:167}
  \min_{U\in X} \cE(t^n_\tau,U)+\sfd(U^{n-1}_\tau,U).
  \tag{IM$_\sfd$}
\end{equation}
The main aim of the present paper is to study general \emph{viscous
corrections} of \eqref{eq:167}
\begin{equation}
  \label{eq:167delta}
  \min_{U\in X} \cE(t^n_\tau,U)+\sfd(U^{n-1}_\tau,U)+\delta(U^{n-1}_\tau,U),
  \tag{IM$_{\sfd,\delta}$}
\end{equation}
obtained by perturbing the distance $\sfd$ by a ``viscous''
penalization term $\delta:X\times X\to[0,\infty)$,
which should induce a better localization 
of the minimizers.  A typical choice is
the quadratic correction $\delta(u,v):=\frac{\mu}2\sfd^2(u,v)$,
for some $\mu>0$.

We will show that solutions generated by the scheme \eqref{eq:167delta} 
exhibit a sort of intermediate behaviour between Energetic and 
Balanced Viscosity solutions \cite{Mielke-Rossi-Savare12},
since they retain the great structural robustness of the former and
allow for a more localized response typical of the latter.
Before explaining these novel features,
let us briefly recall a few basic facts concerning Energetic and 
Balanced Viscosity solutions.
\paragraph{\em Energetic solutions.}
Energetic solutions to the R.I.S.~$(X,\cE,\sfd)$ 
are curves $u:[0,T]\to X$ 
with bounded variation that are characterized by
two variational conditions, called \emph{stability} (S$_\sfd$) and 
\emph{Energy Balance} (E$_\sfd$):
\begin{equation}
  \label{eq:165}
  \cE(t,u(t))\le \cE(t,v)+\sfd(u(t),v)\quad\text{for every }v\in X,\
  t\in [0,T],
  \tag{S$_\sfd$}
\end{equation}
\begin{equation}
  \label{eq:166}
  \cE(t,u(t))+\var(u,[0,t])=\cE(0,u_0)+\int_0^t \cP(r,u(r))\,\d r
  \quad\text{for every }t\in [0,T].
  \tag{E$_\sfd$}
\end{equation}
In \eqref{eq:166} $\var(u,[0,t])$ denotes the usual pointwise total variation of
$u$ on the interval $[0,t]$ (see \eqref{eq:1}) and
$\cP(t,u)=\partial_t \cE(t,u)$ is the partial derivative of the energy
$\cE$ with respect to (w.r.t.) time, which we assume to be continuous
and satisfying the uniform bound
\begin{equation}
  \label{eq:168}
  |\cP(t,x)|\le C_0\big(\cE(t,x)+C_1\big)\quad\text{for every }x\in X
\end{equation}
for some constants $C_0,C_1\ge0$.

As we mentioned, one of the strongest features of the energetic approach is the
possibility to construct energetic solutions by solving the time
\emph{Incremental Minimization scheme} \eqref{eq:167}
(also called \emph{Minimizing Movement method} in the 
De Giorgi approach to metric gradient flows, see \cite{Ambrosio-Gigli-Savare08}).
If $\cE$ has compact sublevels then for every ordered partition 
$\tau=\{t^0_\tau=0,t^1_\tau,\cdots,t^{N-1}_\tau,t^N_\tau=T\}$ of the
interval $[0,T]$ with variable time step
$\tau^n:=t^n_\tau-t^{n-1}_\tau$ and for every
initial choice $U^0_\tau= u(0)$
we can construct by induction an approximate
sequence $(U^n_\tau)_{n=0}^N$ solving \eqref{eq:167}.

If $\overline U_\tau$ denotes the left-continuous piecewise constant
interpolant
of $(U^n_\tau)_n$ which takes the value $U^n_\tau$ on the interval
$(t^{n-1},t^n_\tau],$ then the family of discrete solutions $\overline
U_\tau$
has limit curves with respect to pointwise convergence as 
the maximum of the step sizes
$|\tau|=\max \tau^n$ vanishes, and every limit curve $u$ is an energetic solution.

A second important fact concerns the mutual interaction between
the Stability and the Energy Balance
conditions \eqref{eq:165}-\eqref{eq:166}: it is possible to prove that
for every curve
$u$ satisfying \eqref{eq:165}, relation \eqref{eq:166} is in fact equivalent to
the Energy-Dissipation inequality
\begin{equation}
  \label{eq:166bis}
  \cE(t,u(t))+\var(u,[0,t])\le \cE(0,u_0)+\int_0^t \cP(r,u(r))\,\d r
  \quad\text{for every }t\in [0,T].
\end{equation}
When 
\begin{equation}
X=\R^d,\quad \sfd(x,y):= \alpha|y-x|,\ \alpha>0,\quad 
\text{$\cE$ is sufficiently
smooth,}\label{eq:176}
\end{equation}
and $\rmD^2_x\cE(t,x)\ge\lambda I$, $\lambda>0$, so that
$\cE(t,\cdot)$ is uniformly convex, 
then it is possible to prove 
that energetic solutions are continuous and can be equivalently
characterized by the doubly nonlinear evolution inclusion
\begin{equation}
  \label{eq:169}
  \alpha\,\partial\psi(\dot u(t))+\rmD\cE(t,u(t))\ni0,\quad
  \psi(v):=|v|.
\end{equation}
Even simple $1$-dimensional nonconvex examples,
e.g.~when the energy has the form
\begin{equation}
\text{$\cE(t,x):=W(x)-\ell(t)x$ for a double
well potential such as $W(x)=(x^2-1)^2$,  $x\in \R$,}\label{eq:171}
\end{equation}
show that energetic solutions have jumps, preventing the 
violation of the global stability condition
(different kind of jumps arise from time-discontinuities of the
energy, see e.g.~\cite{Krejci-Liero09}). 
In fact, combining stability and energy balance, it is possible to check that
at every jump point
$t\in \Jump{}u$, the left and right limits $u(t\leftl),\,u(t\rightl)$
of a solution $u$ satisfy the energetic jump conditions 
\begin{equation}
  \label{eq:170}
  \sfd(u(t\leftl),u(t))=\cE(t,u(t\leftl))-\cE(t,u(t)),\quad
  \sfd(u(t),u(t\rightl))=\cE(t,u(t))-\cE(t,u(t\rightl)),
\end{equation}
which are strongly influenced by the global energy landscape of $\cE$.
This reflects the global constraint imposed by the stability condition,
whose violation induces the jump (see
e.g.~\cite[Ex.\,6.3]{Knees-Mielke-Zanini08},
\cite[Ex.\,1]{Mielke-Rossi-Savare09}).

For instance, in the case of example \eqref{eq:176}-\eqref{eq:171}
with 
$\ell\in \rmC^1([0,T])$
strictly increasing with $u_0<-1$ and $\ell(0)=\alpha+W'(u_0)$, 
it is possible to prove \cite{Rossi-Savare13}
that an energetic solution $u$ is an increasing selection of the 
equation
\begin{equation}
  \label{eq:172}
  \alpha+\partial W^{**}(u(t))\ni \ell(t)
\end{equation}
where $W^{**}$ is the convex envelope $W^{**}(x)=\big((x^2-1)_+\big)^2$,
independently of the parameter $\alpha>0$. 
\begin{figure}[!h]
\label{fig:1}
\centering
\begin{tikzpicture}
\draw[->] (-3,0) -- (3,0) node[above] {$u(t)$};
\draw[->] (0,-2.8) -- (0,2.8) node[right] {$\ell(t)-\alpha$};

 \begin{scope}[scale=1.8]
    \draw[domain=-1.4:1.4,samples=100] plot ({\x}, {(\x)^3-\x}) node[above] {$W'(u)$};
     \draw[domain=-1.4:-1,very thick,samples=100] plot ({\x}, {(\x)^3-\x}) [arrow inside={}{0.33,0.66}];
     \draw[domain=1:1.4,very thick,samples=100] plot ({\x}, {(\x)^3-\x}) [arrow inside={}{0.50}]    ;
      \foreach \Point in {(-1,0) ,(1,0) }{
    \draw[fill=blue,blue] \Point circle(0.04);
    }
\draw[->, dashed,blue] (-1,0) to [bend left=80, looseness=1.3] (0.96,0.05);
\node[below] at (-1,0) {\qquad\,\,$u(t-)$};
\node[below] at (1,0) {\qquad$u(t+)$};
  \end{scope}
  
\end{tikzpicture} 
\caption{Energetic solution for a double-well energy $W$ with an
  increasing load $\ell$, see \eqref{eq:172}.}
\end{figure}
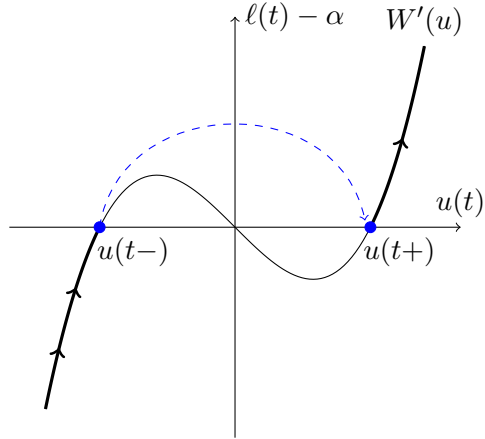

\paragraph{\emph{Balanced Viscosity solutions}.}
In order to obtain a formulation where local effects are more
relevant (see 
 \cite[Sec.\,6]{Miel03EFME}, 
 \cite{DalMaso-Toader02, NegOrt07?QSCP,Efendiev-Mielke06,Larsen10,Mielke-Zelik14})
 various kinds of corrections have been considered.
A natural one introduces 
a viscous correction to the incremental minimization scheme
\eqref{eq:167}, penalizing
 the square of the distance from the
previous step
\begin{equation}
  \label{eq:167e}
  \min_{U\in X} \cE(t^n_\tau,U)+\sfd(U^{n-1}_\tau,U)+\frac {\eps^n}{2\tau^n}\sfd^2(U^{n-1}_\tau,U),
  \tag{IM$_{\sfd,\eps}$}
\end{equation}
for a parameter $\eps^n=\eps^n(\tau)\down0$ with 
$\frac {\eps^n(\tau)}{|\tau|}\up+\infty$.
In the previous Euclidean framework \eqref{eq:176}, \eqref{eq:167} corresponds to 
the discretization of the generalized gradient flow
\begin{equation}
  \label{eq:173}
  \alpha\partial\psi(\dot u(t))+\eps\dot u(t)+\rmD\cE(t,u(t))\ni0.
\end{equation}
Such kinds of approximations have been
studied in a series of contributions 
\cite{Rossi-Mielke-Savare08,Mielke-Rossi-Savare09,Mielke-Rossi-Savare12,Mielke-Rossi-Savare13},
also dealing with more general corrections in metric and linear
settings
(ses also a comparison between other possible notions in
\cite[Sec.~5]{Mielke-Rossi-Savare09}, \cite{Mielke11-CIME} and a
similar approach for finite-strain elasto-plasticity in \cite{Rindler15}). 
Under suitable smoothness and lower semicontinuity assumptions
involving the metric slope of $\cE$ it is possible to prove that 
all the limit curves satisfy a local stability assumption and a
modified Energy Balance, involving 
an augmented total variation that 
encodes a more refined description of the jump behaviour of $u$:
roughly speaking, a jump between $u(t-)$ and $u(t+)$ occurs only when
these values can be connected by a rescaled solution $\vartheta$ of
\eqref{eq:173}, where the energy is frozen at the jump time $t$
(see the next section 
\ref{subsec:BVsolutions}):
\begin{equation}
  \label{eq:173bis}
  \alpha\partial\psi(\dot \vartheta(s))+\dot \vartheta(s)+\rmD\cE(t,\vartheta(s))\ni0.
\end{equation}
\begin{figure}[!h]
\centering
  \begin{tikzpicture}
\draw[->] (-3,0) -- (3,0) node[above] {$u(t)$};
\draw[->] (0,-2.8) -- (0,2.8) node[right] {$\ell(t)-\alpha$};

 \begin{scope}[scale=1.8]
    \draw[domain=-1.4:1.4,samples=100] plot ({\x}, {(\x)^3-\x}) node[above] {$W'(u)$};
    \draw[domain=-1.4:-0.57,very thick,samples=100] plot ({\x}, {(\x)^3-\x}) [arrow inside={}{0.33,0.66}];
     \draw[domain=1.15:1.4,very thick,samples=100] plot ({\x}, {(\x)^3-\x}) [arrow inside={}{0.50}]    ;
     \draw[domain=-0.57:1.15,thick,blue,samples=100] plot ({\x}, {0.385}) [arrow inside={}{0.33,0.66}];
     \foreach \Point in {(-0.57,0.385),(1.15,0.385)}{
    \draw[fill=blue,blue] \Point circle(0.04);
}
\draw[dotted] (-0.57,0.385) -- (-0.57,0) node[below] {$u(t-)$};
\draw[dotted] (1.15,0.385) -- (1.15,0) node[below] {\,\,\,$u(t+)$};

 \end{scope}
  
\end{tikzpicture}

  \caption{BV solution for a double-well energy $W$ with an
  increasing load $\ell$. The blue line denotes the 
  path described by the optimal transition $\vartheta$ solving 
\eqref{eq:173bis}.}
  \label{fig:2}
\end{figure}
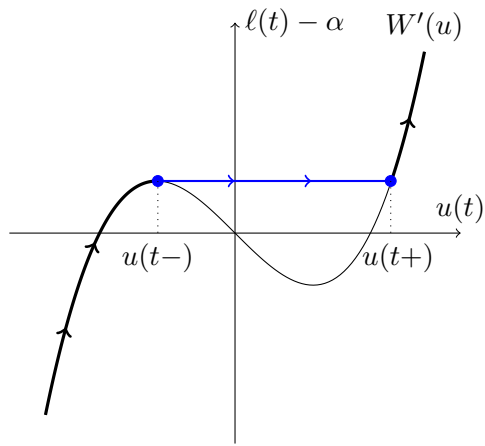

\noindent
One of the main technical difficulties of the theory of Balanced
Viscosity solutions is related to the properties of the slope of $\cE$,
which 
can be difficult to check when highly nonsmooth-nonconvex energies are
involved.

More degenerate situations when $\alpha=0$ can also be considered,
both from the continuous (see \cite{Agostiniani-Rossi16})
and the discrete point of view
(see \cite{Artina-Cagnetti-Fornasier-Solombrino15}, who considers 
a different dependence with respect to time, given by a time-dependent linear constraint):
the main difficulty here relies on the loss of time-compactness, since
simple estimates of the total variation of the approximating curves 
are missing.
\paragraph{\emph{Viscous corrections of the Incremental Minimization Scheme}}
The present paper introduces and studies an intermediate situation
between
Energetic and Balanced Viscosity solutions, when one keeps constant the ratio
$\mu:=\eps^n/\tau^n$ in \eqref{eq:167e}. In this way 
the metric dissipation $\sfd$ is corrected by an extra viscous
penalization term
$\delta(u,v):=\frac{\mu}2\sfd^2(u,v)$ which induces a localization 
of the minimizer, tuned by the parameter $\mu>0$.
At each step $n$ we thus
propose to solve a modified
Incremental Minimization scheme of the form
\begin{equation}
  \label{eq:167mu}
  \begin{aligned}
    &\text{select $U^n_\tau\in \rmM(t^n_\tau,U^{n-1}_\tau),\qquad$ where for every
      $t\in [0,T]$ and $x\in X$}\\
    &\rmM(t,x):=\argmin_{y\in X} \Big\{
    \cE(t,y)+\sfd(x,y)+\delta(x,y)\Big\},\qquad
    \delta(x,y):=\frac{\mu}2\sfd^2(x,y).
  \end{aligned}
  \tag{IM$_{\sfd,\delta}$}
\end{equation}
 In the particular setting of crack propagation, a similar kind of 
corrections have already been considered by
\cite{DalMaso-Toader02}: in that case $\delta$ arises from a different
(semi-)distance $\mathsf d_\ast$.
 Even if general viscous corrections $\delta$ could be considered
(but still satisfying suitable compatibility conditions, see Section \ref{subsec:viscous-correction}), in this Introduction
we will choose the simpler quadratic one for ease of exposition.

Notice that in the finite dimensional case \eqref{eq:176} when the Hessian
of the energy is bounded from below, 
i.e.~$\rmD_x^2\cE(t,x)\ge -\lambda I$ for every $t,x$ and some
$\lambda\ge0$, 
the choice $\mu\alpha^2\ge \lambda$ yields a \emph{convex} incremental
problem \eqref{eq:167mu}, which could
greatly help in the effective computation of the solution.
Differently from \cite{Artina-Cagnetti-Fornasier-Solombrino15},
we do not need to construct $U^n_\tau$ by freezing the time variable
at $t^n_\tau$ and iterating the minimization scheme 
to converge to a critical point: after each incremental minimization
step the energy is immediately updated to the new value
at the time $t^{n+1}_\tau$.

Since $\delta\ge0$, it is not difficult to check that 
the family of discrete solutions $\overline U_\tau$ has uniformly
bounded $\sfd$-total variation and takes value in a compact set of
$X$, so that it always admits
limit curves $u\in \BV{\sfd}([0,T];X)$. 
The difficult task here
concerns the characterization of such limit curves. 
One of the main problems underlying the simple scheme \eqref{eq:167mu} is 
the loss of the triangle inequality for the total dissipation 
\begin{equation}
  \label{eq:174}
  \sfD(u,v):=\sfd(u,v)+\delta(u,v)
  = \sfd(u,v)+\frac\mu2 \sfd^2(u,v).
\end{equation}
In the case of Energetic solutions, the triangle inequality of $\sfd$
lies at the core of two crucial properties:
\begin{enumerate}[a)]
\item every solution $U^n$ of the minimization step \eqref{eq:167}
  satisfies the stability condition \eqref{eq:165} at $t=t^n_\tau$;
\item the computation of the total variation of a piecewise constant map $\overline U_\tau$ 
  associated with some partition $\tau$ involves only consecutive
  points, i.e.
  \begin{displaymath}
    \var(\overline U_\tau,[0,T])=
    \sum_{n=1}^N\sfd(U^{n-1}_\tau,U^n_\tau)
  \end{displaymath}
  and the total variation functional $u\mapsto \var(u,[0,T])$ is lower
  semicontinuous w.r.t.~pointwise convergence, so that at least an
  Energy inequality corresponding to \eqref{eq:166} can be 
  easily deduced from the corresponding version at the discrete level.
\end{enumerate}
Such properties fail in the case of the augmented dissipation $\sfD$
of \eqref{eq:174}. In particular, even in the finite-dimensional
setting \eqref{eq:176} with $\alpha=1$, it is easy to check that e.g.~Lipschitz curves 
$u:[0,T]\to \R^d$ can be approximated by piecewise constant
interpolants $\overline U_\tau$ on uniform partitions 
$\tau=\{n T/N\}_{n=0}^N$ with $U^n_\tau=u(nT/N)$ and $|\tau|=T/N$, so that 
\begin{equation}
  \label{eq:175}
  \lim_{|\tau|\down0}\sum_{n=1}^N
  |U^{n-1}_\tau-U^n_\tau|+\frac\mu2|U^{n-1}_\tau-U^n_\tau|^2=
  \var(u,[0,T])=\int_0^T |\dot u(t)|\,\d t.
\end{equation}
\paragraph{\emph{Visco-Energetic solutions.}}
Nevertheless, by using more refined arguments and guided by the
results
obtained in the Balanced Viscosity approach, we are able to obtain 
a precise variational characterization of the limit curves (called
\emph{Visco-Energetic solutions}), still
stated in terms of suitably adapted stability and energy balance
conditions.

Concerning stability, we obtain a natural generalization of \eqref{eq:165}
\begin{equation}
  \label{eq:165D}
  \cE(t,u(t))\le \cE(t,v)+\sfD(u(t),v)\quad\text{for every }v\in X,\
  t\in [0,T]\setminus \Jump{}u,
  \tag{S$_\sfD$}
\end{equation}
which is naturally associated with the $\sfD$-\emph{stable set}
\begin{equation}
  \label{eq:182}
  \SSD:=\Big\{(t,x):\cE(t,x)\le \cE(t,y)+\sfD(x,y)\quad
  \text{for every }y\in X\Big\}.
\end{equation}
\eqref{eq:165D} is in good accordance with
\cite{DalMaso-Toader02}, where a similar condition has been found (see
Theorem 3.3(b)).
Notice that in the  finite dimensional case \eqref{eq:176} when
$\mu$ is sufficiently big so that 
$\rmD^2_x\cE(t,x)\ge -\mu \alpha^2 I$, \eqref{eq:165} is in fact a local
condition, which can be restated as
\begin{equation}
  \label{eq:177}
  |\rmD_x\cE(t,u(t))|\le \alpha\quad\text{for every }t\in [0,T]\setminus \Jump{}u.
\end{equation}
This shows that Visco-Energetic solutions also satisfy the basic
local stability condition, shared by all kind of solutions to
variationally driven
rate-independent problems (see \cite{Mielke11-CIME}, \cite[Sec.~1.8, 3.3]{Mielke-Roubicek15}). 

The right replacement of the Energy Balance condition is harder to
formulate and it is one of the main contribution of the present
paper. 
Since \eqref{eq:165D} is weaker than \eqref{eq:165}, it is
clear
that the Energy Dissipation inequality \eqref{eq:166bis} (which still trivially
holds for limits of \eqref{eq:167mu}) will not be enough
to recover the energy balance: in particular, important pieces of
information are lost along the jumps. 
This is a typical situation 
arising in many other approaches (see e.g.~the discussion in
\cite[Sec.~3.3.3]{Mielke-Roubicek15})
and leading to ad-hoc reinforcements of the jump conditions,
as for BV solutions (see also the notion of
maximally dissipative solutions \cite{Stefanelli09,Roubicek15}, whose
existence in general cases is
however not clear at the present stage of the theory).

In the present case of the Visco-Energetic approach, a heuristic idea, which one can figure out by the direct analysis of simple cases
such as \eqref{eq:176}-\eqref{eq:171}, is that jump transitions
between $u(t-)$ and $u(t+)$ should be described by discrete
trajectories
$\vartheta:Z\to X$ defined in a subset $Z\subset \Z$ 
such that each value $\vartheta(n)\in \rmM(t,\vartheta(n-1))$ is a minimizer of the ``frozen''
incremental problem at time $t$ with datum $\vartheta(n-1)$.
In the simplest cases $Z=\Z$, the left and right jump values
$u(t\pm)$ are the limit of $\vartheta(n)$ as $n\to\pm\infty$, but more
complicated situations can occur, when $Z$ is a proper subset of $\Z$
or one has to deal with concatenation of (even countable) discrete
transitions and sliding parts parametrized by a continuous variable, where the stability condition
\eqref{eq:165D} holds. 

In order to capture all of these possibilities, we will introduce a quite general
notion of transition parametrized by a continuous map $\vartheta:E\to X$ defined
in an arbitrary compact subset of $\R$ such that $\vartheta(\min
E)=u(t-)$ and $\vartheta(\max E)=u(t+)$.
The cost of such kind of transition results from the contribution of three
parts: the first one is the usual total variation $\var(\vartheta,E)$ (see the
next \eqref{eq:1}). The second contribution arises at each 
``gap'' in $E$, i.e.~a bounded connected component $I=(I^-,I^+)$ 
of $\R\setminus E$: denoting by $\frH(E)$ the collection of all these
intervals, we will set
\begin{equation}
  \label{eq:179}
  \Cd(\vartheta,E):=\sum_{I\in \frH(E)}\delta(\vartheta(I^-),\vartheta(I^+)).
\end{equation}
The last contribution detects if $\vartheta$ violates the stability
condition at $s\in E$: it
is defined as the sum
\begin{equation}
  \label{eq:180}
  \sum_{s\in E\atop s<\max E}\cR(t,\vartheta(s))
\end{equation}
where $\cR$ is the residual stability function
\begin{equation}
  \label{eq:181}
  \cR(t,x):=\max_{y\in X}\cE(t,x)-\cE(t,y)-\sfD(x,y)=
  \cE(t,x)-\min_{y\in X}\Big(\cE(t,y)+\sfD(x,y)\Big).
\end{equation}
Since it is easy to check that $\cR(t,\theta)=0$ if and only if
$(t,\theta)\in \SSD$, $\cR(t,\cdot)$ provides a measure of 
the violation of the stability constraint. 

The total cost of a transition $\vartheta:E\to X$ at a jump time $t$ is therefore
\begin{equation}
  \label{eq:183}
  \Cf(t,\vartheta,E):=\var(\vartheta,E)+\Cd(\vartheta,E)+\sum_{s\in E\atop s<\max E}\cR(t,\vartheta(s)),
\end{equation}
and the corresponding cost $\sfc$ for a jump from $u(t-)$ to $u(t+)$
passing through the value $u(t)$
is given by
\begin{equation}
  \label{eq:184}
  \begin{aligned}
    \sfc(t,u(t-),u(t),u(t+)):=\inf\Big\{&\Cf(t,\vartheta,E):\vartheta\in
    \rmC(E,X),\ \vartheta(E)\ni u(t),\\
    &\ \vartheta(\min E)=u(t-),\ \vartheta(\max E)=u(t+)\Big\},
  \end{aligned}
\end{equation}
where the infimum is attained whenever there is at least one
admissible transition with finite cost. Notice that the cost $\sfc$ 
is always bigger than the corresponding value computed by the
dissipation distance $\sfd$, i.e.~the quantity
\begin{equation}
  \label{eq:186}
  \Delta_\sfc(t,u(t-),u(t),u(t+)):= \sfc(t,u(t-),u(t),u(t+))-\sfd(u(t-),u(t))-\sfd(u(t),u(t+))
\end{equation}
is nonnegative. 
$\sfc$ always controls the energy dissipation along the
jump, i.e.~
\begin{equation}
  \label{eq:188}
  \Cf(t,\vartheta,E)\ge \sfc(t,u(t-)),u(t),u(t+))\ge 
  \cE(t,u(t-))-\cE(t,u(t+)).
\end{equation}
With these notions at our disposal, we can eventually write the 
Energy Balance condition for Visco-Energetic solutions
\begin{equation}
  \label{eq:166ve}
  \cE(t,u(t))+\mVar{\sfd,\sfc}(u,[0,t])=\cE(0,u_0)+\int_0^t \cP(r,u(r))\,\d r
  \quad\text{for every }t\in [0,T],
  \tag{E$_{\sfd,\sfc}$}
\end{equation}
where 
the augmented total variation $\mVar{\sfd,\sfc}(u,[a,b])$ 
differs from the usual one $\var(u,[a,b])$ by an extra 
contribution at the jump points $t\in \Jump{}u$:
\begin{equation}
  \label{eq:185}
  \begin{aligned} 
    &\mVar{\sfd,\sfc}(u,[a,b]):= \var(u,[a,b])
    +\sum_{t\in
      \Jump{}u\cap[a,b]}\Delta_\sfc(t,u(t-),u(t),u(t+)).
  \end{aligned}
\end{equation}
As in the case of energetic solutions, once 
the stability condition \eqref{eq:165D} is satisfied,
it is sufficient to check the Energy-Dissipation inequality associated
with \eqref{eq:166ve}, since the extra term appearing in the definition
of 
$\mVar{\sfd,\sfc}(u,[0,t])$ \eqref{eq:185} 
provides the right correction to compensate 
the weaker stability property. 
At each jump point we thus obtain the Visco-Energetic jump conditions
corresponding to \eqref{eq:170}
\begin{equation}
  \label{eq:170ve}
  \sfc(t,u(t\leftl),u(t))=\cE(t,u(t\leftl))-\cE(t,u(t)),\quad
  \sfc(t,u(t),u(t\rightl))=\cE(t,u(t))-\cE(t,u(t\rightl)).
\end{equation}
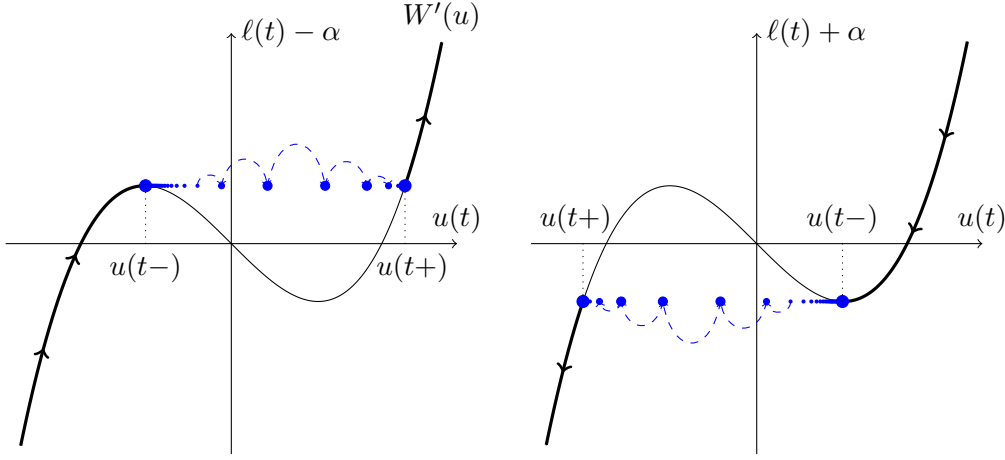
\begin{figure}[!h]
\centering
\begin{tikzpicture}

\draw[->] (-3,0) -- (3,0) node[above] {$u(t)$};
\draw[->] (0,-2.8) -- (0,2.8) node[right] {$\ell(t)-\alpha$};

 \begin{scope}[scale=2]
    \draw[domain=-1.4:1.4,samples=200] plot ({\x}, {(\x)^3-\x}) node[above] {$W'(u)$};
     \draw[domain=-1.4:-0.57,very thick,samples=200] plot ({\x}, {(\x)^3-\x}) [arrow inside={}{0.33,0.66}];
     \draw[domain=1.1547:1.4,very thick,samples=200] plot ({\x}, {(\x)^3-\x}) [arrow inside={}{0.50}]    ;
     \foreach \Point in {(-0.57,0.3849),(1.1547,0.3849)}{
       \draw[fill=blue,blue] \Point circle(0.04);
}
     \foreach \Point in {(-0.542,0.3849),(-0.536,0.3849), (-0.530,0.3849), (-0.524,0.3849), (-0.518,0.3849), (-0.512,0.3849), (-0.506,0.3849), (-0.500,0.3849), (-0.494,0.3849), (-0.487,0.3849), (-0.479,0.3849), (-0.469,0.3849),(-0.457,0.3849),(-0.442,0.3849),(-0.423,0.3849),(-0.398,0.3849),(-0.363,0.3849),(-0.311,0.3849),(-0.225,0.3849),(-0.065,0.3849),(0.241,0.3849),(0.624,0.3849),(0.901,0.3849),(1.045,0.3849),(1.109,0.3849),(1.136,0.3849),(1.147,0.3849),(1.151,0.3849),(1.153,0.3849)}{
    \draw[fill=blue,blue] \Point circle(0.01);
}
\draw[->, dashed,blue]  (-0.225,0.3849) to [bend left=90,looseness=1.5]
(-0.065,0.3849);
\draw[fill=blue,blue] (-0.065,0.3849) circle(0.02);
\draw[->, dashed,blue]  (-0.065,0.3849) to [bend left=80,looseness=2]
(0.241,0.3849);
\draw[fill=blue,blue] (0.241,0.3849) circle(0.03);
\draw[->, dashed,blue]  (0.241,0.3849) to [bend left=80,looseness=2.5]
(0.624,0.3849);
\draw[fill=blue,blue] (0.624,0.3849) circle(0.03);
\draw[->, dashed,blue]  (0.624,0.3849) to [bend left=80, looseness=2]
(0.901,0.3849);
\draw[fill=blue,blue] (0.901,0.3849) circle(0.03);
\draw[->, dashed,blue]  (0.901,0.3849) to [bend left=80, looseness=1.5]
(1.045,0.3849);
\draw[fill=blue,blue](1.045,0.3849) circle(0.02);

\draw[dotted] (-0.57,0.3849) -- (-0.57,0) node[below] {$u(t-)$};
\draw[dotted] (1.1547,0.3849) -- (1.1547,0) node[below] {\,\,\,$u(t+)$};
  \end{scope}
\end{tikzpicture}
\quad
\begin{tikzpicture}

\draw[->] (-3,0) -- (3,0) node[above] {$u(t)$};
\draw[->] (0,-2.8) -- (0,2.8) node[right] {$\ell(t)+\alpha$};

 \begin{scope}[scale=2]
    \draw[domain=-1.4:1.4,samples=200] plot ({-\x}, {(-\x)^3+\x});
     \draw[domain=-1.4:-0.57,very thick,samples=200] plot ({-\x}, {(-\x)^3+\x}) [arrow inside={}{0.33,0.66}];
     \draw[domain=1.1547:1.4,very thick,samples=200] plot ({-\x}, {(-\x)^3+\x}) [arrow inside={}{0.50}]    ;
     \foreach \Point in {(0.57,-0.3849),(-1.1547,-0.3849)}{
    \draw[fill=blue,blue] \Point circle(0.04);
}
     \foreach \Point in {(0.542,-0.3849),(0.536,-0.3849), (0.530,-0.3849), (0.524,-0.3849), (0.518,-0.3849), (0.512,-0.3849), (0.506,-0.3849), (0.500,-0.3849), (0.494,-0.3849), (0.487,-0.3849), (0.479,-0.3849), (0.469,-0.3849),(0.457,-0.3849),(0.442,-0.3849),(0.423,-0.3849),(0.398,-0.3849),(0.363,-0.3849),(0.311,-0.3849),(0.225,-0.3849),(0.065,-0.3849),(-0.241,-0.3849),(-0.624,-0.3849),(-0.901,-0.3849),(-1.045,-0.3849),(-1.109,-0.3849),(-1.136,-0.3849),(-1.147,-0.3849),(-1.151,-0.3849),(-1.153,-0.3849)}{
    \draw[fill=blue,blue] \Point circle(0.01);
}

\draw[->, dashed,blue]  (0.225,-0.3849) to [bend left=90,looseness=1.5]
(0.065,-0.3849);
\draw[fill=blue,blue] (0.065,-0.3849) circle(0.02);
\draw[->, dashed,blue]  (0.065,-0.3849) to [bend left=80,looseness=2]
(-0.241,-0.3849);
\draw[fill=blue,blue] (-0.241,-0.3849) circle(0.03);
\draw[->, dashed,blue]  (-0.241,-0.3849) to [bend left=80,looseness=2.5]
(-0.624,-0.3849);
\draw[fill=blue,blue] (-0.624,-0.3849) circle(0.03);
\draw[->, dashed,blue]  (-0.624,-0.3849) to [bend left=80,looseness=2]
(-0.901,-0.3849);
\draw[fill=blue,blue] (-0.901,-0.3849) circle(0.03);
\draw[->, dashed,blue]  (-0.901,-0.3849) to [bend left=80,looseness=1.5]
(-1.045,-0.3849);
\draw[fill=blue,blue](-1.045,-0.3849) circle(0.02);

\draw[dotted] (0.57,0.-0.3849) -- (0.57,0) node[above] {$u(t-)$};
\draw[dotted] (-1.1547,-0.3849) -- (-1.1547,0) node[above] {$u(t+)\,\,\,$};
  \end{scope}
\end{tikzpicture}

  \caption{Visco-Energetic solutions for a double-well energy $W$ with an
  increasing (on the left) and a decreasing (on the right) load $\ell$
  and the choices $\mu\alpha^2\ge -\min W''$. 
  Jumps occur when $u$ reaches a local maximum (or minimum) of $W$
  as in the Balanced Viscosity case.
  The optimal transitions $\vartheta$ are infinite sequences of jumps.}
  \label{fig:3}
\end{figure} 

\noindent
Differently from other situations where only partial asymptotic
information can be recovered in the limit (see e.g.~\cite[Def.~3.2,
Thm.~3.3(c,e)]{DalMaso-Toader02}),
one of the beautiful aspects of the Visco-Energetic setting
is that for each $t\in \Jump{}u$ there always exists an optimal
transition 
$\vartheta:E\to X$ connecting $u(t-)$ to $u(t+)$ and passing through
$u(t)$ such that 
\begin{equation}
  \label{eq:187}
  \cE(t,u(t\leftl))-\cE(t,u(t\rightl))=\Cf(t,\vartheta,E).
\end{equation}
In the case when $(t,\vartheta(s))\not\in \SSD$ for some $s\in E$ we
can
prove that $s$ is isolated and denoting by $s_-:=\max E\cap
(-\infty,s)$ we recover the property
\begin{equation}
  \label{eq:189}
  \vartheta(s)\in \rmM(t,\vartheta(s_-)),
\end{equation}
which provides an important description of optimal transitions
(see Figure \ref{fig:3} for a simple example).

\paragraph{\emph{Further generalizations and scope of the
    Visco-Energetic theory.}}
In the paper we try to develop the ideas above at the 
highest level of generality;
in particular
\begin{enumerate}[a)]
\item we separate the roles of the dissipation distance 
  and of the topology, by considering a general metric-topological setting, where 
  the compactness assumptions are stated in terms of a weaker topology
  $\sigma$
  (see Section \ref{subsec:metric-topological}).
\item we consider general lower semicontinuous 
  asymmetric quasi-distances $\sfd$,
  possibly taking the value $+\infty$ 
  \eqref{eq:16}.
  As in the energetic framework, in this case a further closedness
  condition   involving the stable set will play a crucial role: 
  in the visco-energetic setting, we will need the closure of the 
  $Q$-quasi stable sets, $Q\ge 0$, 
  of the points $x\in X$ satisfying
  \begin{equation}
    \label{eq:191}
    \cE(t,x)\le \cE(t,y)+\sfD(x,y)+Q\quad\text{for every }y\in X.
  \end{equation}
  Notice that \eqref{eq:191} reduces to the definition of
  the stable set when $Q=0$.
\item we try to relax the assumptions concerning
  the time-differentiability of $\cE$, thus allowing for super-differentiable
  energies: this is particularly useful to cover the important case of 
  product spaces $X:=Y\times Z$ 
  (see Section \ref{subsec:marginal}) 
  where 
  $\sfd$ controls only the $Z$-component and 
  one has to deal with
  reduced/marginal energies
  \begin{equation}
    \label{eq:190}
    \widetilde\cE(t,z):=\min_{y\in Y}\cE(t,y,z).
  \end{equation}
\item we consider quite general viscous corrections $\delta$, not
  necessarily obtained as a function of the quasi-distance $\sfd$
  (see Section \ref{subsec:viscous-correction}).  
\end{enumerate}

This generality aims to develop a broad-ranging theory, which can
hopefully reach the same 
power of the energetic one. In particular, d) ensures a lot of
flexibility on the choice of the localizing term, c) guarantees the
potential applicability
to the challenging cases of quasi-static evolutions where part of the
unknowns are not stabilized by
dissipation effects, a) and b) are intended for 
separating the ``technical'' choice of the topology (which should be
sufficiently weak to gain the compactness of the energy sublevels)
from the choice of the dissipation, which is dictated by the model.
As the examples of Section \ref{sec:examples} show, 
one can expect that the similarity between the closure condition of the $Q$-stable set
\eqref{eq:121}
and the closure condition of the standard stable set \eqref{eq:165} (which is one
of the
main requirements of the Energetic theory) will allow
to extend a large part of the available tools and techniques
originally developed
for the Energetic setting to the Visco-Energetic framework.

Even if the Visco-Energetic approach involves a more complicated
characterization of the jump transitions, it preserves two of the crucial aspects of
the Energetic theory: a quite robust metric/variational description of the evolution in
terms of stability and energy balance, combined with a simple approximation algorithm 
whose convergence relies on a few basic structural properties of
energy and dissipation.
It moreover keeps the same localization effect of the Balanced-Viscosity
approach under considerably weaker assumptions on the data. 
It is worth noticing that one can recover the Energetic solutions
(respectively, the Balanced Viscosity solutions) as a limit as the
viscous parameter $\mu$ of \eqref{eq:174} goes to $0$ (resp.~to
$+\infty$):
see \cite{Rossi-Savare17-preprint}.

Besides applications to various models where the Energetic approach
have been successfully employed (we plan to apply the present theory to elastoplasticty and crack propagation
models, following the approaches of
\cite{MaiMie08?GERI,Rindler15,MRS16,DalMaso-Francfort-Toader05}),
further important developments have to be better
understood: one of the most interesting one concerns the case of a
viscosity term $\delta$ which is not ``controlled'' by the distance
$\sfd$.
This situation may occur when $\sfd=0$, causing severe
compactness issues (this is the most difficult case, see
\cite{Agostiniani-Rossi16,Artina-Cagnetti-Fornasier-Solombrino15}),
or when the evolution involves a coupling between quasi-static and
viscous laws \cite{DalMaso-Toader02}, leading to a problematic
formulation of the energy balance.

\paragraph{\emph{Plan of the paper.}}
In the preliminary section \ref{sec:prelresults}
we briefly recall the canonical metric-topological setting, 
how to deal with BV and regulated functions, the properties
of the energy $\cE$ and its power $\cP$, 
the basic framework of Energetic and Balanced Viscosity solutions.

\emph{\bfseries We collect our main results in section \ref{sec:VE}:}
we start by discussing admissible viscous corrections $\delta$
and we introduce in full detail the associated viscous jump cost
$\sfc$,
relying on generalized transitions, and the residual
stability function $\cR$ (Section \ref{subsec:residual}).
\emph{\bfseries Section \ref{subsec:VE} contains the precise definition
of VE solutions, their basic characterizations, 
and the main existence theorem \ref{thm:existence}.}
In the case when the dissipation distance $\sfd$ 
is not continuous w.r.t.~$\sigma$, the properties
of $\cR$ will play a crucial role, so we 
investigate them in Section \ref{subsec:residual-stability}
and we will apply these results to 
elucidate the structure of optimal jump transitions
in Section \ref{subsec:optimal-transitions}.

Examples and applications are collected in Section \ref{sec:examples},
starting from the simplest convex or $1$-dimensional cases,
and moving towards more complicated situations,
where $\delta$ may depend on an accessory distance $\sfd_*$ 
(Section \ref{subsec:genconvexity}), 
$\sfd$ is degenerate but still separates the stable set
(Section \ref{subsec:degenerate}),
$X$ is a product space and we will have to
introduced a reduced marginal energy as in \eqref{eq:190} 
(Section \ref{subsec:marginal}).

The last sections contain all the proofs and the relevant 
properties of the transition cost and the Viscous Incremental 
Minimization scheme.
Section \ref{sec:dissipationcost} is devoted to 
the properties of the cost $\mathrm{Trc}$ 
of a transition $\vartheta$ 
and to the existence of optimal transitions.

Section \ref{sec:energy-inequalities} contains the crucial
lower energy estimates
along jumps and along arbitrary BV curves 
satisfying the stability condition \eqref{eq:165D}, thus proving that 
\begin{equation}
  \label{eq:166lower}
  \cE(t,u(t))+\mVar{\sfd,\sfc}(u,[s,t])\ge\cE(0,u(s))+\int_s^t \cP(r,u(r))\,\d r
  \quad\text{for every }s,t\in [0,T],\ s\le t,
\end{equation}
whenever \eqref{eq:165D} holds in $[0,T]$.

The last Section \ref{sec:convergenceproof} 
contains all the main steps of the proof, which
follows a canonical strategy: discrete estimates 
for the Viscous Incremental Minimization scheme
\eqref{eq:167mu},
compactness, energy-dissipation inequality 
\begin{equation}
  \label{eq:166upper}
  \cE(t,u(t))+\mVar{\sfd,\sfc}(u,[0,t])\le\cE(0,u_0)+\int_0^t \cP(r,u(r))\,\d r
  \quad\text{for every }t\in [0,T],
\end{equation}
obtained
by the lower semicontinuity results of Section
\ref{sec:dissipationcost},
and conclusion by reinforcing energy convergence at each time $t$ 
thanks to \eqref{eq:166lower}.
\subsection*{Acknowledgment}
We would like to thank Riccarda Rossi for insightful
discussions on the whole subject and for valuable comments on 
the manuscript.

Part of this paper has been written while the second author 
was visiting the Erwin Schr\"odinger Institute for Mathematics and Physics
(Vienna), whose support is gratefully acknowledged.

\newpage
\section*{List of notation}
\addcontentsline{toc}{subsection}{\it List of notation.}

\smallskip
\halign{$#$\hfil\ &#\hfil
\cr
(X,\sigma)&The reference topological Hausdorff space, Section \ref{subsec:metric-topological}
\cr
\sfd&the asymmetric (quasi-)distance on $X$, \eqref{eq:16}\cr
\sfd\text{ separates } U&\eqref{eq:25}\cr
\var(u,E)&pointwise total variation w.r.t.~$\sfd$ of $u:E\to X$,
\eqref{eq:1}\cr
(\sigma,\sfd)\text{-regulated functions}&Definition
\ref{def:regulated}\cr
\Jump{}u&Jump set of a $(\sigma,\sfd)$-regulated function $u$, \eqref{eq:3}\cr
\BV{\sigma,\sfd}([a,b];X)&Space of $(\sigma,\sfd)$-regulated function
with bounded variation, \ref{def:regulated}\cr
\Delta_\sfe(t,u^-,u^+)&incremental cost function associated with $\sfe$,
\eqref{eq:31}\cr
\Jmp\sfe(u,[a,b])&incremental jump variation induced by
$\Delta_\sfe$, Definition \ref{def:augmented}\cr
\mVar{\sfd,\sfe}(u,[a,b])&Augmented total variation, Definition
\ref{def:augmented}\cr
\cE(t,u)&the energy functional, Section
\ref{subsec:energy-functional}\cr
\cP(t,u)&the power functional, $\partial_t \cE(t,u)$, Section
\ref{subsec:energy-functional}\cr
\cF(t,u),\cF_0(u)&perturbed energy through the distance $\sfd$,
\eqref{eq:35}\cr
\sfD(u,v),\delta(u,v)&modified viscous dissipations, \eqref{eq:6}\cr
(X,\cE,\sfd,\delta)&the basic Visco-Energetic Rate-Independent System\cr
\SSD,\SSD(t)&the $\sfD$-stable set and its sections, Definition
\ref{def:stable-set}\cr
\cR(t,u)&the residual stability function, Definition
\ref{def:slope}\cr
\rmM(t,u)&the set of minimizers of the Incremental
Minimization scheme, \eqref{eq:107}\cr
\tau&a finite partition $\{t^0_\tau,t^1_\tau,\cdots,t^N_\tau\}$ 
of the time interval $[0,T]$, see page \pageref{page:incremental} \cr
U^n_\tau&discrete solutions to the time incremental minimization
scheme
\eqref{ims}
\cr
\overline U_\tau&left continuous piecewise constant interpolant of the values $U^n_\tau$\cr
E^-,E^+&infimum and supremum of a set $E\subset \R$, Section
\ref{subsec:residual}\cr
\Holes(E)&collection of the bounded connected components of $\R\setminus
E$, Section
\ref{subsec:residual}\cr 
\Pf(E)&collection of all the finite subset of $E$\cr
\rmC_{\sigma,\sfd}(E;X)&$\sigma$- and $\sfd$- continuous functions
$f:E\to X$, \eqref{eq:46}\cr
\Cf(t,\vartheta,E)&the transition cost, Definition \ref{def:transition-cost}\cr
\Cd(\vartheta,E)&one of the component of the transition cost, Definition \ref{def:transition-cost}\cr
\sfc(t,u^-,u^+)&the Visco-Energetic jump dissipation cost, Definition
\ref{dissipationcost}\cr
&\cr
\text{\textbf{Assumptions:}}&\cr
\text{\mytag A{}}&Energy and power, Page \pageref{A}\cr
\text{\mytag B{}}&Admissible viscous corrections, Page \pageref{B}\cr
\text{\mytag C{}}&Closure and separation of the (quasi)-stable set, Page \pageref{C}\cr
}

\newpage
\section{Notation, assumptions and preliminary results} \label{sec:prelresults}
In this section we recall some notation and properties related
to asymmetric (quasi-)distances in topological spaces, regulated $\rm{BV}$
functions and 
\textit{Energetic} and \textit{Balanced Viscosity
  $\mathrm{(BV)}$ solutions} of a general rate-independent system.

\subsection{The metric-topological setting.}
\label{subsec:metric-topological}
Let $(X,\topol)$ be a Hausdorff topological space satisfying the first
axiom of countability; we will fix a reference point $x_o\in X$ and 
a time interval
$[0,T]\subset \R$, $T>0$. 

\paragraph*{Asymmetric dissipation distances.}
\addcontentsline{toc}{subsubsection}{\it Asymmetric dissipation distances.}
The first basic object characterizing a Rate-Independent System (R.I.S.)
is
\begin{equation}
  \label{eq:16}
  \begin{gathered}
    \text{a l.s.c.~\emph{asymmetric (quasi-)distance}
      $\sfd:X\times X\to[0,\infty]$, satisfying}\\
    \sfd(x,x)=0,\quad \sfd(x_o,x)<\infty,\quad
    \sfd(x,z)\le \sfd(x,y)+\sfd(y,z)\quad\text{for
      every }x,y,z\in X.
  \end{gathered}
\end{equation}
%
We say that a subset $U\subset X$ is \emph{$\sfd$-bounded} if 
$\sup_{u\in U}\sfd(x_o,u)<\infty$.
We say that $\sfd$ separates the points of $U\subset X$ if 
\begin{equation}
  \label{eq:25}
  u,v\in U,\quad \sfd(u,v)=0\qquad\Rightarrow\qquad
  u=v.
\end{equation}
We will often deal with subsets of the product space $Y:=\R\times X$,
which will be endowed with the product topology $\sigmap$, the asymmetric distance
\begin{equation}
  \label{eq:72}
  \sfdp((s,x),(t,y)):=|t-s|+\sfd(x,y)\quad
  \text{and the distinguished point $y_o:=(0,x_o)$. }
\end{equation}
Notice that 
$\UU\subset \R\times X$ is separated by $\sfdp$ if
$\sfd$ separates the points of all its sections
$\UU(t):=\{u\in X:(t,u)\in \UU\}$, $t\in \R$. 

The relation between $\sigma$ and $\sfd$ will be
clarified by the following Lemma: we will typically choose $W$ as 
a sequentially compact subset of $X$ or
$\R\times X$.
\begin{lemma}
  \label{le:obvious}
  Let $(Z,\sigma_Z)$ and 
  $(W,\sigma_W)$ be Hausdorff topological spaces satisfying the
  first axiom of countability; we suppose that 
  $W$ is sequentially compact and it is endowed with a l.s.c.~asymmetric
  quasi-distance $\sfd_W$ as in \eqref{eq:16} and we fix 
  an accumulation point of $z_0\in Z$, with 
  neighborhood basis $\cN(z_0)$.
  \begin{enumerate}[i)]
  \item If $v:Z\to W$ satisfies 
    \begin{equation}
      \label{eq:73}
      \lim_{z\to z_0}\sfd_{W}(v(z),v(z_0))\land\sfd_{W}(v(z_0),v(z))=0,\quad
      \bigcap_{N\in \cN(z_0)}\overline{v(N)}\quad\text{is separated by }\sfd_W,
    \end{equation}
    then $\lim_{z\to z_0}v(z)=v(z_0).$
  \item If $v:Z\to W$ satisfies
    \begin{equation}
      \label{eq:74}
      \lim_{z,z'\to z_0}\sfd_{W}(v(z),v(z'))\land 
      \sfd_{W}(v(z'),v(z))=0,\quad
            \bigcap_{N\in \cN(z_0)}\overline{v(N)}\quad\text{is separated
              by }\sfd_W,
          \end{equation}
    then there exists the limit $\bar v:=\lim_{z\to z_0}v(z)$ and 
    $\lim_{z\to z_0}\sfd_{W}(v(z),\bar v)\land\sfd_{W}(\bar v,v(z))=0$.
  \end{enumerate}
\end{lemma}
\begin{proof}
  We will prove the claim ii), since the proof of point i) is
  completely analogous. We set
  $\sfd_{W,\land}(w,w'):=\sfd_W(w,w')\land \sfd_W(w',w)$.
  Since 
  $W$ is sequentially compact and
  $\sigma_Z,\sigma_W$ satisfy the first
  countability axiom, in order to prove the existence of the limit it is sufficient to show that
  whenever sequences $z_n',z_n''\to z_0$ with
  $z_n',z_n''\in Z$ and $v(z_n')\to v',\ v(z_n'')\to v''$ then
  $v'=v''$. 

  By the first of \eqref{eq:74} for every $\eps>0$ we find $\bar n\in
  \N$ such that 
  $\sfd_{W,\land}(v(z_n'),v(z_m''))\le \eps$ for $n,m\ge \bar n$;
  since $\sfd_{W,\land}$ is $\sigma_W$ lower
  semicontinuous, we obtain
  $$\sfd_{W,\land}(v(z_n'),v)\le \liminf_{m\to\infty}
  \sfd_{W,\land}(v(z_n'),v(z_m''))\le \eps\quad\text{for every }n\ge
  \bar n.
  $$
  Passing to the limit as $n\to\infty$ we obtain
  $\sfd_{W,\land}(v',v'')\le \eps$; since $\eps>0$ is arbitrary 
  we get $\sfd_{W,\land}(v',v'')=0$. 
  Since $v',v''$ belong to $\bigcap_{N\in \cN(z_0)}\overline{v(N)} $ which
  is separated by $\sfd_W$ we conclude that
  $v'=v''$. 
\end{proof}
\begin{remark}[The metric setting]
  \label{rem:metric}
  \upshape
  A typical situation occurs when $\sfd$ is a distance on $X$,
  i.e.~it is symmetric, finite, and separates the points of $X$, so
  that 
  $(X,\sfd)$ is a standard metric space,
  and $\sigma$ is the
  induced topology. In this case, which
  we will simply call \emph{the metric setting}, 
  part of the previous discussion and of the
  next developments 
  can be stated in a much simpler form. 
  A less restrictive notion  in the case of an asymmetric distance 
  (thus not inducing a topology) 
  is the left-continuity property, that we will introduce in 
  formula \eqref{eq:150} below.
\end{remark}
\paragraph{Pointwise total variation and $(\sigma,\sfd)$-regulated
  functions.}
\addcontentsline{toc}{subsubsection}{\it Pointwise total variation and $(\sigma,\sfd)$-regulated
  functions.}
Let $E\subset \R$ be an arbitrary subset; we will denote 
by $\Pf(E)$ the collection
of all the finite subsets of $E$ and we will set $E^-:=\inf E$,
$E^+:=\sup E$.
The \textit{pointwise} total variation $\var(u,E)$
of a function $u:E\to X$ 
is defined in the usual way by
 \begin{equation}
   \var(u,E):=\sup \left\{\sum_{j=1}^M
     \sfd\left(u(t_{j-1}),u(t_j)\right):
     t_0<t_1<\dots{}<t_M,\ \{t_i\}_{i=0}^M\in \Pf(E) \right\};\label{eq:1}
\end{equation}
we set $\var(u,\emptyset):=0$. 

If $\var(u,E)<\infty$ then $u$ belongs to the space
$\BV\sfd(E,X)$ of function with \emph{bounded variation} and
we can define the function
\begin{equation}
  \label{eq:26}
  V_u(t):=\var(u,E\cap [E^-,t]),\quad t\in [E^-,E^+],
\end{equation}
which is monotone non decreasing and satisfies
\begin{equation}
  \label{eq:27}
  \sfd(u(t_0),u(t_1))\le \var(u,[t_0,t_1])=V_u(t_1)-V_u(t_0)\quad
  \text{for every }t_0,t_1\in E,\ t_0\le t_1.
\end{equation}
%
%
When $\sfd$ is a distance and $(X,\sfd)$ is a complete metric space,
it is well known that every function $u$ with bounded variation is
\emph{regulated}, i.e.~it admits
left and right $\sfd$-limits at every time $t$ and the jump set coincides
with the jump set of $V_u$.
In our weaker framework, regulated functions should also take into account
the $\sigma$ topology (which could also be non-metrizable).
We propose the following definition.
\begin{definition}[$(\sigma,\sfd)$-regulated functions]
  \label{def:regulated}
  We say that $u:[a,b]\to X$ is $(\sigma,\sfd)$-regulated if 
  for every $t\in [a,b]$ there exist the left and
  right limits $u(t\lrl)$ 
  (here we adopt the convention $u(a\leftl ):=u(a)$ and $u(b\rightl ):=u(b)$)
  w.r.t.~the $\sigma$ topology, satisfying 
\begin{gather}
  \label{eq:2}
  \begin{aligned}
    u(t\leftl )=\lim_{s\uparrow t}u(s),\quad & \lim_{s\uparrow
      t}\sfd(u(s),u(t\leftl ))=0,
\\
    u(t\rightl
    )=\lim_{s\downarrow t}u(s),\quad &\lim_{s\downarrow t}\sfd(u(t\rightl
    ),u(s))=0,
    \end{aligned}
    \intertext{and}
    \label{eq:193}
    \sfd(u(t\leftl),u(t))=0\ \Rightarrow\ u(t\leftl)=u(t);\qquad
    \sfd(u(t),u(t\rightl))=0\ \Rightarrow\ u(t)=u(t\rightl).
\end{gather}
The \textit{pointwise} jump set $\Jump\pm u$ of a
$(\sigma,\sfd)$-regulated function $u$ is defined by
\begin{equation}
\label{eq:3}
\Jump -u := \left\{t\in[a,b]:u(t\leftl )\neq u(t)\right\},\quad 
\Jump+u:=\left\{t\in[a,b]: u(t)\neq
  u(t\rightl )\right\} ,\quad
\Jump{}u:=\Jump-u\cup\Jump+u.
\end{equation}
We will denote by $\BV{\sigma,\sfd}([a,b];X)$ the 
space of $(\sigma,\sfd)$-regulated functions
with finite $\sfd$-total variation.
In this case $\Jump\pm u$ coincide 
with the corresponding jump sets $\Jump\pm{V_u} $ of the real
monotone function
$V_u$.
In particular, $\Jump{}u=\Jump-u\cup\Jump+u$ is at most countable. 
\end{definition}
Notice that for a monotone function $V:[a,b]\to \R$ 
the jump set $\Jump{}V$ coincides with
$\{t\in [a,b]:V(t\leftl)\neq V(t\rightl)\}$.

As we already mentioned, in the metric setting of Remark \ref{rem:metric} when $(X,\sfd)$
is also complete, it is immediate to check that any function $u\in \BV\sfd([a,b];X)$
is $(\sigma,\sfd)$-regulated
and the values $u(t\pm)$ coincide with the usual left and right limits
of $u$. 
In more general situations, the following simple lemma, that
lies behind \cite[Assumption (A4), Theorems 3.2, 3.3]{Mainik-Mielke05} and 
\cite[Section 2.2]{Mielke-Rossi-Savare13},
provides a sufficient condition for a function $u\in
\BV\sfd(D;X)$, $D$ being a dense subset of $[a,b]$,
to admit a unique $(\sigma,\sfd)$
regulated extension to $[a,b]$;
when $D=[a,b]$ it still provides interesting $\sigma$-continuity
properties of $u$. 
\begin{lemma}
  \label{le:nondeg-regulated}
  Let $D$ be a dense subset of $[a,b]$, let
  $u$ be a curve in $\BV\sfd(D;X)$ with
  $\Jump{}{V_u}\subset D$,
  and let
  $\UU\subset [a,b]\times X$ a sequentially compact 
  set separated by $\sfdp$ such that $u(t)\in \UU(t)$ 
  for every $t\in D\setminus \Jump{}{V_u}$.
  
  If $\UU$ is sequentially compact and $\sfdp$ separates its points
  then $u$ admits a unique extension $\tilde u$ to a 
  function in $\BV{\sigma,\sfd}([a,b];X)$,
  with $\var(u,D)=\var(\tilde u,[a,b])$.
  In particular, when $D=[a,b]$ we get $u=\tilde u\in
  \BV{\sigma,\sfd}([a,b];X)$.  
  Finally, if $\sfd$ is a distance on $X$ then the values
  $u(t\pm)$ coincide with the left and right limits of $u$ with
  respect to $\sfd$ and with respect to $\sigma$. 
\end{lemma}
\begin{proof}
  %
  We fix $t\in (a,b]$ and apply
  Lemma \ref{le:obvious} ii), 
  with $Z:= D\cap [a,t)$, $z_0:=t$, observing that for $r,s\in Z$ with $r\le s$
  \begin{equation}
    \label{eq:61}
    \sfd(u(r),u(s))\le V_u(s)-V_u(r),\quad
    \lim_{s,r\up t}V_u(s)-V_u(r)=0.
  \end{equation}
  Once the existence of the limit has been established,
  the previous estimate and the lower semicontinuity of $\sfd$ show that $\lim_{s\up t}\sfd(u(s),u)=0$.
  The argument for the existence of the right limit is completely
  analogous.
  If $t\in [a,b]\setminus D$, we can extend $u$ by setting $\tilde u(t):=\lim_{s\to t,\
    s\in D}u(s)$, since $t\not\in \Jump{}{V_u}$. It is easy to check
  that $\tilde u\in \BV\sfd([a,b];X)$, $V_{\tilde u}=V_u$, 
  $\tilde u(t\lrl)=u(t\lrl)$,  
  and $\Jump{}{\tilde u}\subset D$.

  If $t\not \in \Jump-{\tilde u}$ then the above argument shows $\lim_{s\up
    t}\sfd(\tilde u(s),\tilde u(t))=0$ so that $\lim_{s\up t}V_{\tilde
    u}(s)=V_{\tilde u}(t)$ and $t$ is
  a left continuity point for $V_u$.
  On the other hand, if $t\in \Jump -{\tilde u}$, $\tilde
  u(t\leftl ),\tilde u(t)=u(t)\in \UU(t)$, the
  separation property
  yields $\sfd(\tilde u(t\leftl ),\tilde u(t))>0$ and 
  \begin{displaymath}
    \sfd(\tilde u(t\leftl ),\tilde u(t))\le \liminf_{s\up t}\sfd(\tilde
    u(s),\tilde u(t))\le
  \liminf_{s\up t}V_{\tilde u}(t)-V_{\tilde u}(s)=V_{\tilde
    u}(t)-V_{\tilde u}(t\leftl )
\end{displaymath}
so that $t\in \Jump-{V_u}$.

In order to show that $\var(u,D)=\var(\tilde u,[a,b])$,
we consider a finite subset $P=\{t_j\}_{j=0}^N$ of $[a,b]$ with $t_j\le
t_{j+1}$ and we fix a small $\eps>0$.
 Since $\Jump{}{\tilde u}=\Jump{}u\subset D$,
every $t_j\in P$ can be approximated
by points $t_j^\pm\in D$ such that $t_{j-1}< t_j^- \le t_j\le t_j^+<t_{j+1}$ 
and
$\sfd(\tilde u(t_j^-),\tilde u(t_j))+\sfd(\tilde u(t_j),\tilde
u(t_j^+))\le \eps/({N+1})$ (we just set $t_j^\pm:=t_j$ whenever $t_j\in P\cap
D$).
Thus we have a new partition (with possible repetitions) 
$t_0^- \le t_0^+ <t_1^- \le t_1^+ ,\cdots,<t_N^- \le t_N^+ $ in $D$ and
\begin{align*}
  \sum_{j=1}^N\sfd(\tilde u(t_{j-1}),\tilde u(t_j))
  &\le
       \sum_{j=0}^{N}\sfd(\tilde u(t_{j}^- ),\tilde u(t_{j}))+
  \sfd(\tilde u(t_{j}),\tilde u(t_{j}^+ )) +
       \sum_{j=1}^N
  \sfd(\tilde u(t_{j-1}^+),\tilde u(t_j^- )) 
  \\&\le     \eps + \sum_{j=0}^N\sfd(\tilde u(t_{j}^- ),\tilde u(t_{j}^+ )) +
      \sum_{j=1}^N\sfd(\tilde u(t_{j-1}^+ ),\tilde u(t_j^- ))   
    \le \eps +\var (u,D).
\end{align*}
Since $\eps>0$ and $P$ are arbitrary we conclude.

The last statement of the Lemma follows easily from
\eqref{eq:2}. 
\end{proof}

\paragraph{Augmented total variation associated with a transition cost.}
\addcontentsline{toc}{subsubsection}{\it Augmented total variation associated with a transition cost.}

In some cases, such as for Balanced Viscosity or Visco-Energetic solutions, 
we will need a modified notion of total variation, increased by a
further contribution along the jumps of
the function.

Such a contribution can be described by a function
$\sfe:[0,T]\times X\times X\to [0,+\infty]$ satisfying
\begin{equation}
  \label{eq:31}
  \Delta_\sfe(t,u_-,u_+)=\sfe(t,u_-,u_+)-\sfd(u_-,u_+)\ge
  0\quad\text{for every }t\in [0,T],\ u_\pm\in X.
\end{equation}
We will also use the notation $\Delta_\sfe(t,u_-,u,u_+)$ 
\begin{equation}
  \label{eq:192}
  \Delta_\sfe(t,u_-,u,u_+):=\Delta_\sfe(t,u_-,u)+\Delta_\sfe(t,u,u_+).
\end{equation}
\begin{definition}[Augmented total variation]
  \label{def:augmented}
  Let $\sfe,\Delta_\sfe$ be as in \eqref{eq:31}.
  For every $(\sigma,\sfd)$-regulated curve
  $u\in \BV{\sigma,\sfd}([0,T];X)$ and every subinterval
  $[a,b]\subset [0,T]$, the \emph{incremental jump variation} of $u$
  on $[a,b]$ induced by $\Delta_\sfe$
  is
  \begin{multline}
    \Jmp{\Delta_\sfe}(u,[a,b]):=\Delta_\sfe
    (a,u(a),u(a\rightl ))+ \Delta_\sfe
    (b,u(b\leftl ),u(b)) 
    \\ 
    +\sum_{t \in \Ju\cap (a,b)}
    \Delta_\sfe
    (t,u(t\leftl ),u(t),u(t\rightl)),
  \end{multline}
  and the corresponding \emph{augmented total variation} is
  \begin{equation} \label{eq:varP}
    \mVar{\sfd,\sfe}(u,[a,b]):=\var(u,[a,b])+
    \Jmp{\Delta_\sfe}(u,[a,b]).
  \end{equation}
\end{definition}
Notice that $\mVar{\sfd,\sfe}(u,[a,b])\ge\var(u,[a,b])$ and they
coincide when $\Ju=\emptyset$ or when $\sfe=\sfd$. As for the
$\sfd$-total variation, $\mVar{\sfd,\sfe}$ satisfies the additive
property
\begin{equation}
  \label{eq:21}
  \mVar{\sfd,\sfe}(u,[a,b])+\mVar{\sfd,\sfe}(u,[b,c])=\mVar{\sfd,\sfe}(u,[a,c])\quad
  \text{whenever }a\le b\le c.
\end{equation}

\subsection{The Energy functional}
\label{subsec:energy-functional}
In this section we briefly recall one of the possible settings for 
energetic solutions to a rate-independent system (R.I.S.) $(X,\cE,\sfd)$, following the 
approach of \cite{Mainik-Mielke05}.
Besides the asymmetric dissipation distance $\sfd$ we have introduced
in the previous section, variationally driven rate-independent
evolutions are characterized by a time-dependent \emph{energy
  functional}
$\cE:[0,T]\times X\rightarrow
    \mathbb{R}$.
%
%
A few basic properties will also involve
the perturbed functionals
\begin{equation}
  \label{eq:35}
  \cF(t,x):=\cE(t,x)+\sfd(x_o,x)+F_o,\quad
  \cF_0(x):=\cF(0,x),\quad x\in X
\end{equation}
and the collection of their sublevels $\big\{(t,x)\in [0,T]\times
X:\cF(t,x)\le C\big\}$ in $[0,T]\times X$;
here $F_o\in [0,\infty)$ is a suitable constant which will ensure
$\cF\ge0$, see \eqref{A.21}. 
We will always make the following standard assumptions
\cite{Mainik-Mielke05,Rossi-Mielke-Savare08,Mielke-Roubicek15},
where we also allow some flexibility in the choice of
the power $\cP$ (see 
\cite{Knees-Mielke-Zanini08,Knees-Zanini-Mielke10,MRS12,MRS13}.
\begin{mainassumption}{$\la$A$\ra$}
  \label{A}
  The R.I.S.~$(X,\cE,\sfd)$ satisfy 
  \begin{enumerate}[\upshape\bfseries $\la${A}.1$\ra$]
  \item \emph{\bfseries Lower semicontinuity and compactness.}
    \label{A1}
    $\cE$ is $\sigma$-l.s.c.~on all the sublevels of $\cF$, which are 
    $\sigma_\R$-sequentially compact in $[0,T]\times X$.
    %
  \item 
    \emph{\bfseries Power-control.}
    \label{A2}
    There exists a map $\cP:[0,T]\times X\to \R$ (the 
    ``time
    superdifferential'' of the energy)
    upper semicontinuous on the sublevels of $\cF$ 
    satisfying
    \begin{gather}
       \label{A.22}
        \liminf_{s\up t}\frac{\cE(t,x)-\cE(s,x)}{t-s}\ge \cP(t,x)\ge
        \limsup_{s\down t}\frac{\cE(s,x)-\cE(t,x)}{s-t}\\
         \label{A.21}
        |\cP(t,x)|\leq C_P \cF(t,x),
    \end{gather}
    for a constant $C_P> 0$ and for every $(t,x)\in[0,T]\times X$.
  \end{enumerate}
\end{mainassumption}
Notice that whenever $t\mapsto \cE(t,x)$ is differentiable at some
$t_0\in [0,T]$
\eqref{A.22} yields $\cP(t_0,x)=\partial_t \cE(t_0,x)$.
On the other hand, \eqref{A.22} shows that 
$t\rightarrow\cE(t,x)$ is upper semicontinuous (and thus continuous by
\mytag A1
and bounded) on
$[0,T]$. It follows that $t\mapsto \cF(t,x)$ is bounded so that
there exists a suitable constant $C$ providing 
$|\cP(t,x)|\le C$ for every $t\in [0,T]$. 

\eqref{A.22} then shows
that $t\mapsto \cE(t,x)$ is Lipschitz continuous 
differentiable a.e.;
by estimating $\cE(t_1,x)=\cE(t_0,x)+\int_{t_0}^{t_1}\cP(s,x)\,\d s$ with
\eqref{A.21} 
and applying Gronwall's lemma, we obtain
\begin{equation} \label{eq:gronwallestimate}
  \cF(t_1,x)
  \leq \cF(t_0,x)
  \exp(C_P|t_1-t_0|)\quad\text{for every }t_0,t_1\in [0,T].
\end{equation}
This estimate will be the basis for the a priori estimate of the
stored and the dissipated energies. 
\relax
In particular it implies that 
for every $t\in [0,T]$ and for every $y\in X$ the map
\begin{equation}
  \label{eq:86}
  x\mapsto \cE(t,x)+ \sfd(y,x) \quad \text{has bounded
      sequentially compact sublevels in $(X,\sigma)$}. 
\end{equation}
Moreover, it would not be difficult to check that \mytag A2 yields 
\begin{align*}
  \cP(t,x)&=\frac{\partial}{\partial t}^{\kern-2pt-}\cE(t,x)=
  \lim_{s\up t}\frac{\cE(t,x)-\cE(s,x)}{t-s}&&
  \text{for every }t\in [0,T],\ x\in X,\ \\
  \cP(t,x)&=\lim_{s\to t}\frac{\cE(t,x)-\cE(s,x)}{t-s}&&
            \text{for a.e.~$t\in [0,T]$, for every }x\in X. 
\end{align*}
\begin{remark}[Left continuity of $\sfd$ and upper semicontinuity of $\cP$]
  \label{rem:Pcont}
  \upshape
  The upper semicontinuity condition of $\cP$
  stated in \mytag A2
  could be relaxed if we know more properties on $\sfd$
  (and on the viscous correction $\delta$, see \ref{B}).
  An example is provided by this condition, that 
  can occasionally replace \mytag A2:
  \begin{enumerate}[\upshape\bfseries $\la${A}.2'$\ra$]
    \label{A2'}
  \item \em $\sfd$ is left-continuous on the sublevels of $\cF_0$, i.e.
      \begin{equation}
        \label{eq:150}
        \cF_0(x_n)\le C,\quad
        x_n\sigmato x\quad\Rightarrow\quad
        \sfd(x_n,v)\to \sfd(x,v)
      \end{equation}
      and the map $\cP:[0,T]\times X\to \R$ satisfies 
      \eqref{A.22}, \eqref{A.21} and the conditional 
      upper-semicontinuity
  \begin{equation}
    \label{eq:135}
    (t_n,x_n)\sigmato (t,x),\quad
    \cE(t_n,x_n)\to \cE(t,x)\quad\Rightarrow\quad
    \limsup_{n\up\infty}\cP(t_n,x_n)\le \cP(t,x).
  \end{equation}  
  \end{enumerate}  
\end{remark}
\relax
\begin{remark}[Extended-valued energies and distances]
\upshape
Our
setting is equivalent to considering an energy
functional 
$
\widetilde \cE:[0,T]\times \widetilde{X} \rightarrow \R\cup \{+\infty\}
$ 
possibly assuming the value $+\infty$,
since any reasonable formulation of \mytag A2 yields that 
the proper domain 
$ \dom(\widetilde\cE(t,\cdot)):=
\{u\in \widetilde{X}:
\widetilde\cE(t,u)<+\infty \}$
should be independent of time, thanks to \eqref{eq:gronwallestimate}.
Also the assumption $\sfd(x_o,u)<\infty$ is not restrictive.
%
In fact, it is sufficient to choose $x_o$ as the initial datum $\bar u$ of
the evolution problem and consider
the restriction of $\tilde \cE$ and $\tilde\sfd$ to the set 
$X:=\{v\in \dom(\widetilde\cE(0,\cdot)):\widetilde \sfd(x_o,v)<\infty\}$.
\end{remark}

\subsection{Energetic solutions to rate-independent problems} 
\label{subsec:energetic}
Hereafter 
we recall the notion of \textit{energetic solution} (see
\cite{Mielke-Theil-Levitas02}, \cite{Mielke-Theil04})  
to the 
\relax Rate-Independent System (R.I.S.) 
$(X,\cE,\sfd)$.
%
Let us first introduce the notion of 
the $\sfd$-stable set $\SS\subset [0,T]\times X$
associated with
$\cE$
\begin{equation}
  \label{eq:17}
  \SSd:=\Big\{(t,u)\in [0,T]\times X:\cE(t,u)\le \cE(t,v)+\sfd(u,v)\quad\text{for
    every }v\in X\Big\}
\end{equation}
with its time-dependent sections $\SSd(t):=\{u\in X:(t,u)\in \SSd\}$.
\begin{definition}[Energetic solutions] A curve $u\in
  \BV\sfd([0,T];X)$ is an energetic solution of the 
R.I.S.~$(X,\mathcal{E}, \sfd)$ if for all $t\in [0,T]$ it satisfies the global stability condition
\begin{equation} \label{enstability}
  u(t)\in \SSd(t), \quad\text{i.e.}\quad
  \mathcal{E}(t,u(t))\leq \mathcal{E}(t,v)+\sfd(u(t),v) \qquad
  \text{for every }v\in X,
  \tag{S$_\sfd$}
\end{equation}
and the energetic balance
\begin{equation} \label{ensolbalance}
\mathcal{E}(t,u(t))+\var(u,[0,t])=\mathcal{E}(0,u(0))+\int_0^t\cP(s,u(s))\rmd s. \tag {E$_\sfd$}
\end{equation}
\end{definition}
\relax 
The Existence of
energetic solutions in such a general framework 
is one of the main results of \textsc{Mainik-Mielke}:
it requires the \emph{closedness} of the stable set $\SSd$ and
a non-degeneracy of $\sfd$ in each sections $\SSd(t)$
\cite[Thm.~4.5]{Mainik-Mielke05}, two conditions which are always
satisfied in the simpler metric setting of Remark \ref{rem:metric}.
\begin{theorem}[Existence of Energetic solutions]
  Let us assume that 
  \begin{equation}
    \label{eq:18}
    \SSd\text{ is $\sigma$-closed in $[0,T]\times X$
      and separated by~$\sfd$}.
  \end{equation}
  Then for every $\bar u\in \SSd(0)$ there exists at least one
  energetic solution to the R.I.S.~$(X,\cE,\sfd)$
  satisfying $u(0)=\bar u$.
\end{theorem}
One of the main features of 
the definition above concerns the jump behaviour of a solution:
assuming \eqref{eq:18}, every energetic solution $u$ is
$(\sigma,\sfd)$-regulated and 
satisfies 
the following \textit{jump conditions} at every jump point $t\in \Ju$:
\begin{equation} \label{eq:Jener}
  \begin{aligned}
\mathcal{E}(t,u(t\leftl ))-\mathcal{E}(t,u(t))&=\sfd(u(t\leftl ),u(t)), \\
\mathcal{E}(t,u(t))-\mathcal{E}(t,u(t\rightl ))&=\sfd(u(t),u(t\rightl )), \\
\mathcal{E}(t,u(t\leftl ))-\mathcal{E}(t,u(t\rightl ))&=\sfd(u(t\leftl ),u(t\rightl )).
\end{aligned} \end{equation}
\paragraph{The time incremental minimization scheme.} 
\addcontentsline{toc}{subsubsection}{\it The time incremental minimization scheme.}
\relax 
\label{page:incremental}
The most powerful method to construct energetic solutions to the
R.I.S.~$(X,\cE,\sfd)$ and to
prove their existence 
is provided by the 
time incremental minimization scheme.

We consider ordered
finite partitions $\tau\subset [0,T]$ 
whose points will be denoted by $t^n_\tau$ for integers $n$  between
$0$ and $N=N(\tau)$,
$0=t^0_\tau<t^1_\tau<\dots{}<t^{N-1}_\tau<t^N_\tau=T$, 
and we set $|\tau|:=\max_n
t^n_\tau-t^{n-1}_\tau$.
In
order to find good approximations $U^n_\tau$ of $u(t^n_\tau)$ we 
choose an initial value $U^0_\tau\approx u_0$ and solve the time incremental minimization scheme
\begin{equation}
\label{eq:IP0}
U^n_\tau\in \argmin_{U\in X} \Big\{\sfd(U^{n-1}_\tau,U)+\mathcal{E}(t^n_\tau,U)\Big\}. \tag{$\mathrm{IM_\sfd}$}
\end{equation}
Setting 
\begin{equation} \label{costinterpolant}
\overline{U}\kern-1pt_\tau(t):=U^n_\tau \quad \text{if } t\in(t^{n-1}_\tau,t^n_\tau],
\end{equation}
it is possible to find a sequence of partitions $\tau_k$ with $|\tau_k|\downarrow 0$ such that
\[
\exists \lim_{k\rightarrow + \infty} \overline{U}\kern-1pt_{\tau_k}(t):=u(t) \quad \text{for every } t\in[0,T]
\]
and $u$ is an energetic solution starting from $u_0$.

\newcommand{\taue}{\tau}

\subsection{Viscosity approximation and Balanced Viscosity (BV)
  solutions}
\label{subsec:BVsolutions}
A different approach to solve \textit{rate-independent} problems is to
use a viscous approximation of the dissipation distance. 
\relax 
For the sake of simplicity, we present this approach in the metric
setting of Remark \ref{rem:metric} with uniform partitions (i.e.~$t^n_\tau=n|\tau|$),
starting from the viscous regularization of the incremental
minimization scheme \eqref{eq:IP0} by a quadratic 
perturbation generated by the same distance $\sfd$, i.e.~a quadratic term
with coefficient $\mu:(0,\infty)\to (0,\infty) $ 
\begin{equation} \label{eq:delta}
\delta_\tau(u,v):=\frac 12 \mu(|\tau|) \sfd^2(u,v),
\quad u,v\in X;\quad
\lim_{r\down0}\mu(r)=+\infty,\quad
\lim_{r\down0} r\mu(r)=0;
\end{equation}
recall that \eqref{eq:delta} corresponds to \eqref{eq:167e} with 
$\mu(\tau)=\eps(\tau)/\tau$ independent of $n$. 

The viscous incremental problem is therefore to find $U^1_{\taue},\dots{},U^N_{\taue}$ such that
\begin{equation} \label{viscincproblem}
U^n_{\taue}\in\argmin_{U\in
  X}\left\{\vphantom{\Big|}\sfd(U^{n-1}_{\taue},U)+\delta_\tau(U^{n-1}_{\taue},U)+\mathcal{E}(t^n,U)\right\}. \tag{$\mathrm{IM_{\sfd,\delta_\tau}}$} 
\end{equation}
 Setting as in \eqref{costinterpolant}
$
\overline{U}\kern-1pt_{\taue}(t):=U^n_{\taue} \quad \text{if } t\in(t^{n-1}_\tau,t^n_\tau],
$
we can study the limit of the discrete solutions when $|\tau|\downarrow 0$.

\relax 
The scaling of the factor $\mu(\tau)$ in \eqref{eq:delta} can be justified by
observing that when $X=\R^d$, $\sfd(u,v):=|u-v|$ and $\cE$ is a
$\rmC^1$ function,
\eqref{viscincproblem} naturally arises as the implicit discretization of the
differential inclusion
\begin{equation}
  \label{eq:5}
  \mathop{\rm sign}(u'(t))+\eps u'(t)+\rmD_u \cE(t,u(t))\ni
  0,\quad
   \mathop{\rm sign}(v):=
   \begin{cases}
     v/|v|&\text{if }v\neq0,\\
     \{w\in \R^d:|w|\le 1\}&\text{if }v=0,
   \end{cases}
\end{equation}
with the choice $\mu(|\tau|):=\eps/|\tau|$. Therefore, 
if $\lim_{|\tau|\down0}\mu(|\tau|)=+\infty$ one can
heuristically expect 
that
the limits of discrete solutions $\overline U\kern-1pt_\tau$ to the 
incremental minimization problems \ref{viscincproblem} 
coincide 
with the limit trajectories of \eqref{eq:5} as $\eps\downarrow0$. 

Under quite general assumptions on $\cE$ it is possible to prove that
limit solutions $u$ of suitable subsequences of
$\overline{U}\kern-1pt_{\taue}$ satisfy 
a local stability condition, which replaces the global one
\eqref{enstability}, and a modified energy balance
defined in terms of an augmented total variation 
$\mVar{\sfd,\sfv}(u,[0,T])\ge 
\var(u,[0,T])$ as in Definition \ref{def:augmented}.
\relax
Both involve the \emph{metric slope} of $\cE$, defined as
\begin{equation}
  \label{eq:metric-slope}
  |\rmD\cE|(t,u):=\limsup_{v\to u}\frac{(\cE(t,u)-\cE(t,v))_+}{\sfd(u,v)}.
\end{equation}
%
$\mVar{\sfd,\sfv}$ is associated with the 
\textit{minimal
  transition cost} between $u_0$ and $u_1\in X$ at the time $t$:
\begin{align} \begin{split}
\sfv
(t,u_0,u_1):=\inf 
\bigg\{&\int_{r_0}^{r_1}|\dot{\theta}|(r)
\left(|\mathrm{D}\mathcal{E}|(t, \theta(r))\vee
  1\right)\,\mathrm{d}r: \\ & \theta\in
\mathrm{AC}([r_0,r_1];X,\sfd),\mbox{ }
\theta(r_0)=u_0,\mbox{ }\theta(r_1)=u_1 \bigg\},
\end{split} \end{align}
where $|\dot\theta|$ denotes the \emph{metric derivative} of $\theta$, 
see \cite[Sect.~1]{Ambrosio-Gigli-Savare08} and
\cite[Sect.~2.1]{Rossi-Mielke-Savare08} 
in the asymmetric case;
clearly 
$\sfv
(t,u_0,u_1)\ge\sfd(u_0,u_1)$.
Based on \eqref{eq:varP}, we can now specify the concept of \textit{Balanced Viscosity $\mathrm{(BV)}$ solution} to the rate-independent system $(X,\cE,\sfd)$.
\begin{definition}[Balanced Viscosity (BV) solutions] A curve $u\in \BV\sfd([0,T];X)$
  is a 
  {\upshape BV} solution of the rate-independent system
  $(X,\mathcal{E},\sfd)$
  with the viscous dissipation \eqref{eq:delta} if it satisfies the local stability 
\begin{equation} \label{sloc}
|\mathrm{D}\mathcal{E}|(t,u(t))\leq 1 \quad \text{for every}\quad
t\in[0,T]\setminus \Ju, \tag{$\mathrm{S_{\sfd,loc}}$}
\end{equation}
and the energy balance
\begin{equation}
\label{eq:EBc}
\mathcal{E}(t,u(t))+\mVar{\sfd,\sfv}(u,[0,t])=\mathcal{E}(0,u(0))+\int_0^t\cP(s,u(s))\rmd
s \quad \text{for all }t\in[0,T]. \tag{$\mathrm{E}_{\sfd,\sfv} 
$}
\end{equation}  
\end{definition}
The 
viscous total variation induced by $\sfv$ in 
the energy balance \eqref{eq:EBc}
(instead of the canonical one induced by the distance $\sfd$)
compensates for the lack of information in the local stability
condition \eqref{sloc}. In particular,
it is possible to prove that a curve $u\in \BV\sfd([0,T];X)$ is a
BV solution of the rate-independent system $(X,\cE,\sfd)$
if and only if it satisfies the local stability condition
$\eqref{sloc}$, the localized energy dissipation inequality
\[
\cE(t,u(t))+\var(u,[s,t])\leq \cE(s,u(s))+\int_s^t\cP(r,u(r))\,\d
r\quad
\text{for every $0\le s\le t\le T$}
\] 
and the following jump conditions at each point $t\in \Ju$ (see \cite[Theorem 3.13]{Mielke-Rossi-Savare13}):
\begin{equation} \label{eq:Jbv}
\begin{aligned}
\mathcal{E}(t,u(t\leftl ))-\mathcal{E}(t,u(t))&=
\sfv(t,u(t\leftl ),u(t)), \\
\mathcal{E}(t,u(t))-\mathcal{E}(t,u(t\rightl ))&=
\sfv(t,u(t),u(t\rightl )), \\
\mathcal{E}(t,u(t\leftl ))-\mathcal{E}(t,u(t\rightl ))&=
\sfv(t,u(t\leftl ),u(t\rightl )).
\end{aligned} 
\end{equation}
%
%

\section{Visco-Energetic (VE) solutions}
\label{sec:VE}
As we have seen in section \ref{sec:prelresults}, 
the choice $\mu(\tau)\equiv0$ in the incremental minimization problem
\eqref{viscincproblem}
\relax corresponds to \eqref{eq:IP0} and
leads
to the notion of \textit{Energetic solutions}, while the case
\relax when $\mu(\tau)\up+\infty$ as $\tau\down0$ 
corresponds to  \textit{Balanced Viscosity solutions}.
In the present paper, we want to study the asymptotic behaviour of the
incremental minimization scheme in the case 
when 
$\mu(\tau)\equiv \mu$ is a constant, 
and to find an appropriate variational characterization for the
corresponding limit trajectories. 
Arguing as in \cite{Rossi-Savare17-preprint} 
it would not be too difficult to consider also the case when
$\mu(\tau)\to\mu\in (0,\infty)$, obtaining the same class of limit solutions
as in the constant case $\mu(\tau)\equiv \mu$; however, instead of focusing on a
quadratic viscosity with a $\tau$-dependent coefficient, we 
prefer to cover a more general class of viscous corrections.

\subsection{Viscous correction of the incremental minimization scheme}
\label{subsec:viscous-correction}

In order to cover a wide spectrum of possible applications with
the greatest flexibility, we are considering here general viscous
corrections modeled by a lower semicontinuous map
\begin{equation}
\text{$\delta: X\times X\rightarrow [0,+\infty]$\quad
  with $\delta(x,x)=0$ for every $x\in X$},\label{eq:152}
\end{equation}
and the corresponding modified dissipation
\begin{equation}
\mathsf{D}(x,y):=\sfd(x,y)+\delta(x,y)\qquad \text{for all }x,y\in
X.
\label{eq:6}
\end{equation}
As in the previous section, $\sfd$ and $\cE$ will be 
a dissipation distance and a time-dependent energy functional
satisfying Assumptions \mytag A{} in the metric-topological setting
introduced in Section \ref{subsec:metric-topological}.
Our starting point 
is the following modified variational scheme.
\begin{definition}[The viscous incremental minimization scheme.]
\label{def:vims}
Starting from $U^0_{\tau}\in X$, find recursively $U^1_\tau,
  \dots{},U^N_\tau$ such that
$U^n_\tau$ minimizes
\begin{equation} \label{ims}
U\mapsto 
\sfD(U^{n-1}_\tau,U)+\mathcal{E}(t_n,U)
=\sfd(U^{n-1}_\tau,U)+\delta(U^{n-1}_\tau,U)+\mathcal{E}(t^n,U)
. \tag{$\mathrm{IM_{\sfd,\delta}}$}
\end{equation}
\end{definition}
Since $\sfd$ and $\delta$ are lower semicontinuous, 
the existence of a minimizer for the problem \eqref{ims} follows from
condition \mytag A1.
%
%
\relax
Of course, not every continuous function $\delta$ 
will provide an admissible viscous correction;
the trivial example $\delta=\sfd$ (which doubles the dissipation
distance)
shows that we should impose some 
sufficiently strong vanishing condition of $\delta(x,y)$ 
in the neighborhoods of points where $\sfd(x,y)=0$ in $X\times X$. 

A quite general admissibility criterion is related to the notion of
\emph{global $\sfD$-stability}, which can be easily imagined from the corresponding
property \eqref{enstability} of the energetic case.
We will also introduce the weaker notion of quasi-stability, which will
turn out to be very useful later on.
\begin{definition}[Quasi $\sfD$-stability and $\sfD$-stable set]
  \label{def:stable-set}
  Let $Q\ge0$; we say that $(t,x)\in [0,T]\times X$ is a $(\sfD,Q)$-quasi-stable point
  if it satisfies
  \begin{equation} \label{eq:quasistable}
    \mathcal{E}(t,x)\leq \mathcal{E}(t,y)+\sfD(x,y)+Q\qquad
    \text{for every }y\in X.
  \end{equation}
  In the case $Q=0$, i.e.~when
  \begin{equation} \label{stablepoints}
    \mathcal{E}(t,x)\leq \mathcal{E}(t,y)+\sfD(x,y)\qquad
    \text{for every }y\in X,
  \end{equation}
  we say that $(t,x)$ is $\sfD$-stable. We call $\SSD$ the stable set
  (i.e.~the collection of all the $\sfD$-stable points) and 
  $\SSD(t)$ its section at time $t$.
\end{definition}
As in the case of energetic solutions, we expect that the
$\sfD$-stability condition 
will play a
crucial role. 
\relax 
A first important point concerns the admissibility criterion for the
viscous correction $\delta$: in addition to 
basic compatibility properties between $\delta$ and $\sfd$,
we will essentially require that $\sfD$-stable points
satisfy a sort of local $\sfd$-stability. To better understand the next 
condition,
let us first notice that \eqref{stablepoints} can be equivalently
formulated as
\begin{equation}
  \label{eq:7}
  \sup_{y\neq x} \frac{\mathcal{E}(t,x)-
    \mathcal{E}(t,y)}{\sfD(x,y)}\le 1.
\end{equation}
\begin{mainassumption}{$\la$B$\ra$}[Admissible viscous corrections]
  \label{B}
  An admissible viscous correction $\delta:X\times X\to [0,+\infty]$ for the
  R.I.S.~$(X,\cE,\sfd)$ satisfies the following conditions:
  \begin{enumerate}[\upshape $\la${B}.1$\ra$] 
    \item 
      \label{B1}
      \emph{\bfseries $\sfd$-compatibility.}
      For every $x,y,z\in X$
      \begin{equation}
    \label{D.1}
    \sfd(x,y)=0\quad\Rightarrow\quad \delta(z,y)\le \delta(z,x)\text{ \
      and \ }
    \delta(x,z)\le \delta(y,z).
  \end{equation}
  \item 
    \label{B2}
    \emph{\bfseries Left $\sfd$-continuity.}
    For every sequence
    $x_n$ and every $x\in X$ 
  we have
  \begin{equation}
    \label{D.2}
    \sup_n\cF_0(x_n)<\infty,\quad
    x_n\sigmato x,\quad \sfd(x_n,x)\to 0\quad\Rightarrow\quad
    \lim_{n\to\infty}\delta(x_n,x)=0.
  \end{equation}
\item 
  \label{B3}
  \emph{\bfseries $\sfD$-stability yields local $\sfd$-stability.}
  For every $(t,x)\in \SSD$, $M>1$ there exists $\eta>0$ and a
  neighborhood $U$ of $x$ in $X$ such that 
\begin{equation} \label{D.3}
  \cE(s,y)\le \cE(s,x)+M\sfd(v,x)\quad
  \text{for every }(s,y)\in \SSD,\ 
  s\in (t-\eta,t],\ y\in U,\ \sfd(y,x)\le \eta.
\end{equation} 
Equivalently, 
\begin{equation}
  \label{D.3'}
  \limsup_{(s,y)\to(t,x),\ \sfd(y,x)\to0
    \atop(s,y)\in \SSD, \
    s\le t}
  \frac{\cE(s,y)-\cE(s,x)}{\sfd(y,x)}\le 1.
\end{equation}
\end{enumerate}
\end{mainassumption}
\mytag B1
 is a minimal compatibility condition between $\sfd$ and
$\delta$. Notice that \mytag B1 is trivially satisfied by any
monotonically increasing function of $\sfd$ or if $\sfd$
separates the points of $X$, e.g.~in the simpler metric
setting of Remark \ref{rem:metric}. 

Let us now see an important example of admissible viscous
corrections 
in which assumption \ref{B} is satisfied.
A further example will be discussed in Section \ref{subsec:genconvexity}.
\begin{example} 
\label{ex:1}
\upshape
If $\delta: X\times X\to [0,+\infty)$ satisfies 
\begin{equation}
  \relax \lim_{y\rightarrow x \atop
    \sfd(y,x)\to0 
  }\frac{\delta(y,x)}{\sfd(y,x)}=0 \quad \text{for every $x\in
    \SSD(t),\quad t\in [0,T]$},\label{eq:12}
\end{equation}
then it is immediate to check that $\sfD$ satisfies 
\mytag B3.
In particular, 
\relax any function of the form
\begin{equation}
  \delta(x,y)= h(\sfd(x,y))\quad\text{for a nondecreasing
    $h\in \rmC([0,\infty))$ with \quad
$\lim_{r\down0}h(r)/r=0$}
\label{eq:10}
\end{equation}
provides an admissible correction satisfying \mytag B{}.
A typical choice is the quadratic correction
\begin{equation}
  \label{eq:9}
  \delta(x,y):=\frac{\mu}{2}\sfd^2(x,y),\quad \mu>0.
\end{equation}
\end{example}

\subsection{Transition costs and augmented 
  total variation.}
\label{subsec:residual}
\relax 
Let us first notice that the quasi-stability condition
\eqref{eq:quasistable}
(and therefore the stability condition \eqref{stablepoints})
can be equivalently characterized through 
a sort of \emph{residual function} that we introduce in the definition below.
\begin{definition} \label{def:slope} For every $t\in[0,T]$ and $x\in
  X$ 
  the residual stability function is defined by
  \begin{align}
    \label{eq:108}
    \Gs(t,x):&=\sup_{y\in X}
               \{\mathcal{E}(t,x)-\mathcal{E}(t,y)-\sfD(x,y)\} \\
    \label{eq:109}
&=\mathcal{E}(t,x)-\inf_{y\in X}\{\mathcal{E}(t,y)+\sfD(x,y)\};
\end{align}
\end{definition}
$\Gs(t,x)$ provides the minimal constant $Q\ge0$ such that $(t,x)$
is $(\sfD,Q)$-quasi-stable:
\begin{equation}
  \label{eq:131}
  \Gs(t,x):=\min\Big\{Q\ge0: 
  \cE(t,x)\le \cE(t,y)+\sfD(x,y)+Q\quad\text{for every }y\in X\Big\}.
\end{equation}
By choosing $y:=x$ in \eqref{eq:108}
 we can immediately check that the residual function $\Gs$ is non-negative, i.e.
\[
\Gs(t,x)\geq 0 \quad \text{for all $t\in[0,T]$, $x\in X$}.
\]
\relax $\Gs$ provides a measure of the 
failure of the stability condition \eqref{stablepoints}, since 
for every $x\in X$, $t\in[0,T]$ we get 
\begin{equation}
  \label{eq:62}
  \cE(t,x)\le \cE(t,y)+\sfD(x,y)+\Gs(t,x)
\end{equation}
and
\begin{equation}
\Gs(t,x)=0\quad \Longleftrightarrow \quad \relax x\in
\SSD(t).\label{eq:63}
\end{equation}
Notice that when 
$\sfD$ is $\sigma$-continuous then
\begin{gather}
\text{$\Gs$ is $\sigma_\R$-lower semicontinuous}
\intertext{or, equivalently,}
\text{for every $Q\ge0$ the $(\sfD,Q)$-quasi-stable set is $\sigma$-closed.}\label{eq:133}
\end{gather}
We will see that this property will
play a crucial role in
our general setting and corresponds to the closedness property
of the
stable set \eqref{eq:18} in the energetic framework.

As in the case of Balanced Viscosity solutions, 
we expect that the jumps of a limit trajectory of the viscous
incremental minimization scheme \ref{def:vims} 
can be characterized by a class of curves minimizing a suitable
transition cost. 

The main novelty here is represented by the fact that such curves are
parametrized by continuous maps $\vartheta:E\to X$, defined on a
compact subset $E$ of $\R$, which in general may have a more
complicated structure than an interval.
We will also require that $\vartheta$ satisfies a natural continuity
condition with respect to $\sfd$
\begin{equation}
  \label{eq:46}
  \forall\,\eps>0\ \exists\,\eta>0:\quad
  \sfd(\vartheta(s_0),\vartheta(s_1))\le \eps\quad
  \text{for every }s_0,s_1\in E,\ s_0\le s_1\le s_0+\eta.
\end{equation}
The class of curves satisfying
\eqref{eq:46}
will be denoted by $\rmC_\sfd(E,X)$ and we will set
$\rmC_{\sigma,\sfd}(E,X):=\rmC(E,X)\cap \rmC_\sfd(E,X).$
In order to get a precise description of this (pseudo-) total
variation, we have to introduce a dissipation cost. 

Hereafter for every subset $E\subset \R$ we will call $E^- :=\inf E$,
$E^+ :=\sup E$; whenever $E$ is compact, we will denote by $\Holes(E)$ the
(at most) countable collection of the connected components of the open
set $[E^- ,E^+ ]\setminus E$: each element of $\Holes(E)$ (the ``holes''
of $E$) is therefore an
open interval of $\R$. We also denote by $\Pf(E)$ the collection
of all the finite subsets of $E$.
\begin{definition}[Transition cost]
  \label{def:transition-cost}
  Let $E\subset \R$ compact and $\vartheta\in \rmC_{\sigma,\sfd}(E;X)$. For every
  $t\in[0,T]$ we define the \emph{transition cost function} $\Cf(t,\vartheta,E)$ by
\begin{equation}
  \label{eq:33}
  \Cf(t,\vartheta,E):=\var(\vartheta,E)+\Cd(\vartheta,E)+\sum_{s\in E\setminus \{E^+ \}}\Gs(t,\vartheta(s))
\end{equation}
where the first term is defined as the usual total variation
\eqref{eq:1}, 
the second one 
is
\[
\Cd(\vartheta,E):=\sum_{I\in \Holes(E)}\delta(\vartheta(I^- ),\vartheta(I^+ )),
\]
and
the third term is
\[
\sum_{s\in
  E\setminus\{E^+ \}}\Gs(t,\vartheta(s)):=\sup\left\{\sum_{s\in
    P}\Gs(t,\vartheta(s)): P\in \Pf(E\setminus \{E^+ \})\right\},
\]
where the sum is defined as $0$ if $E\setminus \{E^+ \}=\emptyset$.
\end{definition}
We adopt the convention $\Cf(t,\vartheta,\emptyset):=0$. 
It is not difficult to check that the transition cost
$\Cf(t,\vartheta,E)$ is additive with respect to $E$:
\begin{equation}
  \label{eq:195}
  \Cf(t,\vartheta,E\cap[a,c])=
  \Cf(t,\vartheta,E\cap[a,b])+
  \Cf(t,\vartheta,E\cap[b,c])\quad \text{for every }a<b<c.
\end{equation}
It will be proved (see
Proposition \ref{prop:crineqonjumps})
that for every $t\in[0,T]$ and for every $\vartheta\in \rmC(E;X)$ 
\begin{equation}\label{eq:crinqualitytheta}
  \cE(t,\vartheta(E^+ ))+\Cf(t,\vartheta,E)\geq \cE(t,\vartheta(E^- )).
\end{equation}
The dissipation cost $\Fd(t,u_0,u_1)$ induced by the function $\Cf$ is
defined 
by minimizing $\Cf(t,\vartheta,E)$ among all the transitions $\vartheta$
connecting $u_0$ to $u_1$:
\begin{definition}[Jump dissipation cost and augmented total variation] \label{dissipationcost} 
Let $t\in[0,T]$ be fixed and let us consider $u_0, u_1\in X$. We set 
\begin{equation} \label{eq:dissipationcost}
\Fd(t, u_0,u_1):=\inf\left\{\Cf(t, \vartheta,E): E\Subset \R,\ 
\vartheta\in \rmC_{\sigma,\sfd}(E; X),\  \vartheta(E^- )=u_0,\ \vartheta(E^+ )=u_1\right\},
\end{equation}
with the incremental dissipation cost $\Delta_\sfc(t,
u_0,u_1):=\Fd(t,u_0,u_1)-\sfd(u_0,u_1)$.
The corresponding augmented total variation $\mVar{\sfd,\sfc}$ is then
defined according to Definition \ref{def:augmented}.
\end{definition}
Since $\Gs$ and $\Cd$ are positive, as in the case of BV solutions, it is immediate to check that 
\[
\Fd(t,u_0,u_1)\geq \sfd(u_0,u_1)\quad\text{for every $u_0, u_1\in X$}.
\]
Moreover, from \eqref{eq:crinqualitytheta}, it easily follows that
\begin{equation}
\Fd(t,u_0,u_1)\geq \cE(t,u_0)-\cE(t,u_1).
\end{equation}
As in the case of $\varP$ for Balanced Viscosity solutions, $\varC$ is
not a \textit{standard} total variation functional: for instance, it
is not induced by any distance on $X$. Nevertheless, $\varC$ enjoys
the nice additivity property \eqref{eq:21}.

\subsection{Visco-Energetic solutions}
\label{subsec:VE}
We can now give our precise definition of \textit{Visco-Energetic
  solution} of the rate-independent system
$(X,\mathcal{E},\sfd,\delta)$.
We will always assume that the energy functional satisfy the standard assumptions
\mytag A{}
and $\delta$ is an admissible viscous correction 
(i.e.~\mytag B{} hold).
\begin{definition}[Visco-Energetic (VE) solutions] We say that
  a $(\sigma,\sfd)$-regulated 
  curve
  $u:[0,T]\to X$ 
  is a 
\emph{Visco-Energetic (VE) solution} of the rate-independent system $(X,\mathcal{E},\sfd,\delta)$ if it satisfies the stability condition
\begin{equation}\label{stability}
  u(t)\in \SSD(t)\quad\text{for every }t\in [0,T]\setminus \Ju,
  \tag{S$_\sfD$}
\end{equation}
and the energetic balance
\begin{equation} \label{energybalance}
\mathcal{E}(t,u(t))+\varC(u,[0,t])=\mathcal{E}(0,u(0))+\int_0^t\cP(s,u(s))\rmd s \tag{$\mathrm{E_{\sfd,\sfc}}$}
\end{equation}
for every $t\in[0,T]$, where $\sfc$ is the jump dissipation cost \eqref{eq:dissipationcost}.
\end{definition}

As in the case of energetic and BV solutions, it is not difficult to see that the energy balance \eqref{energybalance} holds on any subinterval of $[t_0,t_1]$ of $[0,T]$:
\[
\mathcal{E}(t_1,u(t_1))+\varC(u,[t_0,t_1])=\mathcal{E}(t_0,u(t_0))+\int_{t_0}^{t_1}\cP(s,u(s))\rmd s.
\]
Indeed, this follows from the additivity property \eqref{eq:21} for
the augmented total variation $\mVar{\sfd,\sfc}$. Moreover, if a
curve  $u\in \BV{\sigma,\sfd}([0,T];X)$ satisfies the stability condition \eqref{stability}, then a \textit{chain-rule} inequality holds:
\begin{equation} \label{eq:crinequality}
\mathcal{E}(t_1,u(t_1))+
\varC(u,[t_0,t_1])\geq\mathcal{E}(t_0,u(t_0))+\int_{t_0}^{t_1}\cP(s,u(s))\sfd
s. 
\end{equation}

As a direct consequence, we have a characterization of VE solutions in
terms of a single, global in time, energy-dissipation inequality or
of a $\sfd$-energy-dissipation inequality combined with
a precise description of the jump behaviour.
The proof can be easily adapted from \cite[Prop.~4.2, Thm.~4.3]{Mielke-Rossi-Savare12}.
\begin{proposition}[Sufficient criteria for VE solutions]
  \label{prop:leqinequality} Let 
  $u\in \BV{\sigma,\sfd}([0,T];X)$ be a curve satisfying the
  stability condition \eqref{stability}.
  Then $u$ is a \VE\ solution of the
rate-independent system 
$(X,\mathcal{E},\sfd,\delta)$ if and only if it satisfies 
one of the following equivalent characterizations:
%
\begin{enumerate}[i)]
\item 
$u$ satisfies the $(\sfd,\sfc)$-energy-dissipation inequality
\begin{equation} \label{leqinequality}
    \mathcal{E}(T,u(T))+\varC(u,[0,T])\leq\mathcal{E}(0,u(0))+\int_0^T\cP(s,u(s))\rmd
    s. 
  \end{equation}
\item $u$ satisfies the $\sfd$-energy-dissipation inequality
  \begin{equation}
    \cE(t,u(t))+\var(u,[s,t])\leq \cE(s,u(s))+\int_s^t\cP(r,u(r))\rmd
    r\quad\text{for all $s\le t\in[0,T]$}\label{eq:110}
\end{equation}
and the following jump conditions at each point $t\in \Ju$
\begin{align}\label{Jve}
\begin{split}
\mathcal{E}(t,u(t\leftl ))-\mathcal{E}(t,u(t))=\Fd(t,u(t\leftl ),u(t)), \\
\mathcal{E}(t,u(t))-\mathcal{E}(t,u(t\rightl ))=\Fd(t,u(t),u(t\rightl )), \\
\mathcal{E}(t,u(t\leftl ))- \mathcal{E}(t,u(t\rightl ))=\Fd(t,u(t\leftl ),u(t\rightl )).
\end{split} \end{align}
\end{enumerate}
\end{proposition}

\relax
\paragraph{Existence of VE solutions.}
\addcontentsline{toc}{subsubsection}{\it Existence of {\rm VE} solutions}
As in the energetic  and BV
cases, existence of Visco-Energetic solutions can be obtained by
proving the convergence of discrete solutions to the incremental
minimization scheme \ref{def:vims}. 
Besides the canonical assumptions \mytag A{} and \mytag B{}
we will further suppose that the following properties hold.
\begin{mainassumption}{$\la$C$\ra$}[Closure and separation properties
  for the stable set]\
  \label{C}
  \begin{enumerate}[\em\bfseries $\la${C}.1$\ra$]
  \item 
    \label{C1}
    For every $Q\ge0$
    {the $(\sfD,Q)$-quasistable sets
        \eqref{eq:quasistable} have $\sigma$-closed intersections with the
        sublevels of $\cF$.}
   \item \label{C2}
     The $\sfD$-stable set {is separated by~$\sfd$.}
  \end{enumerate}
\end{mainassumption}
In the viscous setting these assumptions correspond to \eqref{eq:18}
in the energetic one.
\mytag C1
is always satisfied in the case 
when \mytag A{2'} holds, in particular in the simpler metric case considered in Remark
\ref{rem:metric}.
Notice that \mytag C1 is equivalent to
assume that the residual stability
function
of Definition \ref{def:slope}
\begin{equation}
  \label{eq:134}
  \Gs\text{ is $\sigma$-l.s.c.~on the sublevels of }\cF.
  \tag{C.1'}
\end{equation}
Our main result is stated in the following theorem.
\begin{theorem}[Convergence of the \eqref{ims} scheme and
  existence of VE solutions] \ \label{thm:existence} 
  \\Let us suppose that \emph{Assumptions \mytag A{}, \mytag B{}, \mytag C{}}
  hold. Let $u_0\in X$ be fixed and let $\overline{U}\kern-1pt_\tau$ be the family of piecewise left-continuous constant interpolants of discrete solutions $U^n_\tau$ of  \eqref{ims}, with 
  \begin{equation}
\sfd(x_o,U^0_\tau)\le C,\quad
U^0_\tau\sigmato u_0, \qquad \cE(0,U^0_\tau)\rightarrow
\cE(0,u_0)\quad\text{as }|\tau|\down 0.\label{eq:141}
\end{equation}
Then for all sequence of partitions $k\mapsto \tau(k)$ with
$|\tau(k)|\down0$ there exist a (not relabeled)
subsequence and a 
limit curve $u\in \BV{\sigma,\sfd}([0,T];X)$ such that
\[
\overline{U}\kern-1pt_{\tau(k)}(t)\sigmato u(t),\quad
\cE(t,\overline U_{\tau(k)}(t))\to \cE(t,u(t))\quad \text{as $k\rightarrow \infty$}\quad \text{for every $t\in[0,T],$}
\]
and $u$ is a Visco-Energetic solution of the rate-independent system $(X,\mathcal{E},\sfd,\delta)$ starting from $u_0$.
\end{theorem}

\emph{The proof} of Theorem \ref{thm:existence} will follow
a standard structure, that will be 
exploited in Section \ref{sec:convergenceproof},
strongly relying on the basic and preliminary 
results of Sections \ref{sec:dissipationcost},
\ref{sec:energy-inequalities}.
\begin{itemize}
\item 
  We will first derive discrete stability estimates in Section
  \ref{subsec:discrete-estimates} for the solution of the incremental
  minimization problem \eqref{ims}: here only Assumption \mytag A{}
  will play an important role.
\item We will prove a preliminary
  convergence result by refined compactness arguments (Section
  \ref{subsec:compactness-proof}), 
  where we combine the lower semicontinuity of the residual stability
  function $\cR$
  \mytag C1 with the 
  $\sfd$-separation of the $\sfD$-stable set $\SSD$
  \mytag C2; in this way, we 
  will prove that every limit curve satisfies the 
  stability condition \eqref{stability}.
\item We will then obtain the energy-dissipation inequality 
  \eqref{leqinequality} by
  proving the lower semicontinuity
  of the augmented total variation $\mVar{\sfd,\sfc}$ 
  (section \ref{subsec:limit-dissipation-inequality}):
  this property strongly relies on
  the lower semicontinuity of the jump dissipation cost
  $\sfc$, which will be thoroughly studied in Section 
  \ref{sec:dissipationcost}: here 
  the minimal compatibility properties \mytag B1-\mytag B2 
  of the viscous correction $\delta$ will enter in the game.
\item The whole argument will be concluded
  by showing that along arbitrary stable curves 
  the decay rate of the energy can always be controlled
  by the power integral and the augmented variation $\mVar{\sfd,\sfc}$ 
  \eqref{eq:crinequality}: this topic 
  will be discussed in Section \ref{sec:energy-inequalities}
  and strongly depends on \mytag B3.
\end{itemize}



\subsection{The residual stability function $\cR$}
\label{subsec:residual-stability}
Let us briefly discuss a few properties of the 
residual stability function $\cR$. We first introduce 
the Moreau-Yosida regularization of $\cE$ and its associated minimal set.
\begin{definition}[Moreau-Yosida regularization and minimal set]
  Let us suppose that $\cE$ satisfies \emph{\mytag A1}.
  The $\sfD$-Moreau-Yosida regularization $\cY:[0,T]\times X\to \R$ of $\cE$
  is defined by
  \begin{equation}
    \label{eq:111}
    \cY(t,x):=\min_{y\in X}\cE(t,y)+\sfD(x,y).
  \end{equation}
  For every $t\in [0,T]$ and $x\in X$ the minimal set is 
  \begin{equation}
    \label{eq:107}
    \rmM(t,x):=\argmin_X \cE(t,\cdot)+\sfD(x,\cdot)=
    \Big\{y\in X:\cE(t,y)+\sfD(x,y)=\cY(t,x)\Big\}.
  \end{equation}
\end{definition}
Notice that by \mytag A1 
$\rmM(t,x)\neq\emptyset$ for every $t,x$.
It is clear that 
\begin{equation}
  \label{eq:112}
  \cR(t,x)=\cE(t,x)-\cY(t,x).
\end{equation}
In the next Lemma we collect a list of useful properties, connecting
$\cR$, $\cY$ and $\rmM$.
\begin{lemma}
  \label{le:usefulR}
  Let us suppose that \emph{Assumption \mytag A{}} holds. Then
  \begin{enumerate}[i)]
  \item 
      \begin{equation}
        \label{eq:115}
        \cE(t,y)+\sfD(x,y)+\cR(t,x)\ge \cE(t,x)\quad
        \text{for every $t\in [0,T],\quad x,y\in X$}
      \end{equation}
      and equality holds in \eqref{eq:115} if and only if $y\in \rmM(t,x)$.
    \item The map $(t,x)\mapsto \cY(t,x)$ is $\sigma_\R$-lower
      semicontinuous on every sublevel of $\cF_0$.
    \item If $(t_n,x_n)\xrightarrow {\sigma_\R} (t,x)$ and $\sfD(x_n,x)\to0$ as
      $n\to\infty$,
      we
      have
      \begin{equation}
        \label{eq:113}
        \limsup_{n\to\infty}\cY(t_n,x_n)\le \cY(t,x),\quad
        \liminf_{n\to\infty}\cR(t_n,x_n)\ge \cR(t,x).
      \end{equation}
      In particular, if $y\mapsto \sfD(y,x)$ is $\sigma$-continuous
      (on the sublevels of $\cF_0$)
      then
      $\cR$ is lower semicontinuous (on the sublevels of $\cF_0$).
    \item If $(t_n,x_n)\xrightarrow{\sigma_\R} (t,x)$ and $\cE(t_n,x_n)\to \cE(t,x)$
      then
      \begin{equation}
        \label{eq:114}
        \limsup_{n\to\infty}\cR(t_n,x_n)\le \cR(t,x).
      \end{equation}
      If moreover $\liminf_{n\to\infty}\cR(t_n,x_n)\ge \cR(t,x)$ then 
      any limit point $y$ of a sequence $y_n\in \rmM(t_n,x_n)$ belongs
      to $\rmM(t,x)$.
    \item If $(t_n,x_n)\xrightarrow{\sigma_\R} (t,x)$ and $\sfD(x_n,x)\to0$ as
      $n\to\infty$
      with $\sfd(x_o,x_n)\le C$, we
      have
      \begin{equation}
        \label{eq:113bis}
        \lim_{n\to\infty}\cR(t_n,x_n)=0 \quad\Rightarrow\quad
        \lim_{n\to\infty}\cE(t_n,x_n)=\cE(t,x).
      \end{equation}
  \item $\cR$ is lower semicontinuous on the sublevels of $\cF_0$ if
    and only if for every $\sfd$-bounded sequence $(t_n,x_n)$ converging to
    $(t,x)$ in $[0,T]\times X$ with 
    $\lim_{n\to\infty}\cE(t_n,x_n)=\bar \cE=\cE(t,x)+\eta$,
    $\eta\ge0$ there exists $y\in \rmM(t,x)$ 
    and a sequence $y_n$ 
    such that
    \begin{equation}
      \label{eq:116}
      \liminf_{n\to\infty}\Big(\cE(t,y_n)+\sfD(x_n,y_n)\Big)\le \cE(t,y)+\sfD(x,y)+\eta.
    \end{equation}
  \end{enumerate}
\end{lemma}
\begin{remark}[Mutual recovery sequences]
  \upshape
  The characterization vi) 
  of the lower semicontinuity of $\cR$ (a crucial property in view of
  \eqref{eq:134}) 
  is strongly related to the \emph{mutual recovery sequence}
  condition which typically characterizes the closure of 
  the $\sfd$-stable set in the energetic case 
  (see e.g.~\cite[Lemma 2.1.14]{Mielke-Roubicek15}).
  Notice however that the sequence $(t_n,x_n)$ in vi) is not assumed to
  be stable.
\end{remark}
\begin{proof}[Proof of Lemma \ref{le:usefulR}]
  i) is an immediate consequence of the definition.\\[4pt]
  ii) Let us consider a sequence $(t_n,x_n)_n$ 
  with $\cF_0(x_n)\le C,\ \cY(t,x_n)\le Y$ for every $n\in \N$ and $(t_n,x_n)\xrightarrow{\sigma_\R}
  (t,x)$ as $n\to\infty$. 
  Let $y_n\in
  \rmM(t_n,x_n)$ so that  $\cY(t_n,x_n)=
  \cE(t_n,y_n)+\sfD(x_n,y_n)\le Y$; it is easy to check (see also Theorem
  \ref{thm:discretestimates})
  that $\cF_0(y_n)\le C'$ so that 
  it is not restrictive to assume by \mytag A1 (up to extracting a not relabeled
  subsequence)
  that $y_n\sigmato y$. The lower semicontinuity of $\cE$, $\sfd$ and
  $\delta$ yield 
  $$\cY(t,x)\le \cE(t,y)+\sfD(x,y)\le
  \liminf_{n\up\infty}\cE(t_n,y_n)+\sfD(x_n,y_n)\le Y.$$ 
  iii) Let $y\in \rmM(t,x)$ so that $\cY(t,x)=\cE(t,y)+\sfD(x,y)$; by definition
  \begin{displaymath}
    \cY(t_n,x_n)\le \cE(t_n,y)+\sfD(x_n,y)
    \le \left|\int_{t_n}^t\cP(r,y)\,\d r\right|+\cE(t,y)+\sfD(x_n,y);
  \end{displaymath}
  passing to the limit as $n\up\infty$ and recalling \mytag B2 we
  get the first property of \eqref{eq:113}. The second one follows
  by \eqref{eq:112} and the lower semicontinuity of $\cE$.\\[4pt]
  iv) \eqref{eq:114} is a consequence of i), \eqref{eq:112}, and the
  convergence 
  of the energy. The last statement follows immediately by i).\\[4pt]
  v)   \eqref{eq:115} yields
  \begin{displaymath}
    \cE(t_n,x_n)\le \cE(t_n,x)+\sfD(x_n,x)+\cR(t_n,x_n)
  \end{displaymath}
  so that $\limsup_{n\to\infty}\cE(t_n,x_n)\le \cE(t,x)$. Since $x_n$
  belongs to a sublevel of $\cF_0$
  the $\sigma$-lower semicontinuity of $\cE$ yields \eqref{eq:113bis}.
  \\[4pt]
  vi) If property \eqref{eq:116} holds for every sequence
  $(x_n)_n$, 
  up to extracting a further subsequence
  it is not restrictive to suppose that $\cR(t_n,x_n)$ is converging,
  so that 
  \begin{align*}
    \lim_{n\to\infty}\cR(t_n,x_n)
    &\ge
      \limsup_{n\to\infty}\cE(t_n,x_n)-\cE(t_n,y_n)-\sfD(x_n,y_n)
    \\&\ge
          \cE(t,x)+\eta-\liminf_{n\to\infty}\Big(\cE(t_n,y_n)+\sfD(x_n,y_n)\Big) 
     \\&\kern-6pt     \stackrel{\eqref{eq:116}}\ge \cE(t,x)-\cE(t,y)-\sfD(x,y)=\cR(t,x).
  \end{align*}
  Conversely, let us suppose that $\cR$ is lower semicontinuous and
  let $(t_n,x_n)$ be a sequence satisfying
  $\lim_{n\to\infty}\cE(t_n,x_n)=\bar \cE=\cE(t,x)+\eta$, $\eta\ge0$. 
  We pick up any $y_n\in \rmM(t_n,x_n)$ obtaining
  \begin{align*}
    \limsup_{n\up\infty}\cE(t,y_n)
    &+\sfD(x_n,y_n)
    =\limsup_{n\up\infty}\cE(t_n,y_n)
    +\sfD(x_n,y_n)
      \\&=
      \limsup_{n\up\infty}\cE(t_n,x_n)-\Big(\cE(t_n,x_n)-\cE(t_n,y_n)-\sfD(x_n,y_n)\Big)
    \\&=\cE(t,x)+\eta-\liminf_{n\to\infty}\Big(\cE(t_n,x_n)-\cE(t_n,y_n)-\sfD(x_n,y_n)\Big)
    \\&= \cE(t,x)+\eta-\liminf_{n\to\infty}\cR(t_n,x_n)
        = \cE(t,x)+\eta-\cR(t,x)
        \\&\le \cE(t,x)+\eta-\Big(\cE(t,x)-\cE(t,y)-\sfD(x,y)\Big)=
        \cE(t,y)+\sfD(x,y)+\eta.
  \end{align*}
\end{proof}

\subsection{Optimal jump transitions}
\label{subsec:optimal-transitions}
Thanks to the jump conditions given by \eqref{Jve}, we can give a finer description of the behaviour of Visco-Energetic solutions along jumps. The crucial notion is provided by the following definition.
\begin{definition}[Optimal transitions] Let $t\in[0,T]$ and $u_-$,
  $u_+\in X$. 
We say that a curve $\vartheta\in \rmC_{\sigma,\sfd}(E;X)$, $E$ being a
compact subset of $\R$, is an optimal transition between $u_-$ and $u_+$ if
\begin{equation}
  u_-=\vartheta(E^- ),\quad
  u_+=\vartheta(E^+ ),\quad
  \Fd(t,u_-,u_+)=\Cf(t,\vartheta,E).
\end{equation}
$\vartheta$ is \emph{tight} if for every $I\in \Holes(E)$ 
$\vartheta(I^-)\neq \vartheta(I^+)$.
$\vartheta$ 
is a
\begin{align}
  \text{pure jump transition, if }&E\setminus \{E^-,E^+\}\text{ is discrete,}\\
  \text{sliding transition, if } &\Gs(t,\theta(r))=0\quad \text{for every $r\in E$},  \\
  \text{viscous transition, if } &\Gs(t,\theta(r))>0\quad \text{for every $r\in E\setminus \{E^{\pm}\}$} \label{eq:viscoustransition}.
\end{align}
We will say that a compact set $E$ is almost discrete if 
$E\setminus \{E^-,E^+\}$ is discrete. 
\end{definition}
It is easy to check that 
almost discrete compact sets $E$
can be parametrized by sequences $n\mapsto e_n$
defined in a compact interval
$Z$ of $\Z\cup\{\pm\infty\}$ and continuous
at $n=\pm\infty$ whenever those points belong to $Z$.  

The main interest of optimal transitions derives from the next result,
whose proof follows immediately from Corollary \ref{cor:existenceopt} later on.
\begin{theorem} Under the same assumptions {\em \mytag A{}, \mytag
    B{}, \mytag C{} } 
  of Theorem 
  \ref{thm:existence}, if $u\in \BV{\sigma,\sfd}([0,T];X)$ is a
  Visco-Energetic solution to the rate-independent system $(X,\cE,\sfd,\delta)$,
  then for every $t\in \Ju$ there exists a tight optimal transition
  $\vartheta\in \rmC_{\sigma,\sfd}(E,X)$ 
  between $u(t\leftl )$ and $u(t\rightl )$ such 
  that 
  \begin{equation}
    u(t)\in \vartheta(E)\quad\text{and}\quad \cE(t,u(t\leftl ))-\cE(t,u(t\rightl ))=\Cf(t,\vartheta,E).\label{eq:117}
\end{equation}
\end{theorem}
\begin{remark} \label{rem:viscouspoint}
  \upshape
  If \mytag C1 holds and 
  $\vartheta\in\rmC_{\sigma,\sfd}(E,X)$ is a transition with finite
  cost 
  $\Cf(t,\vartheta,E)<\infty$, 
  then the set
  \begin{equation}
    \label{eq:118}
    E_\cR:=\Big\{r\in E\setminus \{E^+\}:\Gs(t,\theta(r))>0\Big\}\text{ is discrete,
      i.e.~all its points are isolated.}
  \end{equation}
  %
  Indeed, since $\cR$ is lower semicontinuous by \mytag C1,
  there exists $\eta>0$ such that 
  $\cR(t,\vartheta(r))\ge \frac 12 \cR(t,\vartheta(r_0))>0$ 
  for every $r_0\in E_\cR$ and $|r-r_0|<\eta$. 
  On the other hand, the finiteness of the transition cost yields
  $\sum_{s\in E\setminus \{E^+ \}} \Gs(t,\vartheta(s))<\infty$ 
  so that $E_\cR\cap\{r\in \R:|r-r_0|\le \eta\}$ is finite.
\end{remark}
We have another interesting characterization of optimal
Visco-Energetic transitions.
Whenever a set $E\subset \R$ is given, we will use the notation 
\begin{equation}
  \label{eq:119}
  r^-_E:=\sup \big(E\cap (-\infty,r)\big)\cup\{E^-\}
  ,\quad
  r^+_E:=\inf \big(E\cap (r,+\infty)\big)\cup\{E^+\},\quad
  r\in \R.
\end{equation}
\begin{theorem}
  A curve $\vartheta\in \rmC_{\sigma,\sfd}(E,X)$ with $\vartheta(E)\ni
  u(t)$ is
  an optimal transition between $u(t-)$ and $u(t+)$ 
  satisfying \eqref{eq:117} if and only if
  it satisfies 
  \begin{equation}
    \label{eq:120}
    \var(\vartheta,E\cap[r_0,r_1])\le
    \cE(t,\vartheta(r_0))-\cE(t,\vartheta(r_1))\quad
    \text{for every }r_0,r_1\in E,\ r_0\le r_1,
  \end{equation}
  and
  \begin{equation}
    \label{eq:121}
    \vartheta(r)\in \rmM(t,\vartheta(r^-_E))\quad \text{for every
    }r\in E\setminus \{E^-\}.
  \end{equation}
\end{theorem}
\begin{proof}
  By the additivity property of $\Cf$ \eqref{eq:195} and 
  the energy inequality \eqref{eq:crinqualitytheta}
  it is easy to check that \eqref{eq:117} yields
  \begin{equation}
    \label{eq:196}
    \Cf(t,\vartheta,E\cap[a,b])=\cE(t,\vartheta(a))-\cE(t,\vartheta(b))\quad
    \text{for every }a,b\in E,\ a<b.
  \end{equation}
  Since $\var(\vartheta,E\cap[r_0,r_1])\le
  \Cf(t,\vartheta,E\cap[r_0,r_1])$ we get \eqref{eq:120}; 
  particularizing \eqref{eq:196} to the case of $a=r_E^-$, $b=r$, 
  we also get 
  \begin{equation}
    \label{eq:197}
    \sfD(\vartheta(r_E^-),\vartheta(r))+\Gs(t,\vartheta(r_E^-))=
    \cE(t,\vartheta(r_E^-))-\cE(t,\vartheta(r))
  \end{equation}
  showing that $\vartheta(r)\in \rmM(t,\vartheta(r_E^-))$ by 
  Lemma \ref{le:usefulR} i).
  Notice that when $r=r_E^-$ 
  \eqref{eq:196} simply yields $\cR(t,\vartheta(r))=0$,
  i.e.~$\vartheta(r)
  \in \SSD(t)$.

  In order to prove the converse implication, we fix $\eps>0$ and we
  consider 
  a finite subset $H\subset \Holes(E)$ such that 
  \begin{equation}
    \label{eq:198}
    \sum_{I\in H}\delta(\vartheta(I^-),\vartheta(I^+))\ge \Cd(t,\vartheta,E)-\eps,
  \end{equation}
  and let $H_\pm:=\{I^\pm:I\in H\}$. 
  Since $E_\cR
  \subset H_-
  $ by Remark \ref{rem:viscouspoint}, we can
  choose $H$ sufficiently big so that 
  \begin{equation}
    \label{eq:199}
    \sum_{I\in H}\cR(t,\vartheta(I^-))\ge \sum_{s\in E\setminus\{E^+\}}\cR(t,\vartheta(s))-\eps.
  \end{equation}
  Let us consider now an arbitrary finite part
  $F=\{E^-=s_0<s_1<\cdots<s_N=E^+\}\subset E$ 
  containing $H_-\cup H_+$ such that 
  \begin{equation}
  \sum_{n=1}^N \sfd(\vartheta(s_{n-1}),\vartheta(s_n))\ge
  \var(\vartheta,E)-\eps\label{eq:200}
\end{equation}
and let $k\mapsto n(k)$
  be an increasing sequence such that $H_-=\{s_{n(k)}:1\le k\le
  K\}$. Notice that for every $I\in H$ if $I^-=s_{n(k)}$ then
  $I^+=s_{n(k)+1}$, since $I^+\subset F$.
  Setting $n(0)=0$, $n(K+1)=N$,
  we have
  \begin{align*}
    &\cE(t,u(t-))-\cE(t,u(t+))
    =\sum_{k=0}^K \cE(t,\vartheta(s_{n(k)}))
      -\cE(t,\vartheta(s_{n(k+1)}))
      =
      \cE(t,\vartheta(s_0))-\cE(t,\vartheta(s_{n(1)}))\\&\qquad+
     \sum_{k=1}^{N}\Big(\cE(t,\vartheta(s_{n(k)}))
      -\cE(t,\vartheta(s_{n(k)+1}))\Big) 
      +\sum_{k=1}^{N}\Big(\cE(t,\vartheta(s_{n(k)+1}))-\cE(t,\vartheta(s_{n(k+1)}))\Big) 
    \\&\topref{eq:121}\ge 
        \var(\vartheta,E\cap[s_0,s_{n(1)}])+
        \sum_{k=1}^{N}\Gs(t,\vartheta(s_{n(k)}))+\sum_{k=1}^{N}\delta(\vartheta(s_{n(k)}),\vartheta(s_{n(k)+1}))
        \\&\qquad+\sum_{k=1}^{N}
        \sfd(\vartheta(s_{n(k)}),\vartheta(s_{n(k)+1}))
            +
            \sum_{k=1}^{N}\var(\vartheta,E\cap[s_{n(k)+1},s_{n(k+1)}])
        \\&   \ge
            \sum_{n=1}^N\sfd(\vartheta(s_{n-1}),\vartheta(s_n))
            +\sum_{k=1}^{N}\Gs(t,\vartheta(s_{n(k)}))
            +\sum_{k=1}^{N}\delta(\vartheta(s_{n(k)}),\vartheta(s_{n(k)+1}))
    \ge \Cf(t,\vartheta,E)-3\eps,
  \end{align*}
  where the last inequality results from \eqref{eq:198},
  \eqref{eq:199}, and \eqref{eq:200}. Since $\eps>0$ is arbitrary,
  by recalling \eqref{eq:crinqualitytheta} we
  get \eqref{eq:117}.
\end{proof} 
\begin{corollary}[Representation of optimal transitions of viscous type]
  We can always
  represent an optimal viscous transition between $u(t\leftl)$ and $u(t\rightl)$ as a finite or countable
  sequence $n\mapsto
  \vartheta(n)$ defined in a compact interval
  $Z$ of $\Z\cup\{\pm\infty\}$ satisfying 
  \begin{equation}
    \label{eq:122}
    \vartheta({n})\in \rmM(t,\vartheta(n-1))\quad
    \text{for every }n\in Z\setminus\{Z^-\},\quad
    \vartheta(Z^\pm)=u(t\pm),
  \end{equation}
  and the continuity conditions (whenever $\pm\infty\in Z$)
  \begin{equation}
    \label{eq:124}
    \begin{gathered}
      \lim_{n\to\pm\infty}\vartheta(n)=u(t\pm),\quad
      \lim_{n\to-\infty}\sfd(u(t-),\vartheta(n))=0,\quad
      \lim_{n\to+\infty}\sfd(\vartheta(n),u(t+))=0,\\
      \lim_{n\down-\infty}\cE(t,\vartheta(n))=\cE(t,u(t-)),
      \quad
      \lim_{n\up+\infty}\cE(t,\vartheta(n))=\cE(t,u(t+)).
    \end{gathered}
  \end{equation}
\end{corollary}
An optimal transition $\vartheta$ 
can be decomposed in a canonical way into (at most countable)
collections of sliding and pure jump transitions.
\begin{proposition} Let 
$\vartheta\in \rmC_{\sigma,\sfd}(E,X)$ be an
 optimal transition
between $u_-$ and $u_+$. 
Then there exist disjoint closed intervals $(S_j)_{j\in \sigma}$ 
and almost discrete compact sets $\{V_k\}_{k\in \nu}$, with $\sigma,\nu\subset \N$, such that
\begin{equation}
E=(\cup_{j\in\sigma} S_j)\cup
\overline{(\cup_{k\in\nu}V_k)}\label{eq:125}
\end{equation}
and 
\begin{equation}
\vartheta_{|S_j}\text{ is of sliding type,}\quad\vartheta_{|V_k}\text{
  is of pure jump type}.\label{eq:126}
\end{equation}
\end{proposition}
\begin{proof}
We set for every $r\in [E^-,E^+]$
$$E_0:=\bigcup_{I\in 
\Holes(E)}
\{I^-,I^+\},\quad
E_1:=E\setminus {E_0},\quad 
  E_{0}
  (r):=E\cap [r^-_{E_1
  },r^+_{E_1
  }].
$$
Notice that $E_0$ contains all the isolated points of $E$
(in particular it contains $E_\cR$). 
If $r\in E_0$, $E_0
(r)$ is the closure of the ``maximal component'' of
$E_0
$ containing $r$,
in the sense that all the other points of $E_0
$ are separated
from $r$ by some accumulation point in $E_1
$. 
The restriction of $\vartheta$ to $E_0
(r)$ is of pure jump type 
and $E_0(r)$ is almost discrete.
 
We first decompose $E_0$ in the disjoint countable union of $\cup_{k\in
  \nu}V_k\cap E_0$
where $V_k$ is of the form $E_0(r)$ for some $r\in E_0$.
We then set 
\[
V=\overline {E_0}
,\quad S:=E\setminus (V\cup\{E^\pm\}),
\]
observing that $$ S=(E^-,E^+)\setminus 
\overline{\bigcup_{I\in \Holes(E)} I}.$$
We can now decompose the set $S$, open in $\R$, as the 
disjoint union of its connected components $(a_j,b_j)$, $j\in \sigma$, 
and we set $S_j:=[a_j,b_j]$ obtaining
\eqref{eq:125}. Since $E_\cR\subset E_0$ 
we also get \eqref{eq:126}.
\end{proof}
As we have seen in Remark \ref{rem:viscouspoint}, if an optimal
transition $\vartheta:E\rightarrow X$ 
is of viscous type, then the set $E\setminus \{E^- ,E^+ \}$ is
discrete. 
In general it may happen that $E$ is homeomorphic to 
a finite set of $\Z$ or to infinite intervals of the form
$\{-\infty\}\cup-\N$, $\N\cup\{+\infty\}$ or even to $\Z\cup\{\pm\infty\}$.
We can be more precise in the case when 
the functional
\begin{equation}
u\mapsto \cE(t,u)+\sfD(u_0,u)\quad 
\text{admits a unique minimizer in $X$ for every $u_0\in X$.}\label{eq:127}
\end{equation}
%
%
This happens, e.g.~, if $X$ is a linear space and we choose
a sufficiently strong viscous correction $\delta$ so that the map
$u\mapsto\cE(t,u)+\sfD(u_0,u)$ is strictly convex.
\begin{proposition} 
  Let 
$\vartheta:E\rightarrow X$ be a tight optimal transition between $u_-$ and
$u_+$. \\[4pt]
  i) If the energy and the dissipation satisfy 
  \eqref{eq:127} 
 then every $r\in E\setminus (E_\cR\cup\{E^+\})$ 
(in particular $r=E^-$ when
$u_-$ is stable) is a right accumulation point of $E$, i.e.~%
there exists a sequence $r_k\in E\cap(r,\infty)$ such that
$r_k\downarrow r$.\\[4pt]
ii) If $X$ is a vector space, $\sfd$ is the distance induced by
a norm on $X$, $\delta(u,v):=\frac \mu 2\sfd^2(u,v)$ as in
\eqref{eq:9}
and $\cE$ is Gateaux differentiable in $X$ then 
every $r\in E\setminus (E_\cR\cup\{E^-\})$ 
(in particular $r=E^+$ when
$u_+$ is stable) is a left accumulation point of $E$, i.e.~%
there exists a sequence $r_k\in E\cap(-\infty,r)$ such that
$r_k\uparrow r$.
\end{proposition}
\begin{proof}
Let us consider i) and let us suppose by contradiction 
that there exists $s\in E$ such that $(r,s)\in \Holes(E).$
Since $\vartheta(r)\in\SSD(t)$ 
we have $\vartheta(r)\in \rmM(t,\vartheta(r))$;
on the other hand, \eqref{eq:121} yields $\vartheta(s)\in
\rmM(t,\vartheta(r))$ so that \eqref{eq:127} yields
$\vartheta(r)=\vartheta(s)$
which contradicts the tightness of $\vartheta$.

Concerning ii) we still argue by contradiction
assuming that $(s,r)\in \Holes(E)$. 
We denote by $\xi\in X^*$ the unique element 
of the Gateaux subdifferential of $\cE(t,\vartheta(r))$,
by $N$ the subdifferential of $\frac 12\|\cdot\|_X^2$ and by 
$K_*$ the dual unitary ball of $X^*$.
It is not difficult to check that
\begin{equation}
  \label{eq:128}
  \xi\in K_*,\quad
  \frac{N(\vartheta(r)-\vartheta(s))}{\|\vartheta(r)-\vartheta(s)\|}+
  \mu\,N(\vartheta(r)-\vartheta(s))\ni \xi,
\end{equation}
so that we 
obtain
\begin{displaymath}
  (1+\mu\|\vartheta(r)-\vartheta(s)\|) \|N(\vartheta(r)-\vartheta(s))\|_*\le {\|\vartheta(r)-\vartheta(s)\|}
\end{displaymath}
which contradicts the fact that 
$\|N(x)\|=\|x\|$.
\end{proof}
\begin{remark} 
\upshape
When $\cE$ is nonsmooth 
a jump from {a non-stable point to a stable one }may happen 
even with the assumption of strict convexity of the functional
$\mathcal{E}+D$.
 For instance, we can consider the example
\[
X=\R,\quad \cE(t,u)=a|u|,\quad d(u,v)=|u-v|,\quad \delta(u,v)=\frac1{2}|u-v|^2.
\]
If $a>1$, it is immediate to check that $\SSD(t)=\{0\}$ for every $t$. 
If we start from a point $u_-\in (0,a-1)$ 
then $u_+=0$ belongs to $\rmM(t,u_-)$ and it is also a stable point.
\end{remark}

\section{Examples}
\label{sec:examples}
In this section we will discuss some applications of Theorem
\ref{thm:existence} about existence of Visco-Energetic solutions. 
Let us first recall that once Assumption \mytag A{} holds
and
\begin{equation}
  \label{eq:154}
  \sfD\text{ is left continuous on the sublevels of $\cF_0$ (see 
    \eqref{eq:150})
    and $\sfd$ separates $X$,}
\end{equation}
%
conditions \mytag B1, \mytag B2, \mytag C1, \mytag C2 are 
automatically satisfied, so that one can just focus on 
the verification of the canonical compactness-regularity conditions
\begin{equation}
\text{\mytag A{} and on the compatibility condition
\mytag B3.}
\label{eq:132}
\end{equation}
The latter is also satisfied if $\delta(u,v)$ is a
function of $\sfd$ as in \eqref{eq:10}.
\subsection{The convex case}
\label{subsec:convex}
Let us first consider the case when $X$ is a convex subset of a 
vector space 
$V$ and $\sfd$ is induced by a
convex, positively $1$-homogeneous
functional $\psi:V\to [0,+\infty)$. 
\begin{proposition}
  \label{prop:E=VE}
  If $\sfd(x,y):=\psi(y-x)$ for every $x,y\in X$
  and 
  the map
  $x\mapsto\cE(t,x)$ is convex in $X$ for every $t\in
  [0,T]$, 
  we have
  \begin{enumerate}[(i)]
  \item 
    If $u_-\in \SS_\sfd(t)$ and $u_+\in X$ satisfy the energetic
    jump condition $\cE(t,u_+)+\psi(u_+-u_-)=\cE(t,u_-)$
    then $\sfc(t,u_-,u_+)=\sfd(u_-,u_+)$.
    \item If the viscous correction $\delta$ satisfies
    \begin{equation}
      \label{eq:28}
      \lim_{\theta\down0}\frac{\delta(u,(1-\theta)u+\theta
        v)}{\theta}=0\quad
      \text{for every }u,v\in X
    \end{equation}
    then $\SSD=\SS_\sfd$.
  \end{enumerate}
  In particular any energetic solution $u\in \BV{\sfd}([0,T];X)$ 
  of
  $(X,\cE,\sfd)$ is a {\em VE} solution of $(X,\cE,\sfd,\delta)$
  and
  if \eqref{eq:28} holds any {\em VE} solution $u\in \BV{\sfd}([0,T];X)$ 
  of $(X,\cE,\sfd,\delta)$ is an energetic solution of $(X,\cE,\sfd)$.
\end{proposition}
\begin{proof}
  Let us first suppose that $u$ is an energetic solution.
  Since the ``energetic'' stability condition \eqref{enstability} is stronger 
  than the corresponding ``Visco-Energetic'' one \eqref{stability},
  it is sufficient to check that 
  $u$ satisfies the Visco-Energetic balance condition
  \eqref{energybalance}; since $u$ satisfies \eqref{eq:110} it is
  sufficient to check that \eqref{Jve} holds. 
  
  Thus let $u_-,u_+\in \SS_\sfd(t)$ with
  $\cE(t,u_+)+\psi(u_+-u_-)=\cE(t,u_-)$. 
  We consider the convex subset of $V\times \R$
  $$K:=\big\{(v,z)\in V\times \R: u_-+v\in
  X,\ z\le
  \cE(t,u_-)-\cE(t,u_-+v)\big\}.$$
  By the Mazur-Orlicz version of Hahn-Banach Theorem 
  \cite[Theorem 1.1]{Simons98}
  there exists a linear functional $L:V\times \R\to \R$ such that 
  \begin{equation}
    \label{eq:77}
    L(v,z)\le \psi(v)-z\quad \text{for every }(v,z)\in 
    V\times\R,\qquad
    \inf_{(v,z)\in K}L=\inf_{(v,z)\in K}\psi(v)-z.
  \end{equation}
  Writing $L(v,z)=\ell(v)-\alpha z$ for some $\alpha\in \R$ and
  testing the first condition of \eqref{eq:77} with $v=0$ and
  arbitrary $z\in \R$ we get $L(v,z)=\ell(v)-z$, so that 
  \begin{equation}
    \label{eq:78}
    \ell(v)\le \psi(v)\quad\text{for every }v\in V.
  \end{equation}
  Since $u_-$ is $\sfd$-stable, for every $(v,z)\in K$ we have 
  \begin{displaymath}
    \psi(v)-z\ge \cE(t,u_-)-\cE(t,u_-+v)-z\ge0;
  \end{displaymath}
  since $(0,0)\in K$ we conclude that $\inf_K L=0$, which yields in particular
  \begin{equation}
    \label{eq:144}
    \ell(w-u_-)\ge \cE(t,u_-)-\cE(t,w)\quad \text{for every }w\in X.
  \end{equation}
  %
  Choosing $w=u_+$ in \eqref{eq:144} we deduce  $\ell(u_+-u_-)=\psi(u_+-u_-)$.
  Setting $\vartheta(s):=(1-s)u_-+su_+$, $s\in [0,1]$,
  we immediately get
  $\cE(t,\vartheta(s))=(1-s)\cE(t,u_-)+s\cE(t,u_+)=\cE(t,u_-)-
  \psi(\vartheta(s)-u_-)$ and
  \begin{align*}
    \cE(t,v)&\ge \cE(t,u_-)-\ell(v-u_-)
              \\&=\cE(t,\vartheta(s))-\ell(v-\vartheta(s))+
              \Big(\cE(t,u_-)-\cE(t,\vartheta(s))-\ell(\vartheta(s)-u_-) \Big)
    \\&\ge \cE(t,\vartheta(s))-\ell(v-\vartheta(s))+\Big(\psi(\vartheta(s)-u_-)-\ell(\vartheta(s)-u_-)\Big)
    \\&\ge \cE(t,\vartheta(s))-\psi(v-\vartheta(s)),
  \end{align*}
  so that $\vartheta(s)\in \SS_\sfd(t)\subset \SSD(t)$. It follows
  that $\sfc(t,u_-,u_+)\le \var(\vartheta,[0,1])=\sfd(u_-,u_+)$.

  In order to prove the converse implication, we simply have to check that
  $\SSD\subset \SS_\sfd.$ If $u\in \SSD(t)$, $v\in X$ and
  $v_\theta:=(1-\theta)u+\theta v$ with $\theta\in [0,1]$ we have
  \begin{displaymath}
    \cE(t,u)\le \cE(t,v_\theta)+\psi(v_\theta-u)+\delta(u,v_\theta)\le 
    (1-\theta)\cE(t,u)+\theta\Big(\cE(t,v)+\psi(v-u)+\frac{\delta(u,v_\theta)}{\theta}\Big),
  \end{displaymath}
  which yields
  \begin{displaymath}
    \cE(t,u)\le \cE(t,v)+\psi(v-u)+\frac{\delta(u,v_\theta)}{\theta}.
  \end{displaymath}
  Passing to the limit as $\theta\down0$ we conclude.
\end{proof}
\subsection{The $1$-dimensional case.}
\label{subsec:1D}
In the space $X:=\R$ consider a function $W\in \rmC^2(\R)$ bounded
from below with $-\lambda:=\inf_\R W''>-\infty$, a function $\ell\in
\rmC^1([0,T])$ and positive numbers $\alpha_\pm,\mu$;
the standard example for $W$ is the double-well potential $W(u)=\frac 14(1-u^2)^2$.  We set
\begin{equation}
  \label{eq:129}
  \cE(t,u):=W(u)-\ell(t)u,\quad
  \sfd(u,v):=\sum_\pm \alpha_\pm(v-u)_\pm,\quad
  \delta(u,v):=\frac\mu2 |u-v|^2.
\end{equation}
Since we are in the simplified setting recalled
at the beginning of Section \ref{sec:examples},
it is easy to check that all the assumptions \mytag A{}, \mytag B{},
\mytag C{} hold. 
A careful analysis (see \cite{Minotti16}) shows that when 
$\ell$ is strictly increasing, the initial datum $u_0$ 
satisfies a suitable stability condition and $\mu \alpha_+^2 > \lambda$
then $u\in \BV{}([0,T];\R)$ is a VE solution of $(X,\cE,\sfd,\delta)$ 
if and only if it is nondecreasing in $[0,T]$ and
\begin{equation}
  \label{eq:130}
  W'(u(t))=\ell(t)-\alpha_+,
\end{equation}
so that the evolution of $u$ can be described in terms of the upper
monotone envelope of $W'$ starting from $u_0$,
as in the case of Balanced Viscosity solutions, see 
\cite{Rossi-Savare13} and Figure \ref{fig:3} in the Introduction.
When $\ell$ is strictly decreasing then
$u$ should be non-increasing and \eqref{eq:130} should be replaced by
$W'(u(t))=\ell(t)+\alpha_-$.

In the case when $0<\mu\alpha_+^2<\lambda$ we have 
a sort of intermediate behaviour between the previous situation and 
the energetic case, corresponding to $\mu=0$ where
increasing jumps between $u(t-)<u(t+)$ obey the Maxwell rule 
\[ 
\int_{u(t-)}^{u(t+)}\Big(W'(r)-\ell(t)+\alpha_+\Big)\,\d r=0.
\] 
In particular, in the visco-energetic case, 
an increasing jump occurs at $t$ 
when we have the modified Maxwell rule
\begin{equation}
\int_{u(t-)}^{u_+}\Big(W'(r)-\ell(t)+\alpha_++\mu(r-u(t-))\Big)\,\d r=0
\quad\text{for some }u_+>u(t-).\label{eq:201}
\end{equation}
In this case, however, $u(t+)$ may differ from $u_+$, see
Figure \ref{fig:4}:
we refer to \cite{Minotti16} for a detailed analysis.

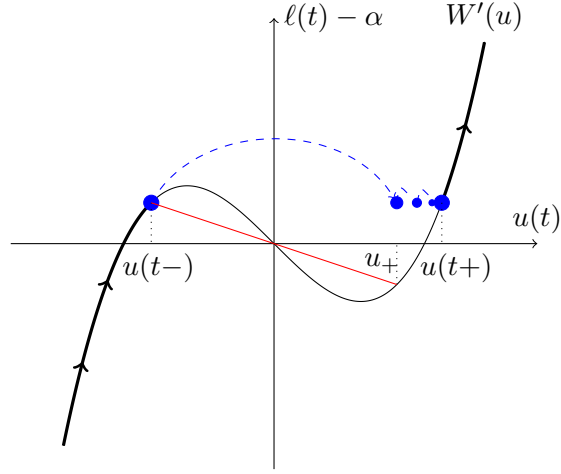
\begin{figure}[!h]
\label{fig:4}
\centering
\begin{tikzpicture}

\draw[->] (-3.5,0) -- (3.5,0) node[above] {$u(t)$};
\draw[->] (0,-3) -- (0,3) node[right] {$\ell(t)-\alpha$};

 \begin{scope}[scale=2]
    \draw[domain=-1.4:1.4,samples=200] plot ({\x}, {(\x)^3-\x}) node[above] {$W'(u)$};
     \draw[domain=-1.4:-0.816,very thick,samples=200] plot ({\x}, {(\x)^3-\x}) [arrow inside={}{0.33,0.66}];
     \draw[domain=1.1153:1.4,very thick,samples=200] plot ({\x}, {(\x)^3-\x}) [arrow inside={}{0.50}]    ;
     \foreach \Point in {(-0.816,0.272),(1.1153,0.272)}{
     \draw[fill=blue,blue] \Point circle(0.05);
     }
     \draw[fill=blue,blue] (0.816,0.272) circle(0.04);
     \foreach \Point in {(1.082,0.272), (1.112,0.272), (1.115,0.272)}{
    \draw[fill=blue,blue] \Point circle(0.01);
}
\draw[->, dashed,blue] (-0.816,0.272) to [bend left=60] (0.8,0.3);
\draw[dashed ,blue] (0.81,0.272) to [bend left=60,looseness=3] (0.92,0.3);
\draw[dashed,blue] (0.95,0.272) to [bend left=60,looseness=3] (1.05,0.3);

\draw[fill=blue,blue] (0.95,0.272) circle(0.03);
\draw[fill=blue,blue] (1.05,0.272) circle(0.02);

\draw[dotted] (-0.816,0.272) -- (-0.816,0) node[below] {\,\,\,$u(t-)$};
\draw[dotted] (0.816,-0.272) -- (0.816,0) node[below] {$u_+\quad$};
\draw[dotted] (1.1153,0.272) -- (1.1153,0) node[below] {\quad$u(t+)$};

\draw[red] (-0.816,0.272) -- (0.816,-0.272);
  \end{scope}

\end{tikzpicture}
\caption{Visco-Energetic solutions for a double-well energy $W$ with an
  increasing load $\ell$ and $0<\mu\alpha^2< -\min W''$. 
  In this case the solution $u$ jumps before reaching the local
  maximum of $W$ and the optimal transition $\vartheta$ 
  makes a first jump connecting $u(t-)$ with $u_+$ according to the
  modified
  Maxwell rule of \eqref{eq:201}: $u(t-)$ and $u_+$ corresponds to
  the intersection of the graph of $W'$ with the red line, whose slope is $-\mu$.
  After the first jump, $\vartheta$ makes an infinite sequence of
  jumps accumulating to
  $u(t+)$.}
\end{figure}

\subsection{The choice of $\delta$: $\alpha$-$\Lambda$ geodesic convexity.}
\label{subsec:genconvexity} 
In some situations it could be interesting to choose a viscous
correction $\delta$ associated with a metric different from $\sfd$: 
we want to show a typical example where \mytag B3 still holds
and a related application to the evolution of the Allen-Cahn energy.

Let us consider for simplicity 
\begin{equation}
  \begin{gathered}
    \text{the metric setting of Remark \ref{rem:metric} with
      $\delta(u,v):=\frac 12\sfd^2_*(u,v)$,}\\
    \text{where $\sfd_*$ is  
another distance on $X$,
continuous on each sublevel of $\cF_0$,}
  \end{gathered}
\label{eq:142}
\end{equation}
so that
\mytag B1-\mytag B2 
and \mytag C1-\mytag C2 hold.
\begin{definition}[$\alpha\text{-}\Lambda$
  convexity]
  \label{def:aLconvexity}
  Let $\alpha>0$, $\Lambda\ge0$.
  We say that $(\cE,\sfd,\sfd_*)$ satisfies 
  the \emph{weak} $\alpha\text{-}\Lambda$
  convexity property on a set $S\subset X$
  if for every $x,y\in S
  $ there
  exists a curve
  $\gamma:[0,1]\to X$ 
  such that
  \begin{gather}
    \label{eq:genconvexity} 
      \cE(t,\gamma(\theta))\leq (1-\theta)\cE(t,x)+\theta\cE(t,y)
      -\frac 12\theta(1-\theta)\Big[\alpha \sfd_*^2(x,y)-
      \Lambda\sfd (x,y)\sfd_*(x,y)\Big],\\
      \label{eq:genconvexity2}
      \liminf_{\theta\down0}\frac{\sfd(x,\gamma(\theta))}\theta\le \sfd(x,y),\quad
      \lim_{\theta\down0}\frac{\sfd_*(x,\gamma(\theta))}{\sqrt
        \theta}=0.
  \end{gather}
  We say that $(\cE,\sfd,\sfd_*)$ satisfies 
  the \emph{strong} $\alpha\text{-}\Lambda$
  convexity property if  for every $x,y\in X$ 
  there
  exists a curve
  $\gamma:[0,1]\to X$ connecting $x$ to $y$ 
  satisfying \eqref{eq:genconvexity} and
  \begin{equation}
    \label{eq:202}
    \sfd(\gamma(\theta),\gamma(\theta'))=|\theta-\theta'|\sfd(x,y),\quad
    \sfd_*(\gamma(\theta),\gamma(\theta'))=|\theta-\theta'|\sfd_*(x,y)
  \end{equation}
  for every $\theta,\theta'\in [0,1]$.
\end{definition}
Observe that \eqref{eq:genconvexity} is a generalization of the 
$\lambda$-convexity along geodesics, involving two distances: see \cite{Mielke-Rossi-Savare13}.

Let us show that if $(\cE,\sfd,\sfd_*)$ satisfies the 
weak $\alpha\text{-}\Lambda$ convexity on $S:=\cup_{t\in
  [0,T]}\SSD(t)$ then 
\mytag B3 holds.
In fact, if $x\in \SSD(t),y\in \SSD(s)$ and 
$\gamma$ satisfies \eqref{eq:genconvexity}-\eqref{eq:genconvexity2}, then
\begin{align*}
&\cE(t,x)\topref{stablepoints}\le
\cE(t,\gamma(\theta))+\sfd(x,\gamma(\theta))+\frac
          12\sfd_*^2(x,\gamma(\theta)) 
\\&\topref{eq:genconvexity}\le \!\!
    (1-\theta)\cE(t,x)+\theta\cE(t,y)-\frac{\theta(1-\theta)}2\sfd_*(x,y)\Big[\alpha
    \sfd_*(x,y) -\Lambda\sfd(x,y)\Big]+\sfd(x,\gamma(\theta))+\frac12\sfd_*^2(x,\gamma(\theta)).
\end{align*}
Subtracting $(1-\theta)\cE(t,x)$ and dividing by $\theta$ we obtain
\[
\cE(t,x)\leq \cE(t,y)+\frac{\sfd (x,\gamma(\theta))}\theta+\frac 1{2\theta}
\sfd_*^2(x,\gamma(\theta))-\frac 12(1-\theta)\Big[\alpha\sfd_*^2(x,y) -\Lambda\sfd(x,y)\sfd_*(x,y)\Big].
\]
Passing to the limit as $\theta\downarrow 0$ and using \eqref{eq:genconvexity2} we get
\begin{equation}
\label{eq:13}
\cE(t,x)-\cE(t,y)-\sfd(x,y) \le
-\frac{\alpha}{2}\sfd_*^2(x,y)+\frac{\Lambda}{2}\sfd(x,y)\sfd_*(x,y)  
\le \frac{\Lambda^2}{\relax 8\alpha}\sfd^2(x,y).
\end{equation}
To recover \mytag B3 is enough to divide by $\sfd(x,y)$ and to
pass to the limit as $x\to y$.

As a further consequence of the above conditions we can also prove an
enhanced
BV estimate, which is related to 
a coercivity property of $\cR$, see Lemma \ref{le:Rprop}. 
The proof will be collected in the
last section \ref{subsec:last}.
\begin{theorem}[BV estimates w.r.t.~$\sfd_*$]
  \label{thm:bvestimate} 
  Let us assume that {\em\mytag A{}} holds and 
  $(\cE,\sfd,\sfd_*)$ satisfies the \emph{strong}
  $\alpha\text{-}\Lambda$ convexity property. 
      If
    \begin{equation} \label{eq:powerlipschitz} |\cP(t,x)-\cP(t,y)|\le
      L\sfd_*(x,y)\quad \text{if $t\in[0,T]$ and $x,y\in X$},
    \end{equation}
    and \eqref{eq:141} holds, 
    then any $\mathrm{VE}$ solution $u$ obtained as a pointwise limit of
    the time incremental minimization scheme \eqref{ims} belongs to
    $\BV{\sfd_*}([0,T];X)$.
\end{theorem}
\addcontentsline{toc}{subsubsection}{\it VE Evolution for the Allen-Cahn functional}
\begin{example}[VE evolution for the Allen-Cahn functional]
\label{subsec:AllenCahn}
\upshape
Let us consider  a bounded open
Lipschitz domain $\Omega\subset \R^d$,
a function $W\in \rmC^2(\R)$ as in the previous Example
\ref{subsec:1D} and let us set $X=\{u\in W^{1,2}_0(\Omega):W(u)\in
L^1(\Omega)\}$
endowed with the $L^1(\Omega)$-topology.

 The distance $\sfd$ is the usual one induced by the $L^1$ norm, while $\delta$ is the squared distance induced by the $L^2$ norm.
\[
\sfd(u,v):=\int_\Omega|u(x)-v(x)|\rmd x,\qquad \delta(u,v):=\frac{\mu}{2}\int_\Omega |u(x)-v(x)|^2\rmd x.
\]
We also consider the energy functional
\begin{equation} \label{eq:energyW}
\cE(t,u)=\begin{cases} 
\displaystyle
\int_\Omega\left(\frac{1}{2}|\nabla u|^2+W(u)-\ell(t)u\right)\rmd x \quad &\text{if $u\in \rmW^{1,2}_0(\Omega)$}; \\ +\infty \quad&\text{otherwise},
\end{cases}
\end{equation}
where 
$\ell\in C^1([0,T];L^2(\Omega)).$
It is immediate to check that for all $u\in X$ the function $t\mapsto \cE(t,u)$ is differentiable, with derivative 
\[
\cP(t,u)=-\int_\Omega\ell'(t)u\,\d x
\] 
so that Assumptions \mytag A{} are satisfied since the sublevels of
the energy are compact in $L^2(\Omega)$. 
Thus we are in the
canonical metric setting and the only nontrivial assumption is
\mytag B3 since $\delta$ is continuous on the sublevels of the energy. 
We will check that the $\alpha$-$\Lambda$ convexity discussed in
\eqref{subsec:genconvexity} is satisfied.
If $W$ is $\lambda$-convex with $\lambda>0$, we can use the estimate
\begin{equation} \label{eq:genconvexityexample}
\cE(t,(1-\theta)u+\theta v)\le (1-\theta)\cE(t,u)+\theta\cE(t,v)-\frac{\theta(1-\theta)}{2}\left(\|\nabla(u-v)\|^2_{L^2(\Omega)}+\lambda\|u-v\|^2_{L^2(\Omega)}\right),
\end{equation} 
hence we have \eqref{eq:genconvexity} with $\alpha=\lambda$ and $\Lambda=0$. If $\lambda<0$ we use the estimate (see \cite[Example 5.1]{Mielke-Rossi-Savare13})
\[
-\|\nabla(u-v)\|_{L^2(\Omega)}^2\le -(1+|\lambda|)\|u-v\|_{L^2(\Omega)}^2 +M_\lambda\|u-v\|_{L^1(\Omega)}^2
\]
for some $M_\lambda>0$. Inserting this into
\eqref{eq:genconvexityexample} we obtain the generalized convexity
\eqref{eq:genconvexity} with
$\frac{\alpha}{2}=(1+|\lambda|)+\lambda=1>0$ and
$\frac{\Lambda}{2}=(1+|\lambda|)M_\lambda$ and then also \mytag B1 is
satisfied. 
We can therefore apply Theorem \ref{thm:existence} and prove the existence of a
Visco-Energetic solution for the rate-independent system
$(X,\cE,\sfd,\delta)$. 
\end{example}
\subsection{Product spaces and degenerate-singular distances}
\label{subsec:degenerate}
\newcommand{\llbracket}{[\kern-1.5pt[}
\newcommand{\rrbracket}{]\kern-1.5pt]}
\newcommand{\PHI}{F}
\renewcommand{\sfY}{\Phi}
In many important examples the space $X$ is a cartesian product 
$X=\PHI\times Z$ (whose points can be written as $u=(\varphi,z)$, $\varphi\in \PHI$,
$z\in Z$) but $\sfd$ only depends on the $z$-component
\begin{equation}
  \label{eq:4}
  \sfd(u,u'):=\widetilde\sfd(z,z')\quad\text{if }u=(\varphi,z),\ u'=(\varphi',z'),
\end{equation}
for a quasi-distance $\widetilde \sfd$ separating $Z$.
In these cases it is natural
to consider a viscous correction $\delta(u,u')=\widetilde\delta(z,z')$
which still depends only on $z$ (but more
general interesting situations can occur, see
e.g.~\cite{DalMaso-Toader02} or
\cite{Knees-Negri15,Negri16} where an alternate minimization scheme
has been studied):
therefore, even if $\widetilde\sfd$ separates $Z$, the
distance $\sfd$ does not separate $X$.

It may happen that 
for every $z\in Z$ the set
\begin{equation}
  \label{eq:15}
  \sfY(t,z):=\argmin_\PHI\cE(t,\cdot,z)
\end{equation}
contains only one point. Since $\delta$ and $\sfD$ do not depend on
$\varphi$, one can easily check
that 
\begin{equation}
  \label{eq:147}
  (\varphi,z)\in \SSD(t)\quad \Rightarrow\quad \varphi\in \sfY(t,z),
\end{equation}
and $\sfd_\R$ separates 
$\SSD$. 
As an example, we consider the following model discussed in 
  \cite[Sect.~6.2]{Mainik-Mielke05} (we refer to
  \cite{Mainik-Mielke05} and \cite{Minotti16T} for the
  interpretation and more
  details).
\addcontentsline{toc}{subsubsection}{\it A delamination problem}
\begin{example}[A delamination problem]
  \upshape
  Let $O$ be a sufficiently regular open connected domain of $\R^d$, 
  $\Gamma_{\rm dir}\subset \partial O$ with positive surface measure and let
  $\Gamma\subset O$ be a piecewise smooth hypersurface, such that
  $\Omega:=O\setminus \Gamma$ is still connected.
  Let $\phi_{\rm dir}\in H^1(\Omega;\R^d)$ and $\PHI:=\big\{\varphi\in
  H^1(\Omega;\R^d): \varphi=\phi_{\rm dir}\ \text{on }\Gamma_{\rm
    dir}\big\}$ endowed with the weak topology $\sigma_F$ of $H^1$ and
  we set $Z:=L^\infty(\Gamma;[0,1])$ endowed with the weak$^*$
  topology $\sigma_Z$.
  
  We thus define
  \begin{equation}
    \label{eq:145}
    \cE(t,\varphi,z):=\int_\Omega \sfW(\rmD \varphi(x))\,\d x+
    \int_\Gamma z(x)\,\sfQ(\llbracket\varphi\rrbracket(x))\,\d \mathcal H^{d-1}(x)-
    \la \ell(t),\varphi\ra
  \end{equation}
  where $\sfW:\R^{d\times d}\to \R$ is the quadratic form of 
  linearized elasticity, $\sfQ:\R^d\to \R^d$ is a nonnegative quadratic form,
  $\llbracket \varphi\rrbracket\in H^{1/2}(\Gamma;\R^d)$ denotes the jump of the deformation of $\varphi$ across
  $\Gamma$, and $\ell\in \rmC^1([0,T];(H^1(\Omega))')$.
  We eventually introduce the dissipation
  $\sfd((\varphi,z),(\varphi',z')):=\widetilde\sfd(z,z')$ with
  \begin{equation}
    \label{eq:146}
    \widetilde\sfd(z,z'):=\int_\Gamma \psi(z'(x)-z(x))\,\d\mathcal
    H^{d-1}(x),\quad
    \text{where}\quad
    \psi(r):=
    \begin{cases}
      r&\text{if }r\ge0\\
      +\infty&\text{otherwise,}
    \end{cases}
  \end{equation}
  and the viscous correction $\delta((\varphi,z),(\varphi',z'))=h(\widetilde\sfd(z,z'))$ as in
  \eqref{eq:10}.
  
  Arguing as in \cite{Mainik-Mielke05} and taking into account 
  Example \ref{ex:1} it is easy to check that \mytag A{}
  and \mytag B{} are satisfied; also
  the separation property \mytag C2 follows by the above remarks
  since the set $\sfY(t,z)$ defined by \eqref{eq:15} contains only one
  element. 

  The only property that remains to be checked is 
  the closure of the $(\sfD,Q)$-quasi stable set \mytag C1.
  By \eqref{eq:134} we can apply Lemma \ref{le:usefulR} vi):
  if $(\varphi_n,z_n)\sigmato (\varphi,z)$ in $F\times Z$
  with $\cE(t,\varphi_n,z_n)\to \cE(t,\varphi,z)+\eta$
  and $(\varphi',z')$ is a minimizer of 
  $\cE(t,\cdot)+\sfD((\varphi,z),\cdot)$ in $\rmM(t,(\varphi,z))$,
  we have $\varphi'\in \sfY(t,z')$ and $z'\ge z$, so that we can apply
  \cite[Lemma 6.1]{Mainik-Mielke05} to find another sequence $z_n'\in
  Z$ satisfying $z_n'\le z_n$, $z_n'\stackrel{\sigma_Z}\to z'$ in $Z$ and
  $\widetilde\sfd(z_n,z_n')
  \to \widetilde\sfd(z,z')$; this also implies that 
  $\widetilde\delta(z_n,z_n')\to \widetilde\delta(z,z')$.
  Correspondingly, we set $\varphi_n':= \sfY(t,z_n')$ (with a slight
  abuse of notation, we still denote by $\sfY(t,z)$ the unique element
  of the set). Since the maps $z\mapsto \sfY(t,z)$ and
  $z\mapsto \cE(t,\sfY(t,z),z)$ are continuous (see \cite{Mainik-Mielke05})
  with respect to the topology of $Z$ 
  we deduce that 
  $\cE(t,\varphi_n',z_n')\to \cE(t,\varphi',z')$.
  We conclude that \eqref{eq:116} is satisfied since
  \begin{align*}
    \liminf_{n\to\infty}\Big(\cE(t,\varphi_n',z_n')+\sfD((\varphi_n,z_n),(\varphi_n',z_n'))\Big)
    &=\cE(t,\varphi',z')+\sfD((\varphi,z),(\varphi',z'))
      \\&\le \cE(t,\varphi',z')+\sfD((\varphi,z),(\varphi',z'))+\eta.
          \quad\Box
  \end{align*}
\end{example}
\subsection{Marginal energies}
\label{subsec:marginal}
In the same cartesian setting of the previous Section,
\ref{subsec:degenerate}
let us now consider the case when the set $\sfY(t,z)$ of \eqref{eq:15} contains more than one
element.
One can try to write a reduced model in the space $Z$ by
introducing the marginal energy functionals and its generalized power
\begin{equation}
  \label{eq:143}
  \widetilde\cE(t,z):=\min\Big\{\cE(t,\varphi,z):\varphi\in
  F\Big\},\quad
  \widetilde\cP(t,u):=\max\Big\{\cP(t,\varphi,z):\varphi\in \sfY(t,z)\Big\}.
\end{equation}
If $\cE$ satisfies \mytag A{}
one can easily prove that $\sfY(t,z)$ is compact in $F$ for every
$t,z$ and
\begin{equation}
  \label{eq:148}
  (t_n,z_n)\to(t,z),\quad
  \widetilde\cE(t,z_n)\le C\quad\Rightarrow\quad
  \Ls_{n\to\infty}\sfY(t,z_n)\subset \sfY(t,z),
\end{equation}
where $\Ls$ denotes the Kuratowski superior limit, see Definition
\ref{def:Kuratowski}.
The following Lemma allows to easily check conditions
\mytag A{}.
\begin{lemma}
  \label{le:cP}
  If the functionals $\cE,\cP$ satisfy \emph{Assumptions \mytag A{1},
    \mytag A2 (resp.~\mytag A{2'})}
  in $(X,\sigma,\sfd)$ then $(\widetilde \cE,\widetilde\cP)$ satisfy
  \emph{Assumptions \mytag A{1}, \mytag A2 (resp.~\mytag A{2'})}
  in $(Z,\sigma_Z,\widetilde\sfd)$.
\end{lemma}
\begin{proof}
  Notice that if $(t,z)$ belongs to the sublevel $\big\{\widetilde\cF\le C\big\}$ 
  then $(t,\varphi,z)$ belongs to $\big\{\cF\le C\big\}$ for every $\varphi\in
  \sfY(t,z)$. Property \mytag A1 is easy to verify, so we consider
  \mytag A2.
  
  The upper semicontinuity of $\widetilde\cP$ follows immediately by
  \eqref{eq:148}:
  selecting $\varphi_n\in \sfY(t_n,z_n)$ so that
  $\widetilde\cP(t_n,z_n)=
  \cP(t_n,\varphi_n,z_n)$ and observing that $(t_n,\varphi_n,z_n)$
  belong to a sublevel of $\cF$, we can suppose that 
  $\varphi_n$ converges to some $\varphi\in \sfY(t,z)$ so that
  the upper semicontinuity of $\cP$ yields
  \begin{displaymath}
    \limsup_{n\to\infty}\widetilde\cP(t_n,z_n)=
    \limsup_{n\to\infty}\cP(t_n,\varphi_n,z_n)\le 
    \cP(t,\varphi,z)\le \widetilde\cP(t,z).
  \end{displaymath}
  In the case of \mytag A{2'}, we observe that if $\widetilde
  \cE(t_n,z_n)\to
  \widetilde \cE(t,z)$ and $ \sfY(t_n,z_n)\ni\varphi_n\to \varphi$ as in the
  above argument, 
  we have 
  $$\widetilde\cE(t,z)\le \cE(t,\varphi,z)\le \liminf_{n\to\infty}
  \cE(t_n,\varphi_n,z_n)=\liminf_{n\to\infty}\widetilde\cE(t_n,z_n)=
  \widetilde\cE(t,z)$$
  so that $\cE(t_n,\varphi_n,u_n)\to \cE(t,\varphi,u)$ and we
  can apply the conditional upper semicontinuity of $\cP$.
  
  Concerning \eqref{A.21} we observe that for some $\varphi\in \sfY(t,z)$
  \begin{displaymath}
    |\widetilde\cP(t,z)|=
    |\cP(t,\varphi,z)|\le C_P\cF(t,\varphi,z)=
    C_P\widetilde\cF(t,z).
  \end{displaymath}
  As for \eqref{A.22}, since 
  \begin{align*}
    \liminf_{s\up t}\frac{\widetilde\cE(t,z)-\widetilde\cE(s,z)}{t-s}
    \ge \liminf_{s\up
    t}\frac{\cE(t,\varphi,z)-\cE(s,\varphi,z)}{t-s}
    \ge \cP(t,\varphi,z)\quad\text{for every }\varphi\in \sfY(t,z),
  \end{align*}
  so that
  \begin{displaymath}
    \liminf_{s\up
      t}\frac{\widetilde\cE(t,z)-\widetilde\cE(s,z)}{t-s}\ge \widetilde\cP(t,z);
  \end{displaymath}
  the corresponding right $\limsup$ inequality of \eqref{A.22} follows
  by the same argument.
\end{proof}
Let us consider for the sake of simplicity the case when
$\widetilde\sfd$ is left continuous and $\widetilde
\delta=h(\widetilde \sfd)$.
\begin{theorem}
  \label{thm:existence2}
  Let us suppose that the energy functionals $\cE,\cP$ satisfy
  \emph{Assumptions \mytag A1, \mytag A{2'}}, $\widetilde\sfd$
  separates $Z$ and
  $\widetilde\delta=h(\widetilde\sfd)$ as in \eqref{eq:10}.
  Then for every $z_0\in Z$ there exists a \emph{VE} solution
  to the R.I.S.~$(Z,\widetilde \cE,\widetilde\sfd,\widetilde\delta)$.
  Equivalently, 
  there exist a map $z\in \BV{\sigma_Z,\widetilde \sfd}([0,T];Z)$ 
  and a map $\varphi:[0,T]\to \PHI$ 
  (which is measurable, if $\PHI$ is Souslin) 
  such that 
  $\varphi(t)\in \sfY(t,z(t))$ for every $t\in [0,T]$,
  \begin{equation}
    \label{eq:158}
    \cE(t,\varphi(t),z(t))\le
    \cE(t,\varphi',z')+\widetilde\sfd(z(t),z')+\widetilde\delta(z(t),z')
    \quad\text{for every }t\in [0,T]\setminus\Jump{}z,
  \end{equation}
  \begin{equation}
    \label{eq:149}
    \cE(t,\varphi(t),z(t))+\mVar{\widetilde\sfd,\widetilde\sfc}(z,[s,t])
    =\cE(s,\varphi(s),z(s))+\int_s^t \cP(r,\varphi(r),z(r))\,\d r.
  \end{equation}
\end{theorem}
The proof is immediate by applying Theorem \ref{thm:existence}
to the R.I.S.~$(Z,\widetilde \cE,\widetilde\sfd,\widetilde\delta)$
and recalling the remarks stated at the beginning of Section
\ref{sec:examples}.
We then select 
\begin{equation}
  \label{eq:153}
  \varphi(t)\in \sfY(t,z(t))\quad\text{such that }\cP(t,\varphi(t),z(t))=\widetilde\cP(t,z(t));
\end{equation}
by the Von Neumann-Aumann selection Theorem 
\cite[Section III.6]{Castaing-Valadier77} $\varphi$ can also be
supposed
to be measurable, if $\PHI$ is a Souslin space (in particular, if 
$\sigma_F$ is metrizable, since the sets $\sfY(t,z(t)) $ are contained
in a compact set).

An interesting application of the above result concerns 
a material model driven by a nonconvex elastic energy,
discussed in \cite[Sect.~4]{Francfort-Mielke06}
in the framework of energetic evolutions.  
\begin{example}[A material model with a nonconvex elastic energy]
  \upshape
  \addcontentsline{toc}{subsubsection}{\it A material model with a nonconvex elastic energy}
  We consider a Lipschitz and bounded open set $\Omega\subset \R^d$,
  a compact set $K\subset \R^m$,
  two exponents $\alpha,p>1$ and
  two maps
  \begin{equation}
    \label{eq:159}
    \varphi_{\rm dir}\in \rmC^1([0,T];W^{1,p}(\Omega)),\quad
    \ell\in \rmC^1([0,T](W^{1,p}(\Omega))').
  \end{equation}
  The spaces $\PHI$ and $Z$ are defined by
  \begin{equation}
    \label{eq:160}
    \PHI:= W^{1,p}_0(\Omega),\quad
    Z:=\Big\{z\in W^{1,\alpha}(\Omega;\R^m):z(x)\in K\Big\}
  \end{equation}
  endowed with their weak topologies and the energy functional is
  \begin{equation}
    \label{eq:161}
    \cE(t,\varphi,z):=
    \int_\Omega W(\rmD \varphi(x)+\rmD\varphi_{\rm dir}(t,x),z(x))\,\d
    x
    +\lambda\int_\Omega |\rmD z(x)|^\alpha\,\d x
    -\la \ell,\varphi+\varphi_{\rm dir}\ra,
  \end{equation}
  where $W\in\rmC(\R^{d\times d}\times K;\R_+)$ is $\rmC^1$ and quasiconvex with respect to 
  its first variable and satisfies
  \begin{equation}
    \label{eq:162}
    c |D|^p-C\le W(D,z)\le C(1+|D|^p)
    \quad\text{for every }D\in \R^{d\times d},\ z\in K
  \end{equation}
  for some constants $0<c<C<\infty$.
  $\widetilde\sfd$ is an asymmetric distance on $Z$ satisfying
  \begin{equation}
    \label{eq:163}
    C^{-1}\|z-z'\|_{L^1}\le \widetilde \sfd(z,z')\le
    C\|z-z'\|_{L^1}\quad\text{for every }z,z'\in Z.
  \end{equation}
  Notice that in this case
  \begin{equation}
    \label{eq:164}
    \cP(t,\varphi,z)=\int_\Omega \rmD W(\varphi+\varphi_{\rm
      dir}(t))\cdot
    \rmD \partial_t\varphi_{\rm dir}(t)\,\d x-
    \la \partial_t \ell(t),\varphi+\varphi_{\rm dir}(t)\ra
    -\la \ell(t),\partial_t\varphi_{\rm dir}(t)\ra,
  \end{equation}
  satisfies the assumptions stated in \mytag A{2'}
  thanks to an argument of \cite{DalMaso-Francfort-Toader05}, 
  see \cite[Prop.~4.4]{Francfort-Mielke06}.
  
  By choosing a viscous correction as in \eqref{eq:10} we can
  therefore
  apply Theorem \ref{thm:existence2} and prove the existence
  of a VE solution. We refer to 
  \cite{Minotti16T} for more details. 
\end{example}
\section{Main structural properties of the viscous dissipation cost}
\label{sec:dissipationcost}
\relax
In this section we will prove some relevant properties of the viscous
transition and dissipation costs $\Cf(t,\vartheta,E)$ and $\Fd(t,u_0,u_1)$ introduced in
Definition \ref{def:transition-cost} and \ref{dissipationcost}.
They lie at the core of the structure of Visco-Energetic solutions and
of our existence proof.

\subsection{Additivity and Invariance by rescaling}
\label{subsec:invariance}
A first simple fact concerns the possibility of performing suitable
rescaling of the domain
$E$ of a transition $\vartheta:E\to X$ without affecting the cost.
This is related to the following additivity property of $\var$ and
$\Cd$:
for every $a,b,c\in E$ with $a<b<c$ we have
\begin{equation}
  \label{eq:36}
  \begin{aligned}
    \var(\vartheta,E\cap[a,c])&= \var(\vartheta,E\cap[a,b])+
    \var(\vartheta,E\cap[b,c]), \\
    \Cd(\vartheta,E\cap[a,c])&=
    \Cd(\vartheta,E\cap[a,b])+ \Cd(\vartheta,E\cap[b,c]).
  \end{aligned}
\end{equation}
\begin{lemma}
  \label{le:reparametrization}
   Let $E\subset \R$ compact and $\vartheta\in \rmC_{\sigma,\sfd}(E,X)$ with 
   \begin{equation}
     \label{eq:137}
     \var(\vartheta,E)+\Cd(\vartheta,E)=C.
   \end{equation}
  There exists a compact set $\tilde E$ with $\tilde E^- =0,\tilde E^+ =C+1$ and
  a bijective Lipschitz map $\sfs:\tilde E\to E$ such that 
  the new transition $\tilde \vartheta:=\vartheta\circ \sfs$ satisfies
  \begin{equation}
    \label{eq:48}
    \var(\tilde\vartheta,\tilde
    E\cap[r_0,r_1])+\Cd(\tilde\vartheta,\tilde E\cap[r_0,r_1])\le
    |r_0-r_1|\quad
    \text{for every }r_0,r_1\in \tilde E,\ r_0<r_1,
  \end{equation}
  \begin{equation}
    \label{eq:79}
    \begin{gathered}
      \var(\vartheta,E)=\var(\tilde\vartheta,\tilde E),\quad
      \Cd(\vartheta,E)=\Cd(\tilde\vartheta,\tilde E).
    \end{gathered}
  \end{equation}
  Moreover, for every $\sft:E\to [0,T]$ setting $\tilde\sft
  :=\sft\circ \sfs$ we have
  \begin{equation}
    \label{eq:138}
    \sum_{s\in
      E\setminus\{E^+ \}}\Gs(\sft(s),\vartheta(s))=\sum_{r\in \tilde
      E\setminus\{\tilde E^+ \}}\Gs(\tilde\sft(r),\tilde\vartheta(r)).
  \end{equation}
\end{lemma}
\begin{proof}
  We define $\sfr:E\to [0,C+1]$ by
  \begin{equation}
    \label{eq:49}
    \sfr(s):=
    \frac{s-E^- }{E^+ -E^- }+\var(\vartheta,E\cap[E^- ,s])+
      \Cd(\vartheta,E\cap[E^- ,s]);
  \end{equation}
  it is not too difficult to check that $\sfr$ is continuous and
  strictly increasing so that 
  $\tilde E:=\sfr(E)$ is compact. Moreover
  \begin{equation}
    \label{eq:50}
    |s_1-s_0|\le |E^+ -E^- | |\sfr(s_1)-\sfr(s_0)|\quad\text{for every
    }s_0,s_1\in [E^- ,E^+ ],
  \end{equation}
  so that $\sfr$ admits a Lipschitz continuous inverse $\sfs$ defined
  in $[0,C+1]$, which satisfies \eqref{eq:48} by construction,
  thanks to \eqref{eq:36}, and   
  \eqref{eq:79}-\eqref{eq:138}.  
\end{proof}
\subsection{Lower semicontinuity of the transition and dissipation cost}
\relax
Since the viscous transition cost $\sfc$ involve curves $\vartheta:E\to X$
defined in general compact
parametrization domains $E\subset \R$, 
it will be crucial to study its lower semicontinuity 
along sequence of transition curves $\vartheta_k$ defined in \emph{varying
domains} $E_k$. 

Let us first recall the notion of convergence in the sense of
Kuratowski in a Hausdorff topological space $(Y,\rho)$
satisfying the first axiom of countability.
\begin{definition} [Kuratowski convergence] 
  \label{def:Kuratowski}
  Let $(A_k)_k$ be a sequence of
  subsets of $Y$. The Kuratowski limit inferior (resp.~limit superior)
  of $A_k$, as $k\rightarrow \infty$
  are defined by:
\begin{align}
\label{eq:32}
\underset{k\rightarrow \infty}{\Li} A_k:={}&\left\{a\in Y:\exists\, a_k\in
  A_k\text{ such that }a_k\stackrel\rho\to a\right\}, \\
  \label{eq:34}
\underset{k\rightarrow \infty}{\Ls} A_k:={}&\left\{a\in Y:\exists\,
  n\mapsto k_n\text{ increasing, and }\ a_{k_n}\in A_{k_n}: a_{k_n}\stackrel\rho\to a\right\}.
\end{align}
We say that $A_k\stackrel K\rightarrow A$ in the Kuratowski sense if
$A=\underset{k\rightarrow \infty}{\Li}A_k=\underset{k\rightarrow \infty}{\Ls}A_k.$
\end{definition}
Recall that Kuratowski convergence coincides with $\Gamma$-convergence
of the indicator functions $i_k:=i_{A_k}$ associated with the sets
$A_k$ \cite[Chapter 4]{DalMaso93}, where in general 
\begin{equation}
  \label{eq:57}
  i_A(x):=
  \begin{cases}
    0&\text{if }x\in A,\\
    +\infty&\text{if }x\not\in A.
  \end{cases}
\end{equation}
Whenever $Y$ is a metric space and $A_k,A$ are compact
sets, then Kuratowski convergence coincides with the convergence
induced by the Hausdorff distance. 

Let us now consider a sequence $\vartheta_k\in \rmC(E_k,X)$, where $E_k$
is compact subset of $\R$. 
In order to study the asymptotic behaviour of $\vartheta_k$ to some limit curve
$\vartheta\in \rmC(E,X)$ we can simply consider the Kuratowski
convergence
of the graphs $\graph(\vartheta_k)$ to $\graph(\vartheta)$ in
$\R\times X$ (see e.g.~\cite{Kuratowski55}).
Notice that 
\begin{equation}
  \label{eq:51}
  \graph(\vartheta)\subset \Li_{k\up\infty}\graph(\vartheta_k)\quad
  \Leftrightarrow\quad
  \forall\,s\in E\ \exists\,s_k\in E_k:\quad
  s_k\to s,\ \vartheta_k(s_k)\to \vartheta(s).
\end{equation}
This weak condition is sufficient to 
prove the lower semicontinuity of the function
$\Cf(t,\vartheta,E)$ as stated in the next Theorem, which also covers
a slightly more general situation that will turn out to be useful in
what follows.
\begin{theorem}[Lower semicontinuity of $\Cf$]
  \label{costlsc} 
  %
  Let $\vartheta\in \rmC(E,X)$, $t\in \R$, 
  and let $\vartheta_k\in \rmC(E_k,X)$, $\sft_k:E_k\to \R$, $k\in \N,$
  be
  sequences of functions
  satisfying \eqref{eq:51}.
We have the following lower semicontinuity properties.\\
  a)
  \begin{equation}
    \label{VarD}
    \var(\vartheta,E)\le \liminf_{k\up\infty}\var(\vartheta_k,E_k).
  \end{equation}
  b) If \emph{\mytag C1} holds and
  \begin{equation}
  \lim_{k\to\infty}\sup_{s\in E_k}|\sft_k(s)-t|=0,\quad
  \text{$\vartheta_k(E_k)\subset F$, where $F$ is a sublevel of
  $\cF_0$,}
\label{eq:151}
\end{equation}
then for every $t\in [0,T]$
  \begin{equation} \label{VarG}
    \sum_{s\in E\setminus \{E^+ \}}\Gs(t,\vartheta(s))\leq \liminf_{k\up\infty} \sum_{s\in E_k\setminus \{E_k\rightl \}} \Gs(\sft_k(s),\vartheta_k(s))
  \end{equation}
  c) 
  If there exists a  modulus of continuity
  $\omega:[0,\infty)\to[0,\infty)$ with $\omega(0)=0$ such that 
  \begin{equation}
    \label{eq:41}
    \sfd(\vartheta_k(x),\vartheta_k(y))\le \omega(y-x)\quad 
    \text{for every }k\in \N\ \text{and } x,y\in E_k,\ x\le y,
  \end{equation}
  and $\delta$ satisfies \emph{\mytag B1}, then
  \begin{equation}
    \label{Vardelta}
    \Cd(\vartheta,E)\le \liminf_{k\up\infty} \Cd(\vartheta_k,E_k).
  \end{equation}
  d) If \emph{\mytag B1}, \emph{\mytag C1}, \eqref{eq:151} and  \eqref{eq:41} hold,
  then
  \begin{equation}
    \label{eq:53}
    \Cf(t,\vartheta,E)\le \liminf_{k\up\infty} \Cf(t,\vartheta_k,E_k).
  \end{equation}
\end{theorem}
\begin{proof}
Clearly d) is a consequence of the first three properties a), b), and
c). Let us prove each of them.\\
\textbf{Lower semicontinuity of the total variation.} Let
$E^- =s^0<s^1< \dots <s^N=E^+ $ be a finite subset of $E$. 
By \eqref{eq:51}, for every $s^j$ there exists a sequence $s_k^j\in
E_k$ such that $s_k^j\rightarrow s^j$ and $\vartheta_k(s_k^j)\to\vartheta(s^j)$.
If $k$ is big enough, we can also assume $s_k^{j-1}\leq s_k^j$ for every $j$ and then
\[
\sum_{j=1}^N \sfd(\vartheta_k(s_k^{j-1}), \vartheta_k(s_k^j))\leq \var(\vartheta_k,E_k).
\]
In particular, taking the liminf and recalling that $\sfd$ is lower
semicontinuous we obtain
\begin{equation}
\sum_{j=1}^N \sfd(\vartheta(s^{j-1}), \vartheta(s^j))\leq \liminf_k
\var(\vartheta_k).\label{eq:54}
\end{equation}
Since \eqref{eq:54} holds for every choice of $\{s^1,\cdots,
s^N\}\subset E$, by taking the
supremum among all the finite subsets of $E$ we obtain \eqref{VarD}.\\
\textbf{Semicontinuity of the residual sum.} 
Since the viscous residual functional $\Gs(t,\cdot)$ is $\sigma$-lower
semicontinuos on $[0,T]\times F$
and positive, we can argue as in the previous step:
\[
\sum_{j=0}^{N-1} \Gs(t,\vartheta(s^j))\leq \liminf_{k\up\infty}
\sum_{j=0}^{N-1} \Gs(\sft_k(s_k^j),\vartheta_k(s_k^j))\leq 
\liminf_{k\up\infty} \sum_{s\in E_k\setminus\{E_k\rightl \}} \Gs(\sft_k(s),\vartheta_k(s)).
\] 
\eqref{VarG} then follows by taking the supremum of the left hand side.\\
\textbf{Lower semicontinuity of $\Cd$.} 
We first prove the following
property:
\begin{equation}
  \label{eq:37}
  \text{for every }I\in \Holes(E)\ \exists\ I_k\in \Holes(E_k):\quad
  \lim_{k\up\infty}I_k\uleftl =I^- ,\quad
  \lim_{k\up\infty}I_k\urightl =I^+ .
\end{equation}
Indeed, consider an increasing family of compact intervals
$C^h\uparrow I$, $h\in \N$ and two sequences $s_k\ulrl\in E_k$ such
that 
$s_k\ulrl\to I\ulrl$ and $\vartheta_k(s_k\ulrl)\sigmato \vartheta(I\ulrl)$
(they exist by \eqref{eq:51}, since $I\ulrl \in E\subset \Li_{k\up\infty}E_k$).
We will have
$E_k\cap C^h=\emptyset$ for $k$ sufficiently big, since otherwise $C^h$
should intersect $\Ls_{k\to\infty}E_k=E$. Denoting by $I^h_k$ the
connected component of $\R\setminus E_k$ intersecting $C^h$,
since $E_k\uleftl \le s_k\uleftl \le (I_k^h)\uleftl \le \min C^h$ and $E_k\urightl \ge s_k\urightl 
\ge (I_k^h)\urightl \ge \max C^h$, we clearly have
\begin{equation}
  \label{eq:38}
  I^- =\lim_{k\up\infty}s_k\uleftl \le \liminf_{k\to\infty}(I_k^h)\uleftl \le (C^h)\uleftl ,\quad
  I^+ =\lim_{k\up\infty}s_k\urightl \ge\limsup_{k\to\infty}(I_k^h)\urightl \ge (C^h)\urightl .
\end{equation}
Since $\lim_{h\up+\infty}(C^h)\ulrl=I\ulrl$, a standard diagonal
argument yields \eqref{eq:37}.
%
%

Let us now choose a subsequence $n\mapsto k_n$ such that 
\begin{equation}\label{eq:55}
    \vartheta_{k_n}(I_{k_n}\ulrl)\sigmato \theta\ulrl,\quad
    \liminf_{k\up\infty}\delta\left(\vartheta_k(I_k\uleftl ),\vartheta_k(I_k\urightl )\right)=
    \lim_{n\to\infty}\delta\left(\vartheta_{k_n}(I_{k_n}\uleftl ),\vartheta_{k_n}(I_{k_n}\urightl )\right)\ge
    \delta(\theta\uleftl ,\theta\urightl )
\end{equation}
Since $0\le I_k\uleftl -s_k\uleftl \to 0$ and $0\le s_k\urightl -I_k\urightl \to0$ as
$k\to\infty$, \eqref{eq:41} and the lower semicontinuity of $\sfd$ yield
\begin{displaymath}
  \begin{aligned}
    \sfd(\vartheta(I^- ),\theta\uleftl )&\le
    \liminf_{n\to\infty}\sfd(\vartheta_{k_n}(s_{k_n}\uleftl ),\vartheta_{k_n}(I_{k_n}\uleftl ))=0,\\
    \sfd(\theta\urightl ,\vartheta(I^+ ))&\le
    \liminf_{n\to\infty}\sfd(\vartheta_{k_n}(I_{k_n}\urightl ),\vartheta_{k_n}(s_{k_n}\urightl ))=0.
  \end{aligned}
\end{displaymath}
\mytag B1 then yields
$\delta(\vartheta(I^- ),\vartheta(I^+ ))\le
\delta(\theta^-,\vartheta(I^+ ))\le 
\delta(\theta\uleftl ,\theta\urightl )$ so
that \eqref{eq:55} yields
\begin{equation}
  \label{eq:39}
  \liminf_{k\up\infty}\delta\left(\vartheta_k(I_k\uleftl ),\vartheta_k(I_k\urightl )\right)\ge 
  \delta(\vartheta(I^- ),\vartheta(I^+ )).
\end{equation}
Since this holds for every connected component of $\mathbb{R}\setminus
E$, if we consider a finite collection of disjoint open intervals
$(I_n)_{n=1}^N\in \Holes(E)$, we can find sequences $I_{n,k}\in
\Holes(E_k)$ as in \eqref{eq:38} with $I_{n_0,k}\cap
I_{n_1,k}=\emptyset$ for distinct indices $n_0,n_1\in \{1,\cdots,N\}$.
Then
\begin{equation} 
  \label{eq:40}
\sum_{n=1}^N\delta(\vartheta(I_n\uleftl ),\vartheta(I_n\urightl ))\le\liminf_{k\up\infty}
\sum_{n=1}^N \delta(\vartheta_k(I_{n,k}\uleftl )\vartheta_k(I_{n,k}\urightl ))\leq \liminf_{k\up\infty} \Cd(\vartheta_k).
\end{equation}
Taking the supremum of the left hand side with respect to 
finite collections in $\Holes(E)$ we eventually obtain \eqref{Vardelta}.
\end{proof}

\subsection{Existence of optimal transitions}

In this section we will show that whenever $\sfc(t,u_-,u_+ )$ is finite
there exists an optimal transition $\vartheta$ attaining the infimum
in \eqref{eq:dissipationcost}. 
This results from a standard application of the Direct Method in
Calculus of Variations and the following compactness property,
which somehow combines Kuratowski and Arzel\`a-Ascoli Theorems,
see \cite[Prop.~3.3.1]{Ambrosio-Gigli-Savare08}.
\begin{theorem}[Compactness]
  \label{Kuratowski} 
Let $F\subset X$ be a sequentially compact subset of $X$, $C$ be
a compact subset of $\R$, $t\in [0,T]$ and
let $\cR:[0,T]\times F\to [0,+\infty]$ be a $\sigma$-l.s.c.~function
such that $S(t):=\{x\in F:\cR(t,x)=0\}$ is separated by $\sfd$.

If $\sft_k:E_k\to [0,T]$ and $\vartheta_k\in
\rmC_{\sigma,\sfd}(\overline{E_k},F)$ are sequences of functions with
$E_k\subset C$,
satisfying the $\sfd$-equicontinuity property \eqref{eq:41}
and the uniform bounds
\begin{equation}
  \label{eq:43}
  \sup_k \sum_{s\in
    E_k}\cR(\sft_k(s),\vartheta_k(s))=R<\infty,\quad
  \lim_{k\to\infty}\sup_{s\in E_k}|\sft_k(s)-t|=0,
\end{equation}
then
there exist a subsequence $n\mapsto k_n$, a compact set $E\subset \R$ and a function
$\vartheta\in \rmC_{\sigma,\sfd}(E,F)$ such that as $n\up\infty$:\\
1) $E_{k_n}\stackrel K\to E$, \\
2) $\graph(\vartheta)\subset
\Li_{n\up\infty}\graph(\vartheta_{k_n})$,\\
3) whenever $s_{k_n}\in E_{k_n}$ converges to $s$ with
$\Gs(\sft_{k_n}(s_{k_n}),\vartheta_{k_n}(s_{k_n}))\to0$ then $\vartheta_{k_n}(s_{k_n})\to
\vartheta(s)$,\\
4) $\vartheta_{k_n}(E_{k_n}^\pm)\to \vartheta(E^\pm)$.
%
\end{theorem}
Notice that whenever $\sfd$ separates the points of $F$ (e.g.~when $(F,\sfd)$
is a metric space) then we can choose $\cR\equiv 0$ so that
\eqref{eq:43}
is always satisfied.
\begin{proof}
%
Let us now introduce the functions $r_k:\R\to[0,\infty]$ defined by
\begin{equation}
\label{eq:44}
r_k(s):=
\begin{cases}
  \cR(\sft_k(s),\vartheta_k(s))&\text{if }s\in E_k,\\
  +\infty&\text{if }s\in \R\setminus E_k.
\end{cases}
\end{equation}
It is not difficult to check that $r_k$ are lower semicontinuous.
Compactness of $\Gamma$-convergence \cite[Theorem 8.5]{DalMaso93}
provides a further subsequence (still no relabelled) 
and a lower semicontinuous limit function $r:\R\to[0,\infty]$ such that
$\Gamma\text{-}\lim_{k\to\infty}r_k=r$.
By using the bounds
$i_{E_k} \le r_k\le i_{E_k}+R$, one can easily check that
the compact set 
$E:=\{s\in \R:r(s)\le R\}$ coincides with the Kuratowski limit of
$E_k$. It is not difficult to check, arguing as in the 
second step of the proof of Theorem 
\ref{costlsc}, that
\begin{equation}
  \label{eq:45}
  \sum_{s\in E}r(s)\le \liminf_{k\up\infty} \sum_{s\in E_k}r_k(s)\le R.
\end{equation}
It follows that the (relatively) open set $B:=\{s\in E:r(s)>0\}$ 
is at most countable and every point of $B$ is isolated in $E$.
%
%
%

Since
$E$ is separable, we can find a countable set $A$ dense 
in $E\setminus B$ and containing $E^\pm$. 
For every $s\in A\cup B\setminus\{E^\pm\}$ there exists a sequence 
$s_k(s)\in E_k$ such that 
$s_k(s)\to s$ and $r_k(s_k(s))=\cR(\sft_k(s),\vartheta_k(s_k(s))\to
r(s)$.
When $s=E^\pm$ we just choose $s_k(s):=E_k^\pm$.

Since the maps $\vartheta_k$ take values in the sequentially compact
set $F$, by a diagonal argument we can find a subsequence $n\mapsto
k(n)$ and a function $\vartheta:A\cup B\to F$ 
such that 
\begin{equation}
\vartheta_{k(n)}(s_{k(n)}(s))\sigmato
\vartheta(s),\quad
\sfd(\vartheta(s),\vartheta(s'))\le \omega(s'-s)\quad
\text{for every $s,s'\in A\cup B$, $s'\ge s$}.\label{eq:58}
\end{equation}
We now extend $\vartheta$ to the closure of $A$: since
$\vartheta(A)\subset F$ and $\overline A\subset S(t)$, it is sufficient
to 
apply Lemma \ref{le:obvious}.

In order to prove 2) for every $s\in E$ we have to 
exhibit a sequence $s_{k(n)}(s)\in E_{k(n)}$ converging to $s$ such that
$\vartheta_{k(n)}(s_{k(n)}(s))\to \vartheta(s)$.
Such a property is satisfied by construction whenever $s\in A\cup
B$. 
On the other hand, every point of
$s\in E\setminus B$ is limit of sequences $s_k\in E_k$ with
$\cR(\sft_k(s_k),\vartheta_k(s_k))\to 0$. We denote by $\theta$ the limit
of $\vartheta_{k'(n)}(s_{k'(n)}(s))$, where $k'(n)$ is a subsequence of $k(n)$.
By the lower semicontinuity of $\cR$ we get
$\theta\in S(t)$. From \eqref{eq:41} we deduce that for every
$r\in A,$ 
$r\le s$,
\begin{align*}
  \sfd(\vartheta(r),\theta)&\le 
  \liminf_{n\up\infty}\sfd(\vartheta_{k'(n)}(s_{k'(n)}(r)),\vartheta_{k'(n)}(s_{k'(n)}(s)))
  \\&\le \liminf_{n\up\infty}\omega(s_{k'(n)}(s)-s_{k'(n)}(r))\le \omega(s-r).
\end{align*} 
Since $A$ is dense in $E\setminus(A\cup B)$, the previous inequality yields
$\sfd(\theta,\vartheta(s))=0$ so that $\theta=\vartheta(s)$.

The proof of 3) follows by a completely analogous argument.
4) is a consequence of the fact that $E^\pm\in A$ and 
$s_k(E^\pm)=E_k^\pm\to E^\pm$ as $k\to\infty$.
\end{proof}
\begin{corollary}[Existence of optimal transitions]
  \label{cor:existenceopt}
  Let us assume that \emph{\mytag A{1}, \mytag B1, \mytag C{}} hold
  and let $t\in [0,T], u\ulrl\in X$ with $\sfc(t,u\uleftl,u\urightl)=C<\infty$. 
  Then there exists an optimal transition $\vartheta\in
  \rmC_{\sigma,\sfd}(E,X)$ connecting $u^-$ and $u^+$, namely 
  \begin{equation}
    \label{eq:136}
    \vartheta(E^-)=u\uleftl,\quad
    \vartheta(E^+)=u\urightl,\quad
    \sfc(t,u\uleftl,u\urightl)=\Cf(t,\vartheta,E).
  \end{equation}
\end{corollary}
\begin{proof}
  Let $\vartheta_k\in \rmC_{\sigma,\sfd}(E_k,X)$ be an optimizing
  sequence of transitions with $\vartheta_k(E^\pm)=u\ulrl$ and 
  $\Cf(t,\vartheta_k,E_k)\to \sfc(t,u\uleftl,u\urightl)$ as
  $k\up\infty$.
  By Lemma \ref{le:reparametrization}
  it is not restrictive to assume that $E_k^-=0,\ E_k^+\le C$ for a
  sufficiently big constant $C$ and that \eqref{eq:48} holds uniformly.
  In particular $\sfd(u^-,\vartheta_k(r))\le C$ for every $k\in \N$
  and $r\in E_k$ so that $\vartheta_k(E_k)$ is uniformly bounded.
  Moreover, the next Theorem \ref{prop:crineqonjumps} 
  shows that $\cE(t,\vartheta_k(r))\le C$ so that 
  $\vartheta_k(E_k)$ is contained in a sublevel of $\cF_0$.
  Applying Theorem \ref{Kuratowski} (notice that
  $\cR(t,\vartheta_k,E_k^+)$ is uniformly bounded)
  we can extract a subsequence converging to a limit
  transition $\vartheta\in \rmC_{\sigma,\sfd}(E,X)$ 
  with $\vartheta(E^\pm)=u^\pm$. By 
  Theorem \ref{costlsc} we have $\Cf(t,\vartheta,E)\le 
  \liminf_{k\to\infty} \Cf(t,\vartheta_k,E_k)=\sfc(t,u^-,u^+)$,
  so that $\vartheta$ is optimal.
\end{proof}
By a similar argument, we obtain 
\begin{corollary}[Lower semicontinuity of the cost $\sfc$]
  \label{cor:asymptoticcost}
  Let us assume that \emph{\mytag A{}, \mytag B1, \mytag C{}} hold,
  let $F$ be a sublevel of $\cF_0$ and 
  let $(u_k^\pm)_k\subset F$ be sequences of points converging to 
  $u^\pm$. Let $\vartheta_k\in \rmC_{\sigma,\sfd}(E_k,X)$ and
  $\sft_k:E_k\to [0,T]$
  satisfy $\vartheta_k(E_k^\pm)=u^\pm$ and
  $\lim_{k\to\infty}\sup_{s\in E_k}|\sft_k(s)-t|=0$.
  Then
  \begin{equation}
    \label{eq:139}
    \liminf_{k\to\infty}\bigg(\var(\vartheta_k,E_k)+
    \Cd(\vartheta_k,E_k)+\sum_{s\in
      E_k}\Gs(\sft_k(s),\vartheta_k(s))\bigg)\ge \sfc(t,u^-,u^+).
  \end{equation}
  In particular, if $t_k\to t$
  \begin{equation}
    \label{eq:140}
    \liminf_{k\to\infty}\sfc(t_k,u_k^-,u_k^+)\ge \sfc(t,u^-,u^+).
  \end{equation}
\end{corollary}

\section{Energy inequalities}
\label{sec:energy-inequalities}
We can now prove the energetic inequality stated in
\eqref{eq:crinqualitytheta}.
\relax
Our proof is based on the following elementary Lemma,  see \cite{Gal57} for similar arguments.
\begin{lemma}
  \label{le:elementary}
  Let $E\subset \R$ be a compact set with $E^- <E^+ $, let $L(E)$ be the
  set of 
  limit
  points of $E$. 
  We consider a function $f:E\to \R$ upper semicontinuous 
  and continuous on the left and a function $g\in \rmC(E)$ 
  strictly increasing, satisfying the 
  following two conditions:\\
  i) for every $I\in \Holes(E)$ 
  \begin{equation}
    \label{eq:65}
    \frac{f(I^+ )-f(I^- )}
    {g(I^+ )-g(I^- )}\le 1;
  \end{equation}
  ii) for every $t\in L(E)$ which is an accumulation point of
  $L(E)\cap (-\infty,t)$ we have
  \begin{equation}
    \label{eq:22}
    \liminf_{s\up t,\ s\in L(E)} \frac{f(t)-f(s)}
    {g(t)-g(s)}\le 1.
  \end{equation}
  Then the map $s\mapsto f(s)-g(s)$ is non increasing in $E$; in particular
  \begin{equation}
    \label{eq:52}
    f(E^+ )-f(E^- )\le g(E^+ )-g(E^- ).
  \end{equation}
\end{lemma}
\begin{proof}
  By replacing $E$ with $E\cap [r ,s]$, $r<s$, it is easy to see that our
  thesis is in fact equivalent to \eqref{eq:52}.
  In order to prove it, it is not restrictive to assume that
  $f(E^- )=g(E^- )=0$
  and $E$ contains at least three points (otherwise \eqref{eq:52} 
  follows by \eqref{eq:65}).
  
  We argue by contradiction, supposing that 
  $$\gamma:=
  \frac{f(E^+ )}{g(E^+ )}>1$$
  and we consider the map
  $h(s):=f(s)-\gamma g(s)$, $s\in E$.
  Since
  $h(E^- )=h(E^+ )=0,$ $h$ takes its maximum at some point $\bar s\in
  E\cap(E^- ,E^+ ]$. Since
  \begin{displaymath}
    f(s)-\gamma g(s)\le f(\bar s)-\gamma g(\bar s)\quad\text{for every
    }s\in E,
  \end{displaymath}
  we obtain
  \begin{equation}
    \label{eq:56}
    \frac{f(\bar s)-f(s)}{g(\bar s)-g(s)}\ge \gamma>1\quad\text{for
      every }s\in E\cap [E^- ,\bar s).
  \end{equation}
  \eqref{eq:65} shows that $\bar s$ cannot be the right extremum $I^+ $
  for some $I\in \Holes(E)$ and \eqref{eq:22} shows that $\bar s$ is
  isolated
  in $L(E)\cap [E^- ,\bar s]$. Therefore, there exists
  $\eps>0$ such that $(\bar s-\eps,\bar s)$ contains an increasing sequence
  $(s_n)_n$ of isolated points of $E$, converging to $\bar s$.
  Using
  \eqref{eq:65} and the fact that $(s_{n},s_{n+1})\in \Holes(E)$
  we get
  \begin{equation}
    \label{eq:64}
    f(s_{n+1})-f(s_{n})\le g(s_{n+1})-g(s_{n}).
  \end{equation}
  Summing up from $n=1$ to $N-1$ we obtain
  \begin{equation}
    \label{eq:66}
    f(s_N)-f(s_1)\le g(s_N)-g(s_1),
  \end{equation}
  and passing to the limit as $N\up\infty$ by using the left
  continuity of $f$ and the continuity of $g$ we eventually get
  \begin{equation}
    \label{eq:67}
    f(\bar s)-f(s_1)\le g(\bar s)-g(s_1)
  \end{equation}
  which is in contradiction with \eqref{eq:56}.
\end{proof}
As a corollary we obtain a ``dual'' result for functions defined on
intervals.
\begin{lemma}
  \label{le:dual}
  Let $g:[a,b]\to \R$ be strictly increasing, $f:[a,b]\to \R$ be a
  left-continuous function
  whose restriction to $[a,b]\setminus \Jump{}g$ is upper
  semicontinuous. If
  \begin{equation}
    \label{eq:68}
    \limsup_{r\down t}f(r)-f(t)\le g(t\rightl )-g(t\leftl )
    \quad\text{for every }t\in \rmJ_g,
  \end{equation}
  and 
  \begin{equation}
    \label{eq:69}
    \liminf_{s\up t}\frac{f(t)-f(s)}{g(t\leftl )-g(s\leftl )}\le 1\quad\text{for
      every }t\in [a,b],\ 
  \end{equation}
  then the map $t\mapsto f(t)-g(t-)$ is non increasing.
\end{lemma}
\begin{proof}
  Possibly replacing $g$ with $t\mapsto g(t-)$
  it is not restrictive to assume that $g$ is left-continuous.
  Let us define $E$ as the closure
  of $Z:=g([a,b])$. We denote by $\sfs:E\mapsto [a,b]$ the continuous 
  map whose restriction to $Z$ coincides with $g^{-1}$;
  $\tilde f=f\circ \sfs$ is left continuous and upper
  semicontinuous in $Z$; notice moreover that every $I\in \Holes(E)$ 
  is of the form $(g(t),g(t\rightl ))$ for some $t\in \Jump{}g$.
  Thus if $z\in E\setminus Z$ there exists a unique $t\in \Jump{}g$ 
  such that $z=g(t\rightl )$. 
  We set $\tilde
  f(z):=\limsup_{r\down t}f(r)$.
  
  Defining $\tilde g(r):=r$, $r\in E$, 
  it is then easy to check that we can apply Lemma \ref{le:elementary}
  to the couple of functions $\tilde f,\tilde g$ obtaining
  that $r\mapsto h(r)=\tilde f(r)-r$ is nonincreasing in $E$.
  Thus composing with $g$ we get $t\mapsto f(t)-g(t)=h\circ g$
  is non increasing.
\end{proof}
\begin{theorem} \label{prop:crineqonjumps} 
  Suppose that \emph{Assumption \mytag B{}} hold. 
  For every $t\in[0,T]$ and $u\ulrl\in X$ we have
\begin{equation}
  \cE(t,u\urightl )+\Fd(t,u\uleftl ,u\urightl )\geq \cE(t,u\uleftl ).
\end{equation}
\end{theorem}
\begin{proof}
If $\Fd(t,u\uleftl ,u\urightl )=+\infty$ the inequality is trivial. Otherwise, let
$E$ be a compact subset of $\mathbb{R}$,
$\vartheta\in \rmC_{\sigma,\sfd}(E,X)$ a continuous map such that
$\vartheta(E\ulrl)=u\ulrl$ and $\Cf(t,\vartheta)<+\infty$. 
We want to apply the previous Lemma \ref{le:elementary} with the
choices
\begin{displaymath}
  f(s):=-\cE(t,\vartheta(s)),\quad
  g(s):=\Cf(t,\vartheta;E\cap[E^- ,s]).
\end{displaymath}
Notice that $g$ is continuous since $\vartheta\in
\rmC_{\sigma,\sfd}(E,X)$ and $f$ is upper semicontinuous thanks to the lower
semicontinuity of $\cE$ and the continuity of $\vartheta$;
$f$ is also left continuous: whenever $s_n\up s$ is an
increasing sequence in $E$, we have
$\sfd(\vartheta(s_n),\vartheta(s))\to0$ and the property
$\sum_n\cR(t,\vartheta(s_n))<\infty$ shows that
$\lim_{n\to\infty}\cR(t,\vartheta(s_n))=0$, so that
we obtain $\cE(t,\vartheta(s_n))\to\cE(t,\vartheta(s))$
thanks
to Lemma \ref{le:usefulR} v) and \mytag B2.
It remains to check conditions \eqref{eq:65} and \eqref{eq:22}

\eqref{eq:65} 
follows from the definition of $\Gs$, since
\[
\cE(t,\vartheta(I^+ ))+\Gs(t,\vartheta(I^- ))+\sfD(\vartheta(I^- ),\vartheta(I^+ ))\geq \cE(t,\vartheta(I^- )),
\]
and the fact that 
$$\Cf(t,\vartheta,E\cap
  [E^- ,I^+ ])-\Cf(t,\vartheta,E\cap
  [E^- ,I^- ])=
  \Gs(t,\vartheta(I^- ))+\sfD(\vartheta(I^- ),\vartheta(I^+ )).$$
\eqref{eq:22} is a direct consequence of \eqref{D.3'} and of the
inequality 
$$\sfd(\vartheta(r),\vartheta(s))\le \Cf(t,\vartheta,E\cap
[E^- ,s])-\Cf(t,\vartheta,E\cap
  [E^- ,r])=g(s)-g(r);$$ 
recall that the set $E_\cR=\{r:\cR(t,\vartheta(r))>0\}$ is discrete,
so that for every point $s\in L(E)$ we have $\vartheta(s)\in \SSD(t)$
and \eqref{D.3'} can be applied.
To conclude the proof it is sufficient to take the infimum over admissible curves $\vartheta$.
\end{proof}
A direct consequence of Proposition \ref{prop:crineqonjumps} is a description of the behaviour of a bounded variation curve on its jump points. 
\begin{corollary} \label{cor:crinequality} Let $u\in \BV{\sigma,\sfd}([0,T];X)$. Then for every $t\in \Ju$ the following inequalities hold: 
\begin{gather} \label{gtr:crineqeq1} 
  \begin{aligned}
\cE(t,u(t\rightl ))+\Fd(t,u(t ),u(t\rightl ))&\ge \cE(t,u(t)),  \\
\cE(t,u(t))+\Fd(t,u(t\leftl),u(t ))&\ge \cE(t,u(t\leftl)),\\
\cE(t,u(t\rightl ))+\Fd(t,u(t\leftl ),u(t\rightl ))&\ge \cE(t,u(t\leftl)).
\end{aligned}
\end{gather}
\end{corollary}
The chain rule inequality \eqref{eq:crinequality} is a consequence of
\eqref{gtr:crineqeq1}. Indeed, we can recover the inequality also at
the continuity points of $u$ with a similar trick, applying Lemma \ref{le:dual}.
\begin{theorem} \label{thm:crinequality} 
  Let us suppose that \emph{\mytag B{}} and \emph{\mytag A1}, \eqref{A.22}, \eqref{A.21},
  \eqref{eq:135} hold (these properties are verified if \emph{\mytag
    A2} or
  \emph{\mytag A{2'}} hold). 
Let $u\in \BV{\sigma,\sfd}([0,T];X)$ satisfy \eqref{stability}. Then for every $0\leq t_0\leq t_1\leq T$ the following inequality holds:
\begin{equation}
\mathcal{E}(t_1,u(t_1))+ \varC(u,[t_0,t_1])\geq\mathcal{E}(t_0,u(t_0))+\int_{t_0}^{t_1}\cP(s,u(s))\rmd s. 
\end{equation}
\end{theorem}

\begin{proof}
We will apply Lemma \ref{le:dual} in the interval $[0,T]$ with the choices
\begin{displaymath}
  f(t):=\int_{0}^{t}\cP(s,u(s))\rmd s-\cE(t,u(t\leftl )),\quad
  g(t):=\varC(u,[0,t])+\eps t,
\end{displaymath}
for a small parameter $\eps>0$.

Since $u$ is continuous in $[0,T]\setminus \Jump{} g$ 
the upper semicontinuity of $f$ outside $\Jump{}g$ is guaranteed by the 
lower semicontinuity of $\cE$. Its left continuity is a consequence of 
the stability property \eqref{stability} of $u$, of Lemma
\ref{le:usefulR} v) and of \mytag B2.

Condition \eqref{eq:68} is satisfied thanks to Corollary
\ref{cor:crinequality} and
the fact that at every $t\in \Jump{}u$
\begin{displaymath}
  \limsup_{r\down t}f(r)
  \le \int_0^t\cP(s,u(s))\, \rmd s -\cE(t,u(t\rightl )),\quad
  g(t\rightl )-g(t\leftl )=\sfc(t,u(t\leftl
  ),u(t))+\sfc(t,u(t),u(t\rightl )).
\end{displaymath}
In order to check \eqref{eq:69} let us fix a couple of times $s,t\in [0,T]$ with $s<t$ and observe that
\begin{align*}
  f(t)-f(s)&=\cE(s,u(s\leftl ))-\cE(t,u(t\leftl ))+\int_s^t \cP(r,u(r))\,\d r=
                 \\&=
                     \cE(s,u(s\leftl ))-\cE(s,u(t\leftl ))+\cE(s,u(t\leftl))-\cE(t,u(t\leftl))+\int_s^t
                     \cP(r,u(r))
                     \,\d r,\\
  g(t\leftl )-g(s\leftl )&\ge \sfd(u(s\leftl ),u(t\leftl ))+\eps(t-s).                 
\end{align*}
The conditional upper semi-continuity of $\cP$ \eqref{eq:135} 
(recall that the energy is left-continuous) and \eqref{A.22} yield
\begin{align*}
  \limsup_{s\up t}&
                    \frac{1}{\eps(t-s)}\Big(\cE(s,u(t\leftl))-\cE(t,u(t\leftl))+\int_s^t
                    \cP(r,u(r))
                    \,\d
                    r\Big)
  \\&\le
      \frac 1\eps\bigg(-\liminf_{s\up t}
                    \frac{\cE(t,u(t\leftl))-\cE(s,u(t\leftl))}{(t-s)}+
      \limsup_{s\up t}\media_s^t
                    \cP(r,u(r))
                    \,\d
                    r\bigg)
      \\&\le 
          \frac 1\eps\Big(-\cP(t,u(t\leftl))+\cP(t,u(t\leftl))\Big)\le 0.
\end{align*}
On the other hand, 
from assumption \mytag B3 and the stability property \eqref{stability} we have
\[
\limsup_{s\up t}\frac{\cE(s,u(s\leftl ))-\cE(s,u(t\leftl ))}{\sfd(u(s- ),u(t\leftl ))}\le 1.
\]
\end{proof}

\section{Convergence proof for the discrete approximations}
\label{sec:convergenceproof}
In this section we will prove existence of a Visco-Energetic solution,
stated in Theorem \ref{thm:existence}. 
We will always suppose that the energy $\cE$ satisfies assumptions
\mytag A{} (where we will also consider the case \mytag A{2'}), that
the viscous correction $\delta$ is admissible according to
\mytag B{},
and that conditions \mytag C{} hold.

\subsection{Discrete estimates}
\label{subsec:discrete-estimates} 
Hereafter, $\tau$  will be a given partition of $[0,T]$.
 We obtain some preliminary estimates for the minimizing movement scheme.
\begin{theorem}[Discrete estimates] \label{thm:discretestimates} 
Let $U^0_\tau\in X$ be given so that 
\begin{equation}
  \label{eq:87}
  \cF_0(U^0_\tau)=\cE(0,U^0_\tau)+\sfd(x_o,U^0_\tau)+ F_o \le C_0.
\end{equation}
Then every solution $U_\tau^n$ of the incremental problem \eqref{ims}
starting from $U_\tau^0$ satisfies 
a discrete version of stability \eqref{stability} and energy balance \eqref{energybalance}, namely for every $n=1,\dots, N$ we have
\begin{gather}
\label{eq:29}
\mathcal{E}(t^n_\tau,U_\tau^n)\leq
\mathcal{E}(t^n_\tau,V)+\sfd(U_\tau^n,V)+\delta(U_\tau^{n-1},V),
\\
\label{eq:30}
\mathcal{E}(t^n_\tau,U_\tau^n)+\sfD(U_\tau^{n-1},U_\tau^n)
+\Gs(t^n_\tau,U_\tau^{n-1})= \cE(t^{n-1}_\tau, U_\tau^{n-1})+
\int_{t^{n-1}_\tau}^{t^n_\tau}\cP(s,U_\tau^{n-1})\rmd s. 
\end{gather}
Moreover, there exist constants $C_1,C_2$ depending only on $C_0,F_o$ 
(of
\eqref{eq:87}), on $C_P$ (of \eqref{A.21}),
and on $T$,
such that 
\begin{gather} \label{eq:aprioriestimates}
  \cF(t^n_\tau, U_\tau^n)\leq C_0
  \rme^{C_P t^n_\tau}\le C_0 \rme^{C_PT},\quad
  \sfd(x_o,U^n_\tau)\le C_1, \\
  \label{eq:88}
  \sum_{j=1}^N \sfD(U_\tau^{j-1},U_\tau^j)+
  \cR(t^j_\tau,U^{j-1}_\tau)\leq  C_2.
\end{gather}
\end{theorem}
\begin{proof} Since $U^n_\tau$ is a minimizer for \eqref{ims}, the estimate
\[
\mathcal{E}(t^n_\tau,U_\tau^n)+\sfd(U_{\tau}^{n-1},U_\tau^{n})+\delta(U_\tau^{n-1},U_\tau^{n})\leq \mathcal{E}(t^n_\tau,V)+\sfd(U_\tau^{n-1},V)+\delta(U_\tau^{n-1},V),
\]
holds for every $V\in X$. Using the triangle inequality and
$\delta(U_\tau^{n-1},U_\tau^{n})\geq 0$, we have proved the discrete
stability
\eqref{eq:29}.

From the minimality of $U_\tau^n$ and the definition of $\cR$ we have:
\[
\Gs(t^n_\tau,U_\tau^{n-1})=\mathcal{E}(t^n_\tau,U_\tau^{n-1})-\mathcal{E}(t^n_\tau,U_\tau^n)-\sfD(U_\tau^{n-1},U_\tau^n)
\]
and since 
\[
\cE(t^n_\tau,U_\tau^{n-1})=\cE(t^{n-1}_\tau,U_\tau^{n-1})+\int_{t^{n-1}_\tau}^{t^n_\tau}\cP(s,U^{n-1}_\tau)\rmd s
\]
we have also proved the discrete energy balance \eqref{eq:30}.

Using \mytag A1 and 
\eqref{eq:gronwallestimate} in the power term and denoting by
$\tau^n:=t^n_\tau-t^{n-1}_\tau$ we get
\[
\int_{t^{n-1}_\tau}^{t^n}\cP(s,U^{n-1}_\tau) \rmd s\leq 
\Big(\sfd(x_o,U^{n-1}_\tau)+\cE(t^{n-1}_\tau,U^{n-1}_\tau)+F_o\Big)(\mathrm{e}^{ C_P\tau^n}-1).
\]
Then summing up $\sfd(x_o,U^{n-1}_\tau)+F_o$ to both terms of 
the inequality \eqref{eq:30} and using the triangle inequality \eqref{eq:16} we have
\[
\cE(t^n_\tau,U^n_\tau)+\sfd(x_o,U_\tau^n)+F_o\leq
\Big(\sfd(x_o,U^{n-1}_\tau)+\cE(t^{n-1}_\tau,U^{n-1}_\tau)+F_o
\Big)
\mathrm{e}^{C_P\tau^n}.
\]
A simple induction argument yields
\begin{displaymath}
\cE(t^n_\tau,U^n_\tau)+\sfd(x_o,U_\tau^n) +F_o\leq
\Big(\cE(0,U^0_\tau)+\sfd(x_o,U^0_\tau)+F_o\Big)
\mathrm{e}^{ C_Pt^n_\tau}.
\end{displaymath}
This also yields $\sfd(x_o,U^n_\tau)\le C_1$
where $C_1:=\sup\big\{\sfd(x_o,v):\cF_0(v)\le C_0\rme^{2 C_p T}\big\}$.

Finally, 
we estimate the dissipated energy via
\begin{align*}
  \sum_{j=1}^N \sfD(U_\tau^{j-1},U_\tau^j)&+
  \cR(t^j_\tau,U^{j-1}_\tau)
  \\&
      \le \cF_0(U^0_\tau)-\cF(t^N_\tau,U_\tau^N)+\sum_{j=1}^N
    \cF(t_\tau^{j-1},U_\tau^{j-1} )\big(\rme^{C_p\tau^j}-1\big)+\sfd(x_o,U^N_\tau)\\
  &\le \cF_0(U^0_\tau) +\cF_0(U^0_\tau)
  \sum_{j=1}^N(\mathrm{e}^{ C_Pt^j_\tau}-\mathrm{e}^{ C_P t^{j-1}_\tau})+C_1
  \le C_0
    \mathrm{e}^{ C_PT}+C_1
\end{align*}
and the proof is complete with $C_2:=C_0\mathrm{e}^{C_PT}+C_1.$
\end{proof}
\subsection{Compactness}
\label{subsec:compactness-proof}
We introduce the functions
\begin{equation}
  \label{eq:80}
  \sft_\tau(t):=t^n_\tau,\quad
  \tilde\sft_\tau(t):=t^{n+1}_\tau,\quad \overline{U}_\tau(t)=U^n_\tau,\quad
  \sfn_\tau(t):=n\quad
  \quad\text{whenever }t\in (t^{n-1}_\tau,t^n_\tau],
\end{equation}
so that \eqref{eq:30} can be rewritten as
\begin{equation}
  \label{eq:89}
  \begin{aligned}
    \cE(\sft_\tau(t),\overline U_\tau(t)) +\var(\overline
    U_\tau,[s,t])&+
    \sum_{j=\sfn_\tau(s)}^{\sfn_\tau(t)-1}\delta(U^{j}_\tau,U^{j+1}_\tau)+
    \cR(t^{j+1}_\tau,U^{j}_\tau)
    \\&= \cE(\sft_\tau(s),\overline U_\tau(s))+
    \int_{\sft_\tau(s)}^{\sft_\tau(t)}\cP(r,\overline
    U_\tau(r))\,\d r.
  \end{aligned}
\end{equation}
Notice that the variation function $V_\tau$ associated with $\overline
U_\tau$ can be written as
\begin{equation}
  \label{eq:97}
  V_\tau(t):=\var(\overline U_\tau,[0,t])=
  \sum_{j=0}^{\sfn_\tau(t)-1}\sfd(U^j_\tau,U^{j+1}_\tau).
\end{equation}
Similarly, we introduce the nondecreasing function $W_\tau:[0,T]\to [0,\infty)$ 
\begin{equation}
  \label{eq:93}
  W_\tau(t):=\sum_{j=0}^{\sfn_\tau(t)-1}\sfD(U^j_\tau,U^{j+1}_\tau)+\cR(t^{j+1}_\tau,U^j_\tau)
  =V_\tau(t)+\sum_{j=0}^{\sfn_\tau(t)-1}\delta(U^j_\tau,U^{j+1}_\tau)+\cR(t^{j+1}_\tau,U^j_\tau).
\end{equation}
Notice that $W_\tau-V_\tau$ is still a nonnegative and nondecreasing function.
\begin{theorem}[Compactness] \label{thm:compactness} Let $u_0\in X$ be
  fixed and let $(\overline{U}\kern-1pt_\tau)$ be a family of
  piecewise constant left-continuous
  interpolants of the discrete solutions $U^n_\tau$ of \eqref{ims} starting from $U^0_\tau\in X$, with
\begin{equation} \label{eq:initialdatum}
  \cF_0(U^0_\tau)\le C_0,\quad
  U^0_\tau\sigmato u_0\text{ in $X$},\quad \cE(0,U^0_\tau)\to \cE(0,u_0) \text{ as $\tau\downarrow0$}.
\end{equation}
Let $V_\tau,W_\tau$ be defined as in 
\eqref{eq:97} and \eqref{eq:93}.
Then for all sequences of partitions
$k\mapsto \tau(k)$ with $\lim_{k\to\infty}|\tau(k)|=0$ there exist
\\
- a (not relabeled) subsequence $k\mapsto \tau(k)$, \\
- a
 limit curve $u\in  \BV{\sigma,\sfd}([0,T];X)$,
\\
- nondecreasing functions ${V},W:[0,T]\rightarrow [0,+\infty)$ with
$W-V$ nondecreasing,\\
- a real function $\sfE\in \BV{}([0,T])$,\\
- a set
 $\CC\subset [0,T]$ 
 with $\Leb 1([0,T]\setminus
 \CC)=0$\\
such that 
 \begin{equation}
   \label{eq:81}
   V_{\tau(k)}(t)\rightarrow {V}(t),\quad
   W_{\tau(k)}(t)\rightarrow {W}(t),\quad
   \text{for every }t\in [0,T],
 \end{equation}
\begin{equation}
  \label{eq:98}
  \cR(\tilde\sft_{\tau(k)}(t)
  ,\overline U_{\tau(k)}(t))\to 0\quad\text{for
    every }t\in \CC,
\end{equation}
\begin{equation}
  \label{eq:82}
  \overline{U}\kern-1pt_{\tau(k)}(t)\sigmato u(t)\quad\text{for
    every }t\in \CC \cup\Jump{}W
\end{equation}
 \begin{equation}
   \label{eq:83}
   \sfd(u(s),u(t))\le V(t)-V(s),\quad
   \text{for
     every $0\le s\le t\le T,$}
 \end{equation}
 \begin{equation}
   \label{eq:84}
   u(t)\in \SSD(t)
   \quad\text{for every }t\not\in \Jump{}V,
 \end{equation}
 \begin{equation}
   \label{eq:99}
   \cE(\sft_{\tau(k)}(t),\overline U\kern-1pt_{\tau(k)}(t))\to
   \sfE(t)\ge \cE(t,u(t)),\quad
   \text{for every }t\in [0,T],
 \end{equation}
 where 
 \begin{equation}
   \label{eq:155}
   \sfE(t)=\cE(t,u(t))\quad\text{for every $t\in \CC
     $
   }
   \quad\text{if \emph{\mytag A{2'}} holds},
 \end{equation}
 \begin{equation}
   \label{eq:100}
   \sfE(t)+W(t)\le\sfE(s)+W(s)+\int_s^t\cP(r,u(t))\,\d r
   \quad\text{for every }0\le s\le t\le T.
 \end{equation}
 Moreover, for every further subsequence $k\mapsto\tau'(k)$ 
 \begin{equation}
   \label{eq:85}
   \lim_{k\to\infty}\cR(\tilde\sft_{\tau'(k)}(t),\overline
   U_{\tau'(k)}(t))=0\quad\Rightarrow\quad
   \lim_{k\to\infty}\overline U_{\tau'(k)}(t)=u(t)\quad
   \text{for every }t\not\in  \Jump{}V.
 \end{equation}
 Finally, for every $t\in \Jump{}{W}$ 
 there exist sequences $(t_k\ulrl)_k$ such that 
 $t_k\uleftl \up t$, $t_k\urightl \down t$ as $k\up\infty$ and
 \begin{equation}
   \label{eq:96}
   V_{\tau(k)}(t_k\ulrl)\to V(t\lrl),\quad
   W_{\tau(k)}(t_k\ulrl)\to W(t\lrl),\quad
   \overline U_{\tau(k)}(t_k\ulrl)\sigmato u(t\lrl).
 \end{equation}
\end{theorem}
\begin{proof}
Let us first observe that 
the values of $\overline U_\tau$ belong to the bounded sequentially
compact set $F:=\{v\in X:\cF_0(v)\le C_0\rme^{C_PT}\}$.
\\[6pt]\noindent 
$\vartriangleright$ \eqref{eq:81}: 
\eqref{eq:88} shows that the functions $W_\tau$ (and thus a fortiori
$V_\tau$)
are uniformly bounded in $[0,T]$,
so that Helly's theorem provides pointwise convergence (up to
a subsequence)
of $W_{\tau(k)},V_{\tau(k)}$ to some increasing functions $W,V$,
with $W-V$ also increasing and $\Jump{}V\subset \Jump{}W$.
\\[6pt]\noindent 
$\vartriangleright$ \eqref{eq:98}: 
\eqref{eq:88} yields
\begin{equation}
  \label{eq:90}
  \int_0^{T-|\tau|}\cR(\tilde\sft_\tau(s),\overline U_\tau(s))\,\d s\le C_2|\tau|,
\end{equation}
so that there exists a subsequence $k\mapsto \tau(k)$ 
and a subset $\CC\subset [0,T]$ of full measure such that
\eqref{eq:98} holds.
%
\\[6pt]\noindent 
$\vartriangleright$ \eqref{eq:82}, \eqref{eq:83}, \eqref{eq:84}, \eqref{eq:85}: 
By a standard diagonal argument, we can choose a dense countable set
$\CC_1\subset \CC,$ 
a subsequence (still denoted by
$\tau(k)$), and a limit function $u:\CC_1\cup \Jump{}W\to X$
such that 
\begin{displaymath}
  u_{\tau(k)}(t)\sigmato u(t),\quad
  \sfd(u(s),u(t))\le V(t)-V(s)\quad
  \text{for every }0\le s\le t\le T,\ s,t\in \CC_1\cup \Jump{}W.
\end{displaymath}
Clearly $\Jump{}{V_u}\subset \Jump{}V\subset \Jump{}W$ and 
$u(t)\in \SSD(t)$ for every $t\in \CC_1$ 
by \eqref{eq:98} and the lower semicontinuity of $\cR$.
Since $\SSD$ is separated by $\sfdp$, applying Lemma \ref{le:nondeg-regulated}
we can extend $u$ to a function (still denoted by $u$)
in $\BV{\sigma,\sfd}([0,T];X)$. The closure of $\SSD$ and the fact
that $\Jump{}{V_u}\subset \Jump{}V$ yield \eqref{eq:84}.

We can also prove that $\overline
U_{\tau(k)}(t)\to u(t)$ for every $t\in \CC\setminus \Jump{}V$: in fact, \eqref{eq:98}
and the lower semicontinuity of $\cR$
show that any limit point of the sequence $\{\overline
U_{\tau(k)}(t)\}_{k\in \N}$ is contained in $\SSD(t)$, which is
separated by $\sfd$.  If $v$ is an arbitrary limit point,
passing to the limit in the inequalities $\sfd(\overline
U_{\tau(k)}(s),\overline
U_{\tau(k)}(t) )\le V_k(t)-V_k(s)$ we get
$\sfd(u(s),v)\le V(t)-V(s)$ for every $s\in \CC_1$; passing to the
limit as $s\up t$, $s\in \CC_1$, we conclude that $\sfd(u(t),v)=0$ 
which yields $v=u(t)$ since $\SSD(t)$ is separated by $\sfd$.
The same argument yields \eqref{eq:85}.
\\[6pt]\noindent 
$\vartriangleright$ \eqref{eq:99}:  
We notice that \eqref{eq:89} yields for the constant $C:=C_PC_0\exp(C_PT)$
\begin{equation}
  \label{eq:94}
  \cE(\sft_\tau(t),\overline U_\tau(t)) +W_\tau(t)+C\, \sft_\tau(t)\le 
  \cE(\sft_\tau(s),\overline U_\tau(s)) +W_\tau(s)+C\, \sft_\tau(s)
\end{equation}
whenever $0\le s\le t\le T$. Since $W_{\tau(k)}(t)\to W(t)$ and 
$\sft_{\tau(k)}(t)\to t$ as $k\to\infty$, a further application of
Helly's Theorem (and a further extraction of a subsequence) yields
$\cE(\sft_{\tau(k)}(t),\overline U_{\tau(k)}(t))\to \sfE(t)$ for every
$t\in [0,T]$ and
\begin{equation}
  \label{eq:95}
  \sfE(t)\ge \cE(t,u(t))\quad\text{for every }t\in \CC\cup\Jump{}W,
\end{equation}
thanks to the lower semicontinuity of $\cE$. 
Since the uniform bound
\begin{equation}
  \label{eq:156}
  |\sfE(t)-\sfE(s)|\le W(t)-W(s)+C(t-s)
\end{equation}
obtained by passing to the limit in \eqref{eq:89} shows that
$\sfE$ is continuous outside $\Jump{}{W}$, 
we conclude that $\sfE(t)\ge \cE(t,u(t))$ holds everywhere in $[0,T]$.
\\[6pt]\noindent 
$\vartriangleright$ \eqref{eq:155}:
let us first notice that 
\begin{displaymath}
  \sfE(t)\le \limsup_{k\to\infty}
  \Big(\cE(\tilde\sft_{\tau(k)}(t),\overline
  U_{\tau(k)}(t))+C|\tau(k)|\Big)=
  \limsup_{k\to\infty}
  \cE(\tilde\sft_{\tau(k)}(t),\overline
  U_{\tau(k)}(t)).
\end{displaymath}
If \mytag A{2'} holds, by \eqref{eq:82}, \eqref{eq:98}, and Lemma
\ref{le:usefulR} v) we get 
\begin{equation}
  \label{eq:157}
  \lim_{k\to\infty}
  \cE(\tilde\sft_{\tau(k)}(t),\overline
  U_{\tau(k)}(t))= \cE(t,u(t)).
\end{equation}
%
$\vartriangleright$ \eqref{eq:100}:  
since $\overline
U_{\tau(k)}(t)\to u(t)$ for almost every $t\in [0,T]$, the upper
semicontinuity 
of $\cP$ and the fact that $\overline
U_{\tau(k)}(t)$ is contained in a sublevel of $\cF_0$ yield
\begin{equation}
  \label{eq:92}
  \limsup_{k\to\infty}\int_{\sft_\tau(s)}^{\sft_\tau(t)}\cP(r,\overline
  U_{\tau(k)}(r))\,\d r\le \int_s^t \cP(r,u(r))\,\d r\quad
  \text{for every }0\le s<t\le T.
\end{equation}
The same conclusion holds if we assume \mytag A{2'} instead of \mytag
A2, thanks to \eqref{eq:155}.
\eqref{eq:100} then follows by \eqref{eq:89}.
\\[6pt]\noindent 
$\vartriangleright$ \eqref{eq:96}:
this is a general property of double limits; let us check the case of
$t_k\leftl $.
We first select a 
fundamental sequence of open neighborhoods $N_n$
of $u(t\leftl )$, a decreasing vanishing sequence $\eps_n$
and an increasing sequence $(t_n)_n$ in $[0,t)\cap \CC_1$ so that
$t_n\up t$ as $n\up\infty$,
$u(t_n)\in N_n$ and $W(t)-W(t_n)<\eps_n$, $\overline{U}_{\tau(k)}(t_n)\to 
u(t_n)$, $W_{\tau(k)}(t_n)\to W(t_n)$ as $k\to\infty$.
We may find a strictly increasing sequence $n\mapsto \kappa(n)$ such that 
\begin{displaymath}
  |W_{\tau(k)}(t_n)-W(t)|<\eps_n,\quad
  \overline{U}_{\tau(k)}(t_n)\in N_n\quad\text{for every }k\ge \kappa(n).
\end{displaymath}
For every $k\ge \kappa(1)$ we can then define $n(k):=\min\big\{m\in
\N: k\ge \kappa(m)\big\} $; it is not difficult to check that
$n(k)\up\infty$ and 
the sequence $k\mapsto t_k^- :=t_{n(k)}$ satisfies \eqref{eq:96}.
\end{proof}

\subsection{Limit energy-dissipation inequality}
\label{subsec:limit-dissipation-inequality}
We can now prove the energy inequality on jumps.
\begin{lemma}
  \label{le:jump-inequality}
  Let $u_0$, $\overline U_{\tau(k)}$, $u$, $W,\sfE$ be as in the previous Theorem
  \ref{thm:compactness}.
  Then for every $t\in \Jump{}W$ we have
  \begin{equation}
    \label{eq:103}
    W(t)-W(t\leftl )\ge \sfc(t,u(t\leftl ),u(t)),\quad
    W(t\rightl )-W(t)\ge \sfc(t,u(t),u(t\rightl )).
  \end{equation}
\end{lemma}
\begin{proof}
  We will prove the first inequality of \eqref{eq:103}; the proof of
  the second
  inequality is completely analogous.

  Let us fix $t\in \Jump{}W$ and let us choose a sequence $t_k^- \up t$
  as in \eqref{eq:96} so that
  \begin{displaymath}
    \overline U_{\tau(k)}(t_k^- )\sigmato u(t\leftl ),\quad
    \overline U_{\tau(k)}(t)\sigmato u(t).
  \end{displaymath}
  For every $k\in \N$ we consider the compact set
  $E_k:=\{n\in \N:\sft_{\tau(k)}(t_k\uleftl )\le n\le \sft_{\tau(k)}(t)\}$
  and the discrete transition
  $\vartheta_k:E_k\to X$ defined by
  $\vartheta_k(n):=U^n_{\tau(k)}$, $n\in E_k$.
  By construction $\vartheta_k(E_k\uleftl )=\overline U_{\tau(k)}(t_k\uleftl )$
  and $\vartheta_k(E_k\urightl )=\overline U_{\tau(k)}(t)$.
  Moreover
  \begin{equation}
    \label{eq:104}
    \var(\vartheta_k,E_k)=\sum_{n\in E_k\setminus E_k\uleftl }
    \sfd(U^{n-1}_{\tau(k)},U^n_{\tau(k)}),\quad
    \Cd(\vartheta_k,E_k)=\sum_{n\in E_k\setminus E_k\uleftl }
    \delta(U^{n-1}_{\tau(k)},U^n_{\tau(k)})
  \end{equation}
  and
  \begin{equation}
    \label{eq:105}
    \sum_{n\in E_k\setminus E_k\urightl }\cR(t^{n+1}_{\tau(k)},\vartheta_k(n))=
    \sum_{n\in E_k\setminus E_k\urightl }\cR(t^{n+1}_{\tau(k)},U^n_{\tau(k)})
  \end{equation}
  so that 
  \begin{equation}
    \label{eq:106}
    \var(\vartheta_k,E_k)+\Cd(\vartheta_k,E_k)+\sum_{n\in E_k\setminus
      E_k\urightl }\cR(t^{n+1}_{\tau(k)},\vartheta_k(n))
    =W_{\tau(k)}(t)-W_{\tau(k)}(t_k\uleftl ).
  \end{equation}
  Passing to the limit and recalling Corollary \ref{cor:asymptoticcost} we conclude.
\end{proof}
\begin{corollary}
  \label{cor:energy-inequality}
  Let $u_0$, $\overline U_{\tau(k)}$, $u$, $V,W,\sfE$ be as in the previous Theorem
  \ref{thm:compactness}.
  Then
  \begin{equation}
    \label{eq:101}
    V(t)-V(s)\ge \var(u,[s,t]),\quad
    W(t)-W(s)\ge \varC(u,[s,t]).
  \end{equation}
  In particular
  \begin{equation}
    \label{eq:102}
    \cE(T,u(T))+\varC(u;[0,T])\le \cE(0,u_0)+\int_0^T\cP(r,u(r))\,\d r.
  \end{equation}
\end{corollary}
\begin{proof}
  The first inequality of \eqref{eq:101} immediately follows from
  \eqref{eq:83}.

  We now consider an arbitrary ordered finite subset
  $\{t_1,t_2,\cdots,t_N\}$ of $\Jump{}W$. 
  Using the additivity of the total variation and the fact that
  $V(t)-V(s)\le W(t)-W(s)$ we have 
  \begin{align*}
    \var(&u,[0,T])+\sum_{j=1}^N\Delta_\sfc(t,u(t_j\leftl ),u(t_j))+
    \Delta_\sfc(t,u(t_j),u(t_j\rightl ))
                   \\&\le V(t_1\leftl )-V(0)+V(T)-V(t_N\rightl )+
                   \sum_{j=1}^{N-1} V(t_{j+1}\leftl )-V(t_{j}\rightl )
                  \\ &\qquad+\sum_{j=1}^N\sfc(t,u(t_j\leftl ),u(t_j))+
                   \sfc(t,u(t_j),u(t_j\rightl ))
                       \\&\le V(t_1\leftl )-V(0)+V(T)-V(t_N\rightl )+
                   \sum_{j=1}^{N-1} V(t_{j+1}\leftl )-V(t_{j}\rightl )
   \\&\qquad +\sum_{j=1}^NW(t_j\rightl )-W(t_j\leftl )
                           \\&\le W(t_1\leftl )-W(0)+W(T)-W(t_N\rightl )+
                   \sum_{j=1}^{N-1} W(t_{j+1}\leftl )-W(t_{j}\rightl )
    \\&\qquad+\sum_{j=1}^NW(t_j\rightl )-W(t_j\leftl))
        =W(T)-W(0). 
  \end{align*} 
  Taking the supremum with respect to all the finite subsets of
  $\Jump{}W$ we conclude.
\end{proof}
\subsection{Convergence: proof of Theorem \ref{thm:existence}}
\label{subsec:convergence}
We can now conclude the proof of our main Theorem \ref{thm:existence}.

Let $\overline U_\tau$ be a family of piecewise constant left-continuous
interpolants 
of the values $U^n_\tau$ of the incremental minimization scheme
\eqref{ims}
with $U^0_\tau$ satisfying \eqref{eq:initialdatum}
and let $k\mapsto \tau(k)$ be any sequence of partitions
with $|\tau(k)|\to 0$ as $k\to\infty$.

By Theorem \ref{thm:compactness} we can extract a subsequence
(not relabeled) such that $\overline U_{\tau(k)}$ pointwise converges to 
a function $u\in \BV{\sigma,\sfd}([0,T];X)$ 
in a set $\CC$ containing $\Jump{}u$ and 
with $\Leb 1\big([0,T]\setminus \CC\big)=0$.
$u$ satisfies the stability condition \eqref{stability}
by \eqref{eq:83}-\eqref{eq:84} and the
energy inequality \eqref{leqinequality} 
by Corollary \ref{cor:energy-inequality}: 
applying Proposition \ref{prop:leqinequality} we conclude
that $u$ is a VE solution to $(X,\cE,\sfd,\delta)$.

Combining \eqref{eq:99}, \eqref{eq:100} with $s=0$ and \eqref{eq:101}
we deduce that $\lim_{k\to\infty}\cE(\sft_{\tau(k)}(t),\overline
U_{\tau(k)}(t))=\sfE(t)=\cE(t,u(t))$ for 
every $t\in [0,T]$. In particular $t\mapsto \cE(t,u(t))$ is continuous
in $[0,T]\setminus \Jump{}V$.

Let us now prove that $\overline U_{\tau(k)}(t)\sigmato u(t)$ for every
$t\in [0,T]$; the thesis is already true in $\CC$, so we pick a point
$t\not\in \CC$ (in particular $t\not\in \Jump{}{V}$) and we want to show
that 
any limit $u'$ of a converging subsequence $\overline U_{\tau'(k)}(t)$ coincides with $u(t)$.
For every $r,s\in \CC$ with $r<t<s$ the lower semicontinuity of $\sfd$ yields
\begin{displaymath}
  \sfd(u(r),u')\le V(t)-V(r),\quad
  \sfd(u',u(s))\le V(s)-V(t)
\end{displaymath}
so that passing to the limit as $r\up t$ and $s\down t$ we get 
$\sfd(u(t),u')=\sfd(u',u(t))=0$; by the triangle inequality and \mytag B1
we have
\begin{displaymath}
  \sfd(u(t),v)\le \sfd(u',v),\quad \delta(u(t),v)\le \delta(u',v).
\end{displaymath}
Since $\cE(t,u(t))=\sfE(t)\ge \cE(t,u')$ and $u(t)\in \SSD(t)$ we get
for every $v\in X$
\begin{displaymath}
  \cE(t,u')\le \cE(t,u(t))\le \cE(t,v)+\sfd(u(t),v)+\delta(u(t),v)\le \cE(t,v)+\sfd(u',v)+\delta(u',v)
\end{displaymath}
so that $u'\in \SSD(t)$. Since $\sfd$ separates $\SSD(t)$ we conclude
that $u'=u(t)$.


\subsection{A uniform BV estimate for discrete Minimizing Movements}
\label{subsec:last}
\newcommand{\nomu}{}
\newcommand{\unomu}1
The aim of this section is to prove Theorem \ref{thm:bvestimate},
namely a uniform bound for all discrete Minimizing Movements, under
the stronger $\alpha-\Lambda$ convexity assumption of Section
\ref{subsec:genconvexity}
and a Lipschitz property of the power term.

Let us recall that we are considering the metric setting of Remark \ref{rem:metric}, with
\[
\delta(x,y)=\frac{\unomu}{2}\sfd_*^2(x,y)\qquad\text{for every $x,y\in X$},
\] 
where $\sfd_*$ is another continuous distance on $X$, the energy satisfies assumptions \mytag A{}, the generalized convexity \eqref{eq:genconvexity} and the power term is Lipschitz, according to \eqref{eq:powerlipschitz}. 

To prove Theorem \ref{thm:bvestimate} we combine two basic facts:
the first one is the discrete Gronwall-like lemma of \cite[Lemma 7.5]{Mielke-Rossi-Savare13}.
\begin{lemma}[A discrete Gronwall lemma] \label{lem:gronwall}Let $\gamma>0$ and let $(a_n),(b_n)\subset [0,+\infty)$ be positive sequences, satisfying
\[
(1+\gamma)^2a_n^2\leq a_{n-1}^2+b_n a_n\quad\forall n\ge 1.
\]
Then for all $k\in\N$ there holds
\[
\sum_{n=1}^ka_n\leq \frac{1}{\gamma}\left(a_0+\sum_{n=1}^k b_n\right).
\]
\end{lemma}
The second ingredient is provided by the following estimates of 
the residual stability functional. 
We set $\sfp(x,y):=\sfd(x,y)\sfd_*(x,y)$.
\begin{lemma}
  \label{le:Rprop}
  Let us assume that 
  $(\cE,\sfd,\sfd_*)$ satisfies the \emph{strong}
  $\alpha\text{-}\Lambda$ convexity property of Definition
  \ref{def:aLconvexity}.
  For every $t\in [0,T]$, $x\in X$ and $y\in \rmM(t,x)$ we have
  \begin{equation}
      \label{eq:203bis}
      2\cR(t,x)\ge (\alpha+\unomu)\sfd_*^2(x,y)-\Lambda \sfp(x,y)
    \end{equation}
    If moreover $x\in \rmM(s,v)$ for some $(s,v)\in [0,T]\times X$
    and \eqref{eq:powerlipschitz} hold,
  then
  \begin{equation}
    \label{eq:203tris}
    2\cR(t,x)\le -(\alpha+\unomu)\sfd_*^2(x,y)+\Lambda \sfp(x,y)+
    2\nomu\sfd_*(v,x)\sfd_*(x,y)+2L|t-s|\sfd_*(x,y)
  \end{equation}
  so that
  \begin{equation}
    \label{eq:204}
    (2\alpha+\unomu)\sfd_*^2(x,y)
    \le 2\Lambda \sfp(x,y)
    +\nomu\sfd_*^2(v,x)+2L|t-s|\sfd_*(x,y).
  \end{equation}
\end{lemma}
\begin{proof}
  If $y\in \rmM(t,x)$ and $\gamma$ is a curve
connecting $y$ to $x$ as \eqref{eq:genconvexity},
we get
\begin{align*}
\cE(t,y)&+\sfd(x,y)+\frac{\unomu}{2}\sfd_*^2(x,y)
      \le
\cE(t,\gamma(\theta))+\sfd(x,\gamma(\theta))+\frac{\unomu}{2}\sfd_*^2(x,\gamma(\theta))
  \\&\le
(1-\theta)\cE(t,y)+\theta\cE(t,x)-\frac \alpha2\theta(1-\theta)
      \sfd_*^2(x,y)
  \\&\qquad 
      +\frac \Lambda2\theta(1-\theta)\sfp(x,y)
      +\sfd(x,\gamma(\theta))+\frac{\unomu}{2}\sfd_*^2(x,\gamma(\theta)).
\end{align*} 
We obtain by \eqref{eq:202}
\begin{multline*}
\theta(1-\theta)\frac \alpha2\sfd_*^2(x,y)+
\frac{\unomu}{2}(2\theta-\theta^2)\sfd_*^2(x,y)\le 
\theta\Big(\cE(t,x)-\cE(t,y)-
\sfd(x,y)+\frac\Lambda2(1-\theta)\sfp(x,y)\Big)
\end{multline*}
Dividing by $\theta$ and passing to the limit as $\theta\down0$
\begin{displaymath}
\frac{\alpha+\unomu}2\sfd_*^2(x,y)\leq 
\cE(t,x)-\cE(t,y)-
\sfd(x,y)-\frac \unomu2\sfd_*^2(x,y)+
\frac\Lambda2\sfp(x,y)=
\cR(t,x)+\frac\Lambda2\sfp(x,y)
\end{displaymath}
which yields \eqref{eq:203bis}.

The proof of \eqref{eq:203tris} is similar, but now we start from 
the minimality of $x$ and consider the curve $\gamma$ connecting $x$
to $y$ obtaining
\begin{align*}
\cE(s,x)&+\sfd(v,x)+\frac{\unomu}{2}\sfd_*^2(v,x)
      \le
       \cE(s,\gamma(\theta))+\sfd(v,\gamma(\theta))+\frac{\unomu}{2}\sfd_*^2(v,\gamma(\theta))
  \\&\le
(1-\theta)\cE(s,x)+\theta\cE(s,y)-\frac \alpha2\theta(1-\theta)
      \sfd_*^2(x,y)
  \\&\qquad 
      +\frac \Lambda2\theta(1-\theta)\sfp(x,y)
      +\sfd(v,x)+\sfd(x,\gamma(\theta))+\frac{\unomu}{2}\sfd_*^2(v,\gamma(\theta)).
\end{align*} 
so that 
\begin{align*}
  0&\le \theta\Big(\cE(s,y)-\cE(s,x)-\frac\alpha2(1-\theta)\sfd_*^2(x,y)\Big)
     \\&\qquad+\theta\Big( \frac\Lambda 2(1-\theta)\sfp(x,y)+
      \sfd(x,y)+\frac \unomu2\sfd_*(x,y)\big(\sfd_*(v,x)+\sfd_*(v,\gamma(\theta))\big)\Big). 
\end{align*}
Dividing by $\theta$ and passing to the limit as $\theta\down0$ we get
\begin{displaymath}
  0\le \cE(s,y)-\cE(s,x)-\frac\alpha2\sfd_*^2(x,y)+
  \frac\Lambda 2\sfp(x,y)+
  \sfd(x,y)+\nomu\sfd_*(x,y)\sfd_*(v,x). 
\end{displaymath}
Adding $\cR(t,x)=\cE(t,x)-\cE(t,y)-\frac\unomu2\sfd_*^2(x,y)-\sfd(x,y)$ 
we get
\begin{align*}
  \cR(t,x)&\le -\frac{\alpha+\unomu}2\sfd_*^2(x,y)
  +\frac\Lambda 2\sfp(x,y)+\nomu\sfd_*(x,y)\sfd_*(v,x) 
            \\&\qquad+
  \Big(\cE(t,x)-\cE(s,x)\Big)-\Big(\cE(t,y)-\cE(s,y)\Big),
\end{align*}
and estimating the last term by \eqref{eq:powerlipschitz}
\begin{displaymath}
  \Big(\cE(t,x)-\cE(s,x)\Big)-\Big(\cE(t,y)-\cE(s,y)\Big)
  =\int_s^t \Big(\cP(r,x)-\cP(r,y)\Big)\,\d r
 \le L|t-s|\sfd_*(x,y) 
\end{displaymath}
we obtain \eqref{eq:203tris}.
\eqref{eq:204} follows by combining \eqref{eq:203bis} with
\eqref{eq:203tris} and using the elementary inequality
$2\sfd_*(v,x)\sfd_*(x,y)\le \sfd_*^2(v,x)+\sfd_*^2(x,y)$.
\end{proof}
\begin{proof}[Proof of Theorem \ref{thm:bvestimate}] 
If $U^n_\tau$ is a solution of the time incremental minimization
scheme,
we clearly have $U^{n+1}_\tau\in \rmM(t^{n+1}_\tau,U^n_\tau)$ so that
we
can apply the previous Lemma \ref{le:Rprop} with
$v:=U^{n-1}_\tau$, $x:=U^{n}_\tau$, $y:=U^{n+1}_\tau$ and
$s=t^n_\tau, t=t^{n+1}_\tau$ obtaining
\begin{equation}
  \label{eq:194}
  (2\alpha+1)\sfd_*^2(U^n_\tau,U^{n+1}_\tau)\le 
  \sfd_*^2(U^{n-1}_\tau,U^{n}_\tau)+
  \Big(2\Lambda\sfd(U^{n}_\tau,U^{n+1}_\tau)+
  2L\tau^{n+1}\Big)\sfd_*(U^n_\tau,U^{n+1}_\tau).
\end{equation}
We can apply the discrete Gronwall lemma \ref{lem:gronwall} with:
\[
a_n=\sfd_*(U^n_\tau,U^{n+1}_\tau), \quad b_n:=2\Lambda\sfd(U^{n}_\tau,U^{n+1}_\tau)+
2L\tau^{n+1}\quad \gamma:=2\alpha
\]
obtaining
\begin{displaymath}
  \sum_{n=1}^{N_\tau-1}\sfd_*(U^n_\tau,U^{n+1}_\tau)\le 
  \frac1{\alpha}
  \Big(\frac 12\sfd^2_*(U^0_\tau,U^1_\tau)+1+LT+
  \Lambda \sum_{n=1}^{N_\tau-1}\sfd(U^{n}_\tau,U^{n+1}_\tau)\Big)
  \le \frac{(\Lambda+1) C_2+1+LT}\alpha
\end{displaymath}
where $C_2$ is the constant of \eqref{eq:88}:
this estimate shows that the total variation 
$\Var{\sfd_*}(\overline U_\tau,[0,T])$ is uniformly bounded
w.r.t.~$\tau$,
so that any pointwise limit of $\overline U_\tau$ belongs to 
$\BV{\sfd_*}([0,T];X)$.
\end{proof}

\subsection*{Affiliations}
The second author is
a \emph{research associate}
of the Institute for Applied Mathematics and Information Technologies
"Enrico Magenes"
(IMATI-CNR) of Pavia.

\subsection*{Funding}

The second author has been partially supported by
PRIN10/11 grant from MIUR for the project \emph{Caculus of
  Variations}.

\subsection*{Conflict of interest}
The authors declare that they have no conflict of interest.


\begin{thebibliography}{10}

\bibitem{Agostiniani-Rossi16}
{\sc V.~{Agostiniani} and R.~{Rossi}}, {\em {Singular vanishing-viscosity
  limits of gradient flows: the finite-dimensional case}}, ArXiv
e-prints 1611.08105,
  (2016).

\bibitem{Ambrosio-Gigli-Savare08}
{\sc L.~Ambrosio, N.~Gigli, and G.~Savar{\'e}}, {\em Gradient flows in metric
  spaces and in the space of probability measures}, Lectures in Mathematics ETH
  Z\"urich, Birkh\"auser Verlag, Basel, second~ed., 2008.

\bibitem{Artina-Cagnetti-Fornasier-Solombrino15}
{\sc M.~{Artina}, F.~{Cagnetti}, M.~{Fornasier}, and F.~{Solombrino}}, {\em
  {Linearly constrained evolutions of critical points and an application to
  cohesive fractures}}, ArXiv e-prints 1508.02965,  (2015).

\bibitem{AuMiSt08RIMI}
{\sc F.~Auricchio, A.~Mielke, and U.~Stefanelli}, {\em A rate-independent model
  for the isothermal quasi-static evolution of shape-memory materials}, M$^3$AS
  Math. Models Meth. Appl. Sci., 18 (2008), pp.~125--164.

\bibitem{BoMiRo07?CDPS}
{\sc G.~Bouchitt{\'e}, A.~Mielke, and T.~Roub{\'{\i}}{\v{c}}ek}, {\em A
  complete-damage problem at small strains}, Z. Angew. Math. Phys., 60 (2009),
  pp.~205--236.

\bibitem{Castaing-Valadier77}
{\sc C.~Castaing and M.~Valadier}, {\em Convex analysis and measurable
  multifunctions}, Springer-Verlag, Berlin, 1977.
\newblock Lecture Notes in Mathematics, Vol. 580.

\bibitem{DalMaso93}
{\sc G.~{Dal Maso}}, {\em An Introduction to ${\Gamma}$-Convergence}, vol.~8 of
  Progress in Nonlinear Differential Equations and Their Applications,
  Birkh\"auser, Boston, 1993.

\bibitem{DaDeMo06QEPL}
{\sc G.~{Dal Maso}, A.~DeSimone, and M.~G. Mora}, {\em Quasistatic evolution
  problems for linearly elastic-perfectly plastic materials}, Arch. Ration.
  Mech. Anal., 180 (2006), pp.~237--291.

\bibitem{DaDeMoMo06}
{\sc G.~Dal~Maso, A.~DeSimone, M.~G. Mora, and M.~Morini}, {\em Globally stable
  quasistatic evolution in plasticity with softening}, Netw. Heterog. Media, 3
  (2008), pp.~567--614.

\bibitem{DalMaso-Francfort-Toader05}
{\sc G.~Dal~Maso, G.~A. Francfort, and R.~Toader}, {\em Quasistatic crack
  growth in nonlinear elasticity}, Arch. Ration. Mech. Anal., 176 (2005),
  pp.~165--225.

\bibitem{DalMaso-Toader02}
{\sc G.~Dal~Maso and R.~Toader}, {\em A model for the quasi-static growth of
  brittle fractures based on local minimization}, Math. Models Methods Appl.
  Sci., 12 (2002), pp.~1773--1799.

\bibitem{Efendiev-Mielke06}
{\sc M.~Efendiev and A.~Mielke}, {\em On the rate--independent limit of systems
  with dry friction and small viscosity}, J. Convex Analysis, 13 (2006),
  pp.~151--167.

\bibitem{Francfort-Marigo98}
{\sc G.~Francfort and J.~Marigo}, {\em Revisiting brittle fracture as an energy
  minimization problem}, J. Mech. Phys. Solids, 46 (1998), pp.~1319--1342.

\bibitem{Francfort-Mielke06}
{\sc G.~Francfort and A.~Mielke}, {\em Existence results for a class of
  rate-independent material models with nonconvex elastic energies}, J. Reine
  Angew. Math., 595 (2006), pp.~55--91.

\bibitem{Gal57}
{\sc I.~S. G{\'a}l}, {\em On the fundamental theorems of the calculus}, Trans.
  Amer. Math. Soc., 86 (1957), pp.~309--320.

\bibitem{Knees-Mielke-Zanini08}
{\sc D.~Knees, A.~Mielke, and C.~Zanini}, {\em On the inviscid limit of a model
  for crack propagation}, Math. Models Methods Appl. Sci., 18 (2008),
  pp.~1529--1569.

\bibitem{Knees-Negri15}
{\sc D.~Knees and M.~Negri}, {\em Convergence of alternate minimization schemes
  for phase field fracture and damage}, 
Math. Models Methods Appl. Sci. 27, (2017) pp.~1743--1794.

\bibitem{Knees-Zanini-Mielke10}
{\sc D.~Knees, C.~Zanini, and A.~Mielke}, {\em Crack growth in polyconvex
  materials}, Phys. D, 239 (2010), pp.~1470--1484.

\bibitem{KoMiRo06RIAD}
{\sc M.~Ko\v{c}vara, A.~Mielke, and T.~Roub{\'\i}{\v{c}}ek}, {\em A
  rate--independent approach to the delamination problem}, Math. Mech. Solids,
  11 (2006), pp.~423--447.

\bibitem{Krejci-Liero09}
{\sc P.~Krejc\'\i\ and M.~Liero}, {\em Rate independent {K}urzweil processes},
  Appl. Math., 54 (2009), pp.~117--145.

\bibitem{Kuratowski55}
{\sc K.~Kuratowski}, {\em Sur l'espace des fonctions partielles}, Ann. Mat.
  Pura Appl. (4), 40 (1955), pp.~61--67.

\bibitem{Larsen10}
{\sc C.~J. Larsen}, {\em Epsilon-stable quasi-static brittle fracture
  evolution}, Comm. Pure Appl. Math., 63 (2010), pp.~630--654.

\bibitem{Mainik-Mielke05}
{\sc A.~Mainik and A.~Mielke}, {\em Existence results for energetic models for
  rate-independent systems}, Calc. Var. Partial Differential Equations, 22
  (2005), pp.~73--99.

\bibitem{MaiMie08?GERI}
\leavevmode\vrule height 2pt depth -1.6pt width 23pt, {\em Global existence for
  rate-independent gradient plasticity at finite strain}, J. Nonlinear Science,
  19 (2009), pp.~221--248.

\bibitem{Miel03EFME}
{\sc A.~Mielke}, {\em Energetic formulation of multiplicative
  elasto--plasticity using dissipation distances}, Cont. Mech. Thermodynamics,
  15 (2003), pp.~351--382.

\bibitem{Miel04EMIE}
\leavevmode\vrule height 2pt depth -1.6pt width 23pt, {\em Existence of
  minimizers in incremental elasto--plasticity with finite strains}, SIAM J.
  Math. Analysis, 36 (2004), pp.~384--404.

\bibitem{Mielke11}
{\sc A.~Mielke}, {\em Complete-damage evolution based on energies and
  stresses}, Discrete Contin. Dyn. Syst. Ser. S, 4 (2011), pp.~423--439.

\bibitem{Mielke11-CIME}
\leavevmode\vrule height 2pt depth -1.6pt width 23pt, {\em Differential,
  energetic, and metric formulations for rate-independent processes}, in
  Nonlinear {PDE}'s and applications, vol.~2028 of Lecture Notes in Math.,
  Springer, Heidelberg, 2011, pp.~87--170.

\bibitem{Mielke-Rossi-Savare09}
{\sc A.~Mielke, R.~Rossi, and G.~Savar{\'e}}, {\em Modeling solutions with
  jumps for rate-independent systems on metric spaces}, Discrete and Continuous
  Dynamical Systems A, 25 (2009).

\bibitem{Mielke-Rossi-Savare12}
\leavevmode\vrule height 2pt depth -1.6pt width 23pt, {\em B{V} solutions and
  viscosity approximations of rate-independent systems}, ESAIM Control Optim.
  Calc. Var., 18 (2012), pp.~36--80.

\bibitem{MRS12}
\leavevmode\vrule height 2pt depth -1.6pt width 23pt, {\em Variational
  convergence of gradient flows and rate-independent evolutions in metric
  spaces}, Milan J. Math., 80 (2012), pp.~381--410.

\bibitem{Mielke-Rossi-Savare13}
\leavevmode\vrule height 2pt depth -1.6pt width 23pt, {\em Balanced viscosity
  {(BV)} solutions to infinite-dimensional rate-independent systems}, JEMS, to
  appear. ArXiv 1309.6291,  (2013).

\bibitem{MRS13}
\leavevmode\vrule height 2pt depth -1.6pt width 23pt, {\em Nonsmooth analysis
  of doubly nonlinear evolution equations}, Calc. Var. Partial Differential
  Equations, 46 (2013), pp.~253--310.

\bibitem{MRS16}
\leavevmode\vrule height 2pt depth -1.6pt width 23pt, {\em {Global existence
  results for viscoplasticity at finite strain}}, ArXiv e-prints,  (2016).

\bibitem{MieRou06RIDP}
{\sc A.~Mielke and T.~Roub{\'{\i}}{\v{c}}ek}, {\em Rate-independent damage
  processes in nonlinear elasticity}, M$^3\!$AS Math. Models Methods Appl.
  Sci., 16 (2006), pp.~177--209.

\bibitem{Mielke-Roubicek15}
\leavevmode\vrule height 2pt depth -1.6pt width 23pt, {\em Rate-independent
  systems}, vol.~193 of Applied Mathematical Sciences, Springer, New York,
  2015.
\newblock Theory and application.

\bibitem{Mielke-Theil04}
{\sc A.~Mielke and F.~Theil}, {\em On rate-independent hysteresis models},
  NoDEA Nonlinear Differential Equations Appl., 11 (2004), pp.~151--189.


\bibitem{Mielke-Theil-Levitas02}
{\sc A.~Mielke, F.~Theil, and V.~I. Levitas}, {\em A variational formulation of
  rate-independent phase transformations using an extremum principle}, Arch.
  Ration. Mech. Anal., 162 (2002), pp.~137--177.


\bibitem{MieTim06EMMT}
{\sc A.~Mielke and A.~Timofte}, {\em An energetic material model for
  time-dependent ferroelectric behavior: existence and uniqueness}, Math. Meth.
  Appl. Sciences, 29 (2006), pp.~1393--1410.

\bibitem{Mielke-Zelik14}
{\sc A.~Mielke and S.~Zelik}, {\em On the vanishing-viscosity limit in
  parabolic systems with rate-independent dissipation terms}, Ann. Sc. Norm.
  Super. Pisa Cl. Sci. (5), 13 (2014), pp.~67--135.

\bibitem{Minotti16}
{\sc L.~Minotti}, {\em Visco-energetic solutions to 1-dimensional
  rate-independent problems}, ArXiv e-prints 1610.00507,  (2016).
To appear on Discrete Contin. Dyn.
  Syst. Ser. A.

\bibitem{Minotti16T}
\leavevmode\vrule height 2pt depth -1.6pt width 23pt, {\em Visco-Energetic
  Solutions to Rate-Independent Evolution Problems}, PhD thesis, Pavia, 2016.

\bibitem{Negri16}
{\sc M.~Negri}, {\em An ${L}^2$ gradient flow and its quasi-static limit in
  phase-field fracture by alternate minimization}, 2016.

\bibitem{NegOrt07?QSCP}
{\sc M.~Negri and C.~Ortner}, {\em Quasi-static crack propagation by
  {G}riffith's criterion}, Math. Models Methods Appl. Sci., 18 (2008),
  pp.~1895--1925.

\bibitem{Rindler15}
{\sc F.~{Rindler}}, {\em {A two-speed model for finite-strain
  elasto-plasticity}}, ArXiv e-prints 1512.05928,  (2015).

\bibitem{Rossi-Mielke-Savare08}
{\sc R.~Rossi, A.~Mielke, and G.~Savar\'e}, {\em A metric approach to a class
  of doubly nonlinear evolution equations and applications}, Ann. Sc. Norm.
  Super. Pisa Cl. Sci. (5), 7 (2008), pp.~97--169.

\bibitem{Rossi-Savare13}
{\sc R.~Rossi and G.~Savar{\'e}}, {\em A characterization of energetic and {BV}
  solutions to one-dimensional rate-independent systems}, Discrete Contin. Dyn.
  Syst. Ser. S, 6 (2013), pp.~167--191.

\bibitem{Rossi-Savare17-preprint}
{\sc R.~Rossi and G.~Savar\'e}, {\em {From Visco-Energetic to Energetic and
  Balanced Viscosity solutions of rate-independent systems}}, ArXiv
e-prints 1702.00136,
  (2017).

\bibitem{Roubicek15}
{\sc T.~Roub{\'{\i}}{\v{c}}ek}, {\em Maximally-dissipative local solutions to
  rate-independent systems and application to damage and delamination
  problems}, Nonlinear Anal., 113 (2015), pp.~33--50.

\bibitem{SchMie05VPSC}
{\sc F.~Schmid and A.~Mielke}, {\em Vortex pinning in super-conductivity as a
  rate-independent process}, Europ. J. Appl. Math., 16 (2005), pp.~799--808.

\bibitem{Simons98}
{\sc S.~Simons}, {\em Minimax and monotonicity}, vol.~1693 of Lecture Notes in
  Mathematics, Springer-Verlag, Berlin, 1998.

\bibitem{Stefanelli09}
{\sc U.~Stefanelli}, {\em A variational characterization of rate-independent
  evolution}, Math. Nachr., 282 (2009), pp.~1492--1512.

\end{thebibliography}
\end{document}